\documentclass[12pt,CJK]{article}
\usepackage{geometry}
\usepackage{amsfonts,amsmath,amssymb,amsthm,mathrsfs}
\usepackage{latexsym}
\usepackage{graphicx}
\usepackage{indentfirst}
\usepackage{color}
\usepackage{float} %设置图片浮动位置的宏包
\usepackage{subfigure} %插入多图时用子图显示的宏包
\usepackage{fancyhdr}
\usepackage[colorlinks,linktocpage,linkcolor=blue,citecolor=red]{hyperref}
\usepackage{titlesec}

\usepackage{authblk}

\usepackage{pgf}
\usepackage{tikz}
\usepackage[utf8]{inputenc}
\usetikzlibrary{arrows,automata}
\usetikzlibrary{positioning}

\tikzset{
    state/.style={
           rectangle,
           rounded corners,
           draw=black, very thick,
           minimum height=2em,
           inner sep=2pt,
           text centered,
           },
}

\tikzset{global scale/.style={
    scale=#1,
    every node/.append style={scale=#1}
  }
}

\definecolor{darkgreen}{rgb}{0.00,0.50,0.25}
\definecolor{purple}{rgb}{0.54,0.17,0.89}
\definecolor{purple1}{rgb}{1.00,0.00,1.00}

\titleformat*{\subsection}{\bfseries}

\theoremstyle{plain}                       % default
\newtheorem{lemma}{Lemma}[section]
\newtheorem{theorem}[lemma]{Theorem}
\newtheorem{corollary}[lemma]{Corollary}
\newtheorem{remark}[lemma]{Remark}
\newtheorem{definition}[lemma]{Definition}
\newtheorem{proposition}[lemma]{Proposition}

\theoremstyle{remark}

\numberwithin{equation}{section}

\def\Xint#1{\mathchoice
  {\XXint\displaystyle\textstyle{#1}}%
  {\XXint\textstyle\scriptstyle{#1}}%
  {\XXint\scriptstyle\scriptscriptstyle{#1}}%
  {\XXint\scriptscriptstyle\scriptscriptstyle{#1}}%
  \!\int}
\def\XXint#1#2#3{{\setbox0=\hbox{$#1{#2#3}{\int}$}
  \vcenter{\hbox{$#2#3$}}\kern-.5\wd0}}

\def\dashint{\Xint-}
\DeclareMathOperator*{\esssup}{ess\,sup}
\begin{document}
\allowdisplaybreaks
%\pagestyle{myheadings}
%\markboth{$~$ \hfill {\rm Q. Xu,} \hfill $~$}
%{$~$ \hfill {\rm  } \hfill$~$}
%\author{Li Wang
%%\thanks{Email: lwang10@lzu.edu.cn.}
%\quad Qiang Xu
%\thanks{Corresponding author}
%\thanks{Email: xuqiang09@lzu.edu.cn.}
%\quad Peihao Zhao
%%\thanks{Email: phzhao@lzu.edu.cn.}
%\\
%School of Mathematics and Statistics, Lanzhou University, \\
%Gansu, 730000, PR China.
%\vspace{0.5cm}
%}

\author[a,b]{Li Wang\thanks{Email: wangli@math.pku.edu.cn.}}

\author[b]{Qiang Xu\thanks{Email: xuq@lzu.edu.cn.}
}

%\author[b]{Peihao Zhao\thanks{Email: zhaoph@lzu.edu.cn.}}
%
\affil[a]{School of Mathematic Sciences,
Peking University, Peking 100871, China.
%\authorcr Email:wangli@math.pku.edu.cn %\authorcr is for a new line.
}

\affil[b]{School of Mathematics and Statistics, Lanzhou University, Lanzhou, 730000, China.}

%

%\author{Weiren Zhao
%\thanks{Email: xuqiang@math.pku.edu.cn.}
%\thanks{This work was supported by the National Natural Science Foundation of China (Grant No. 11471147).}

\title{
%\raggedright \normalsize{
%\textcolor[rgb]{1.00,0.00,0.00}{\textbf{NOTE:} \textbf{It's not for spread!}}}\\
%\vspace{0.1cm}
\centering
\Large
\textbf{Calder\'on-Zygmund
estimates for stochastic elliptic
systems on bounded Lipschitz domains} }

\maketitle
\begin{abstract}
Concerned with elliptic operators with stationary random coefficients of
integrable correlations and bounded Lipschitz domains, arising from
stochastic homogenization theory, this paper is mainly devoted to studying
Calder\'on-Zygmund estimates. As an application,
we obtain the homogenization error in the sense of oscillation and
fluctuation, respectively. These results are optimal up to a quantity
$O(\ln(1/\varepsilon))$, which is caused by the
quantified sublinearity of correctors in dimension
two and the less smoothness of the boundary.
In this paper, we find a novel form of \emph{minimal radius},
which is proved to be a suitable tool for quantitative stochastic homogenization on boundary value problems,
when we adopt Gloria-Neukamm-Otto's strategy originally inspired
by the pioneering work of Naddaf and Spencer.

\smallskip
\noindent
\textbf{Key words:}
Homogenization error; Calder\'on-Zygmund estimates;
oscillation; fluctuation; Lipschitz domains.
\end{abstract}

%\tableofcontents

\section{Introduction}

%\paragraph{Data availability statement.}
% Data sharing not applicable to this article as no datasets were generated or analysed during the current study.

\subsection{\centering Motivation and basic ideas}\label{subsection:1.1}

\noindent
Quantitative stochastic homogenization theory has been extensively studied over the past decade, and the crucial progress has been made by overcoming
the missing Poincar\'e estimate on probability space to
quantify the sublinearity of correctors.
Two approaches were developed to achieve this purpose.
The first one, based upon quantified ergodicity in terms
of finite range assumption, was established by S. Armstrong and C. Smart
\cite{Armstrong-Smart16} and culminated with
the monograph \cite{Armstrong-Kuusi-Mourrat19}.
The second one, rooted in concentration inequalities to
quantify ergodicity, was developed by A. Gloria, S. Neukamm,
and F. Otto \cite{Gloria-Otto11,Gloria-Neukamm-Otto15,
Gloria-Neukamm-Otto20,Gloria-Neukamm-Otto21}.
We also refer the readers to \cite{Armstrong-Kuusi22,Josien-Otto22}
and references therein for more details.

To the authors' best
knowledge, the second approach seldom touches the boundary value
problems (except for periodic boundary conditions).
With the help of boundary correctors originally developed by M. Avellaneda
and F. Lin \cite{Avellaneda-Lin87} for quantitative periodic
homogenization theory,
J. Fischer and C. Raithel \cite{Fischer-Raithel17} derived large-scale regularity theory for random elliptic operators on the half-space,
and recently regarding a region with a corner in dimension two,
M. Josien, C. Raithel, and M. Sch\"affner \cite{Josien-Raithel-Schaffner22} obtained
large-scale regularity theory by a non-standard expansion
argument. Instead, without introducing boundary correctors,
the present work is managing to extend Gloria-Neukamm-Otto's methods \cite{Gloria-Neukamm-Otto20,Gloria-Neukamm-Otto21} to study boundary value problems involving non-smooth domains.

In the present paper, we take Gloria-Neukamm-Otto's results \cite{Gloria-Neukamm-Otto20,Gloria-Neukamm-Otto21} on correctors as the ``input'', and systematically investigate the impact of lower regularity of the boundary on the relevant conclusions by virtue of Shen's real arguments \cite{Shen20} in conjunction with the new definition of minimal radius which is suitable for boundary value problems,
originally defined in \cite{Gloria-Neukamm-Otto20} (also similar to
``random scale'' $\mathcal{X}_s$ given in \cite{Armstrong-Kuusi-Mourrat19}). On the other hand,
the study of boundary value problems on non-smooth domains is a classical
branch in PDEs and  has yielded abundant research
results (see e.g. \cite{Kenig94,Hofmann-Mitrea-Taylor10}).
The current literature contributes to establishing a
reasonable link between it and quantitative stochastic homogenization.

\medskip

Precisely, we are interested in a family of elliptic operators in the divergence form\footnote{For the ease of statement, we adopt scalar notation and language throughout the paper.}
\begin{equation*}
\mathcal{L}_\varepsilon :=-\nabla\cdot a^\varepsilon\nabla = -\nabla\cdot a(\cdot/\varepsilon)\nabla,\qquad \varepsilon>0,
\end{equation*}
where $a$ satisfies $\lambda$-uniformly elliptic conditions, i.e., for some
$\lambda\in(0,1)$ there holds
\begin{equation}\label{a:1}
\lambda|\xi|^2 \leq  \xi\cdot a(x)\xi \leq \lambda^{-1}|\xi|^2
\qquad
\text{for~all~}\xi\in\mathbb{R}^{d}\text{~and~}x\in\mathbb{R}^d.
\end{equation}
One can introduce that the configuration space is the set
of coefficient fields satisfying $\eqref{a:1}$, equipped with a
probability measure on it, which is also called as an ensemble,
and we denote the expectation by $\langle\cdot\rangle$. This ensemble
is assumed to be stationary, i.e., for all shift vectors $z\in\mathbb{R}^d$, $a(\cdot+z)$ and $a(\cdot)$ have the same law under
$\langle\cdot\rangle$. In order to establish the quantitative theory, we also
assume that the ensemble $\langle\cdot\rangle$ satisfies the spectral gap
condition: for any random variable, which is a functional of $a$, there exists $\lambda_1\geq 1$ such that
\begin{equation}\label{a:2}
 \big\langle (F-\langle F\rangle)^2 \big\rangle
 \leq \lambda_1
 \Big\langle\int_{\mathbb{R}^d}\Big(\dashint_{B_1(x)}\big|\frac{\partial F}{\partial a}\big|\Big)^2 dx\Big\rangle,
\end{equation}
where
%the symbol ``$\lesssim_{\lambda,\lambda_1,\cdots,\lambda_n}$'' reads ``$\leq C$ for a
%constant $C$ depending only on the tuple $(\lambda,\lambda_1,\cdots,\lambda_n)$ of previously parameters'', and
the definition of the functional derivative of $F$ with respect to $a$ can be found in $\eqref{functional}$ below (or see \cite[pp.15]{Josien-Otto22}).
Moreover, for obtaining pointwise estimates we also assume that
there exists $\sigma_0\in(0,1)$ such that for any $0<\sigma<\sigma_0$
and $1\leq \gamma<\infty$ there holds
\begin{equation}\label{a:3}
 \big\langle  \|a\|_{C^{0,\sigma}(B_1(0))}^\gamma
 \big\rangle \leq \lambda_2(d,\lambda,\sigma,\gamma).
\end{equation}
Throughout the paper, let
$B_r(x)$ denote the open ball centered at $x$ of radius $r$, which
is also abbreviated as $B$. When we use this shorthand notation,
%Also, we use the shorthand notation as follows: for any ball $B\subset\mathbb{R}^d$,
let $x_B$ and $r_B$ represent the center and radius of $B$, respectively.

As mentioned before, this paper is mainly devoted to studying the influence of
nonsmooth domain on the quantitative stochastic homogenization problems,
%before presenting the specific partial differential equations (PDEs),
we introduce the definitions of related domains involved in
the main results of the paper.
\begin{definition}[Lipschitz, $C^1$ domains \cite{Kenig94}]
Let $\Omega\subset\mathbb{R}^d$ with $d\geq 2$ be a bounded domain. $\Omega$
is called a Lipschitz domain, with Lipschitz constant less than or equal to $M_0$
if we can cover a neighborhood of $\partial\Omega$ by
finite many balls $B$ (with $x_B\in\partial\Omega$) so
that, in an appropriate orthonormal coordinate system, $B\cap\Omega=
\{x=(x',x_d): x'\in\mathbb{R}^{d-1}\text{~and~}
x_d>\varphi(x')\}\cap B$, where $\varphi:\mathbb{R}^{d-1}\to\mathbb{R}$ is
a Lipschitz function, i.e., $|\varphi(x')-\varphi(y')|\leq M_0|x'-y'|$ for
any $x',y'\in\mathbb{R}^{d-1}$, with $\varphi(0)=0$. The domain is called a
$C^1$ domain if the function $\varphi$ can be additionally chosen to be of
class $C^1$.
\end{definition}

In classical boundary value problems for PDEs of elliptic type,
the layer potential method was widely used to study the
well-posed of solutions (we will explain later on why
it is related to the current work).
It mainly concerns the $L^q$-boundedness (with $1<q<\infty$) and compactness
of ``singular'' integral operators\footnote{Whether it is a singular
integral operator depends on the smoothness of $\Omega$. If
$\Omega$ is a $C^{1,\alpha}$ domain with $\alpha\in(0,1]$,
$\mathcal{K}$ is not the singular integral operator (since there exists a constant
$C>0$ such that $|n(y)\cdot (y-x)|\leq C|y-x|^{1+\alpha}$
holds for any $x,y\in\partial\Omega$).}, e.g.,
\begin{equation*}
\mathcal{K}(g)(x):=\lim_{\varepsilon\to 0}
\int_{y\in\partial\Omega\atop|x-y|>\varepsilon}
\frac{n(y)\cdot (y-x)}{c_d|x-y|^d}g(y)dS(y),
\quad \forall x\in\partial\Omega \text{~and~}g\in L^q(\partial\Omega),
\end{equation*}
where $n$ is the outward unit normal vector to $\partial\Omega$, and
$dS$ is the surface measure on $\partial\Omega$.
It is worth noting the following two results: (1). Due to the remarkable work of
R. Coifman, A. McIntosh, and Y. Meyer \cite{Coifman-McIntosh-Meyer82},
even for $\Omega$ with a large Lipschitz constant, the singular integral operator
$\mathcal{K}$ is verified to be bounded on $L^q(\partial\Omega)$ for any
$q\in(1,\infty)$, but
the compactness of $\mathcal{K}$ fails on $L^q(\partial\Omega)$ (see e.g. \cite{Fabes-Jodeit-Lewis77}).
(2). E. Fabes, M. Jodeit, and N. Riviere \cite{Fabes-Jodeit-Riviere78}
had shown that $\mathcal{K}$ is  compact on $L^q(\partial\Omega)$
for each $q\in(1,\infty)$, provided that $\Omega$ is a $C^1$ domain.
The former yields a research direction
that is to find a new approach\footnote{It's related to
the so-called Rellich identity which
requires the symmetry condition (i.e., $a=a^*$) on the coefficient
(see\cite{Kenig94,Verchota84}).\label{footnote2}} (independent of compactness) to establish
the invertibility of the trace operator (i.e., $\pm(1/2)I+\mathcal{K}$) on
$L^q(\partial\Omega)$ for some $q$, as well as, the sharp range of $q$
(see \cite{Kenig94,Shen07}).
The latter in turn has spawned an research area where people
focuses on exploring the optimal geometric characterization
of $\Omega$ such that Fredholm
theory based on compact operators can still be
implemented. In this regard, to the authors' best
knowledge, a notable progress was made by
S. Hofmann, M. Mitrea, and M. Taylor \cite{Hofmann-Mitrea-Taylor10}.
Roughly speaking, they found that if $\Omega$ satisfied
some relatively weak conditions, \emph{a sufficient and necessary
condition for $\mathcal{K}$
to be compact was that $\Omega$ was a regular SKT domain}
(since its definition is quite involved and will be given in Appendix).
In this sense,
the regular SKT domain is an optimal generalization of class $C^1$
from our perspective.

We now consider the following Dirichlet (or Neumann) boundary value problems:
\begin{equation}\label{pde:2}
(\text{D}_\varepsilon)~\left\{\begin{aligned}
-\nabla\cdot a^\varepsilon\nabla u_\varepsilon
&= \nabla\cdot f
&\quad&\text{in}~~\Omega;\\
u_\varepsilon &= 0
&\quad&\text{on}~~\partial\Omega,
\end{aligned}\right.
\quad\text{or}\quad
(\text{N}_\varepsilon)~\left\{\begin{aligned}
-\nabla\cdot a^\varepsilon\nabla u_\varepsilon
&= \nabla\cdot f
&\quad&\text{in}~~\Omega;\\
n\cdot a^\varepsilon \nabla u_\varepsilon &= -n\cdot f
&\quad&\text{on}~~\partial\Omega,
\end{aligned}\right.
\end{equation}
Based upon the sublinearity of correctors (see  \cite[Theorem 2]{Papanicolaou-Varadhan79}), it follows from Tartar's test function method
that for $\langle\cdot\rangle$-a.e. (admissible) coefficient $a$, the flux
quantity $a^\varepsilon \nabla u_\varepsilon$ weakly converges to
the effective one denoted by $\bar{a}\nabla\bar{u}$,
where $\bar{a}e_i:=\langle a(\nabla\phi_i+e_i)\rangle$ with $e_i$
being the canonical basis of $\mathbb{R}^d$, and the so-called corrector $\phi_i$ satisfies $\nabla\cdot a(\nabla\phi_i+e_i)=0$ in $\mathbb{R}^d$
(see also Lemma $\ref{lemma:3}$).
Moreover, there holds the following strong convergence:
\begin{equation}\label{pri:00}
 \lim_{\varepsilon\to 0}
 \int_{\Omega}\big\langle|u_\varepsilon-\bar{u}|^2\big\rangle = 0,
 \quad \text{and} \quad
 \lim_{\varepsilon\to 0}
 \int_{\Omega}\big\langle|\nabla u_\varepsilon
 -\big((\nabla\phi_i)(\cdot/\varepsilon)+e_i\big)\partial_i\bar{u}|^2\big\rangle = 0
\end{equation}
(see e.g. \cite[Theorem 3]{Papanicolaou-Varadhan79}, and
the Einstein's summation convention for repeated indices is
used throughout),
where $\partial_i\bar u$ represents partial derivative of $\bar u$
with respect to $e_i$-direction,
and $\bar{u}$ satisfies the effective equation $-\nabla\cdot\bar{a}
\nabla \bar{u}=\nabla\cdot f$ in $\Omega$ with the related Dirichlet (or Neumann) boundary conditions.
We mention that the effective equation is totally deterministic,
which shows the potential merit of random homogenization theory, i.e.,
the possible reduction in numerical complexity.

In fact, the real implementation of the numerical
algorithm requires more quantitative theory
(see e.g. \cite{Clozeau-Josien-Otto-Xu,Gloria-Neukamm-Otto15}),
and one of the interesting topics is to quantify the homogenization errors
presented
in $\eqref{pri:00}$.
In this regard, not only the oscillating property of
the error of two-scale expansions but also the fluctuation
estimate will be derived in this work, where
(weighted) Calder\'on-Zygmund estimates play a basic role.

To understand the aforementioned statement,
let us start from the error of the two-scale expansions:
\begin{equation}\label{expansion:1-00}
w_\varepsilon
:= u_\varepsilon - \bar{u} -\varepsilon \phi_i(\cdot/\varepsilon)
\eta\partial_i\bar{u},
\end{equation}
where $\eta$ is a cut-off function such that
$w_\varepsilon\in H_0^1(\Omega)$, and appealing to the notion of flux corrector
(see Lemma $\ref{lemma:3}$) one can verify that $w_\varepsilon$
satisfies the equation $\eqref{pde:2}$ with a new replacement of $f$ (denoted by
$f_\varepsilon$).
Therefore, some Calder\'on-Zygmund estimates in the form of
\begin{equation}\label{pri:c-z}
\bigg(\int_{\Omega} \langle|\nabla w_\varepsilon|^p\rangle^{\frac{q}{p}}
 \bigg)^{\frac{1}{q}}
 \lesssim_{\lambda,\lambda_1,\lambda_2,M_0,d,p,q}
 \bigg(\int_{\Omega} \langle|f_\varepsilon|^p\rangle^{\frac{q}{p}}
 \bigg)^{\frac{1}{q}}
\end{equation}
with any $p,q\in(1,\infty)$ must be very useful to figure out
homogenization errors in a quantitative way!
(See the definition of the notation ``$\lesssim_{\lambda,\lambda_1,\cdots}$''
in Subsection $\ref{notation}$.)
To be precise,
the error estimate can be reduced to compute
the right-hand side of $\eqref{pri:c-z}$, and it can be further controlled by
\begin{equation}\label{pri:la-colay}
\|(1-\eta)\nabla\bar{u}\|_{L^q(\Omega)}
 + \varepsilon \|\eta\nabla^2\bar{u}\|_{L^q(\Omega)}
\end{equation}
multiplied by a quantity $\mu_d(\cdot/\varepsilon)$ and,
\begin{equation}\label{mu_d}
\mu_d(x):=\left\{\begin{aligned}
&\ln^{\frac{1}{2}}(2+|x|) &\quad&~ d=2;\\
&1 &\quad&~ d>2,
\end{aligned}\right.
\end{equation}
which gauges the sublinear growth of corrector and flux corrector
(see Lemma $\ref{lemma:3}$). Moreover, one can
choose the cut-off function $\eta$ such that
$\text{supp}(\eta)\subset\Omega$ and for any $x\in\partial\Omega$
there holds $\text{dist}(x,\text{supp}(\eta))= O(\varepsilon)$
(where $\text{supp}(\eta)$ represents the support set of $\eta$). In such
the case, we have a chance to estimate $\|\eta\nabla^2\bar{u}\|_{L^q(\Omega)}$ even when $\Omega$ is
a Lipschitz domain, whereas it also reveals that
the quantities in $\eqref{pri:la-colay}$ merely give us
\begin{equation}\label{pri:pri:la-colay-1}
\|(1-\eta)\nabla\bar{u}\|_{L^q(\Omega)}
= O(\varepsilon^{\frac{1}{q}})
\quad\text{and}\quad
\|\eta\nabla^2\bar{u}\|_{L^q(\Omega)}
= O(\varepsilon^{-\frac{1}{q}}),
\end{equation}
at best (since $\Omega$
is assumed to be $C^1$ at most in our discussion). By setting $q=2$, it is exactly the well-known
boundary layer phenomenon appeared
in the quantitative homogenization theory.
%and resulted in the error estimate
%that deviated significantly from the optimal one.
To overcome this difficulty,
we upgrade $\eqref{pri:pri:la-colay-1}$ to the following weighted estimates:
\begin{equation*}
\|(1-\eta)\nabla\bar{u}\delta^{\frac{s}{q}}\|_{L^q(\Omega)}
= O(\varepsilon)
\quad\text{and}\quad
\|\eta\nabla^2\bar{u}\delta^{\frac{s}{q}}\|_{L^q(\Omega)}
= O(\ln(1/\varepsilon)),
\end{equation*}
where $s=q-1$, and the weight is given by
the distance function $\delta(x):=\text{dist}(x,\partial\Omega)$ with $x\in\Omega$.
This, in turn, naturally suggested that
the Calder\'on-Zygmund estimate $\eqref{pri:c-z}$ should be a
weighted version, i.e.,
\begin{equation*}
\bigg(\int_{\Omega} \langle|\nabla w_\varepsilon|^p\rangle^{\frac{q}{p}}
 \delta^{q-1}\bigg)^{\frac{1}{q}}
 \lesssim
 \bigg(\int_{\Omega} \langle|f_\varepsilon|^p\rangle^{\frac{q}{p}}
 \delta^{q-1}\bigg)^{\frac{1}{q}}.
\end{equation*}
Hence, from the point of theoretical research,
it is reasonable to generalize the above estimate from the special weight to the class of
$A_q$ weights, which is the main focus of this paper.
\begin{definition}[$A_q$-weights]\label{def:weight}
Let $1\leq q<\infty$. A non-negative, locally integrable function $\omega$
is said to be an $A_q$ weight, if there exists a positive constant
$A$ such that, for every ball $B\subset\mathbb{R}^d$,
\begin{equation}\label{f:w-1} \bigg(\dashint_{B}\omega\bigg)\bigg(\dashint_{B}\omega^{-1/(p-1)}\bigg)^{p-1}
  \leq A,
\end{equation}
if $q>1$, where
the average integral symbol $\dashint_{B}$ is defined by $\frac{1}{|B|}\int_B$
with $|B|$ being the Lebesgue measure of $B$,  or
\begin{equation}\label{f:w-2}
 \bigg(\dashint_{B}\omega\bigg)\esssup_{x\in B}\frac{1}{\omega(x)}\leq A,
\end{equation}
if $q=1$. The infimum over all such constants $A$ is called
the $A_q$ Muckenhoupt characteristic constant of $\omega$, denoted by
$[\omega]_{A_q}$. Moreover, $A_q$ represents
the set of all $A_q$ weights.
\end{definition}

So far, we have clarified the motivation for studying (weighted)
Calder\'on-Zygmund estimates.
We now concisely elucidate why the layer potential theory mentioned earlier
is crucial in the present work.
In fact, it is reconditely connected to the estimate
$\eqref{pri:pri:la-colay-1}$. Let us do a decomposition for
$\nabla\bar{u}$ as follows:
\begin{equation}\label{f:decom-1}
\nabla\bar{u} = \underbrace{\nabla(-\nabla\cdot\bar{a}\nabla)^{-1}\nabla\cdot
 \tilde{f}_{0}}_{N_f}
 + \nabla v \quad\text{in}~\Omega,
\end{equation}
where $\tilde{f}_0$ is a suitable extension of $f$
such that $\text{supp}(\tilde{f}_0)\supsetneq\Omega$ is a bounded
domain in $\mathbb{R}^d$. Therefore, the rationality
of this decomposition relies only on the solvability
of the following boundary value problem:
\begin{equation}\label{pde:DN}
(\text{D})\left\{\begin{aligned}
\nabla\cdot\bar{a}
\nabla v &= 0 &\quad&\text{in}~\Omega;\\
 v &= -N_f
 &\quad& \text{on}~\partial\Omega,
\end{aligned}\right.
\quad \text{or} \quad
(\text{N})\left\{\begin{aligned}
-\nabla\cdot \bar{a}\nabla v
&= 0
&\quad&\text{in}~~\Omega;\\
n\cdot \bar{a} \nabla v &= -n\cdot \bar{a}\nabla N_f
&\quad&\text{on}~~\partial\Omega.
\end{aligned}\right.
\end{equation}
In view of the layer potential methods mentioned before,
the solvability of (D) or (N) in $\eqref{pde:DN}$ is
rooted in the invertibility of the trace operator\footnote{
It's formulated by
$\pm(1/2)I+\mathcal{K}_{\bar{a}}$, where
$\mathcal{K}_{\bar{a}}(g)(x):=\text{p.v.}\int_{\partial\Omega}
n\cdot\bar{a}\nabla\Gamma_{\bar{a}}(x-\cdot)gdS$, and
$\Gamma_{\bar{a}}$ is the fundamental solution of
$-\nabla\cdot\bar{a}\nabla$.\label{footnote3}} on
$L^q(\partial\Omega)$ with $q\in(1,\infty)$, which ultimately relies on
the regularity of $\partial\Omega$.

Then, we turn to make a heuristic statement of
how to get some estimates comparable to $\eqref{pri:c-z}$.
A primary idea is to
transfer the regularity of the effective solution
back to the original one, at least, in the sense of large scales.
Thus, to see $\eqref{pri:c-z}$,
it is reasonable to start from investigating
if the following estimate\footnote{
To the authors's best knowledge,
the estimate $\eqref{pri:c-z*}$ is unclear for general Lipschitz domains.
However, the method developed in the present work
is independent of this result.}
\begin{equation}\label{pri:c-z*}
\bigg(\int_{\Omega} \langle|(\nabla(-\nabla\cdot\bar{a}\nabla)^{-1}_{\Omega}
 \nabla\cdot)(\mathbf{f})|^p\rangle^{\frac{q}{p}}
 \bigg)^{\frac{1}{q}}
 \lesssim_{\lambda,M_0,d,p,q}
 \bigg(\int_{\Omega} \langle|\mathbf{f}|^p\rangle^{\frac{q}{p}}
 \bigg)^{\frac{1}{q}},
\end{equation}
holds true for any $L^q$-integrable Bochner function $\mathbf{f}$ (with
values in $L^p_{\langle\cdot\rangle}$ space\footnote{
That's the space of
$L^p$-integrable functionals with respect to the measure introduced by
the ensemble $\langle\cdot\rangle$.}),
where $(-\nabla\cdot\bar{a}\nabla)^{-1}_{\Omega}$ is
determined by the Green (or Neumman)
function associated with $\Omega$.
The next question is how to ``deliver'' this regularity?
The answer is to hand it over to the ``local'' approximating.
For any ball $B\subset\mathbb{R}^d$ with
its center $x_B\in\partial\Omega$ or $2B\subset\Omega$, we do
the following decomposition associated with $B$,
inspired by the two-scale expansion
$\eqref{expansion:1-00}$,
\begin{equation}\label{f:decom-2}
\nabla u_\varepsilon
 = \underbrace{\nabla u_\varepsilon
 -\big(e_i+(\nabla\phi_i)(\cdot/\varepsilon)\big)\partial_i
 \bar{u}_B}_{W_\varepsilon}
 + \underbrace{(e_i+(\nabla\phi_i)(\cdot/\varepsilon))
 \partial_i\bar{u}_B}_{V},
\end{equation}
where $\bar{u}_B$ is determined by the equation
$-\nabla\cdot\bar{a}\nabla\bar{u}_B
=\nabla\cdot f$ in
$2B\cap\Omega$ with $\bar{u}_B=u_\varepsilon$ on $\partial(2B\cap\Omega)$.
As the approximating part, $W_\varepsilon$
is required to be sufficiently small in a certain norm.
Thereupon,
the so-called minimum radius $\chi_{*}$ defined in
$\eqref{min_rad}$ becomes relevant (it's too
lengthy to be represented here), and it exactly characterizes
the scale at which the approximating term
can be as desired small.
So far, we have noticed the approximation in $\eqref{f:decom-2}$
is not available for all scales. Roughly speaking, that
eventually contributes to our results on
Calder\'on-Zygmund estimates to be the
form of
\begin{equation*}
\int_{\Omega}
 \Big\langle\big(\dashint_{B\cap\Omega}|\nabla u_\varepsilon|^2\big)^{\frac{p}{2}}
 \Big\rangle^{\frac{q}{p}}
 \lesssim_{\lambda,\lambda_1,\lambda_2,d,p,q,M_0}
 \int_{\Omega}\Big\langle\big(\dashint_{B\cap\Omega}|f|^2\big)^{\frac{p}{2}}
 \Big\rangle^{\frac{q}{p}},
\end{equation*}
instead of that given in $\eqref{pri:c-z}$.
The details of how to use the decomposition $\eqref{f:decom-2}$ and Shen's
real method (i.e., Lemma $\ref{shen's lemma2}$) to give
the desired estimates will be postponed to
Subsection $\ref{subsection:3.0}$.

\subsection{Main results and relation to previous works}

For the reader's convenience, we first introduce some geometric notation
on integral regions. For any $x\in\Omega$, we set
$U_{*,\varepsilon}(x):=\Omega\cap B_{*,\varepsilon}(x)$
with $B_{*,\varepsilon}(x):=B_{\varepsilon\chi_{*}(x/\varepsilon)}(x)$
and, $U_\varepsilon(x):=B_\varepsilon(x)\cap\Omega$, where
$\chi_{*}$ is known as the
minimum radius (see Section $\ref{section:2}$). Let $R_0\geq 1$ be the diameter of $\Omega$.

\begin{theorem}[Calder\'on-Zygmund estimates]\label{thm:1}
Let $\Omega\subset\mathbb{R}^d$ with $d\geq 2$ be a bounded Lipschitz domain and $\varepsilon\in(0,1]$.
Suppose that $\langle\cdot\rangle$ is stationary and satisfies spectral gap condition $\eqref{a:2}$,
and the (admissible) coefficient additionally satisfies
$\eqref{a:3}$ with the symmetry condition $a=a^*$.
Let $u_\varepsilon$ be associated with $f$ by the elliptic systems
$\eqref{pde:2}$.
Then, there exists a stationary random field $\chi_*$
with a Lipschitz continuity, satisfying
$\langle\chi_*^\beta\rangle\lesssim_{\lambda,\lambda_1,\lambda_2,d,\beta} 1$
for any $\beta\in[1,\infty)$, such that the quenched Calder\'on-Zygmund estimate
\begin{equation}\label{A}
\bigg(\int_{\Omega}\Big(\dashint_{U_{*,\varepsilon}(x)}
|\nabla u_\varepsilon|^2 \Big)^{\frac{p}{2}}dx\bigg)^{\frac{1}{p}}
\lesssim_{\lambda,d,M_0,p}
\bigg(\int_{\Omega}\Big(\dashint_{U_{*,\varepsilon}(x)}|f|^2
\Big)^{\frac{p}{2}}dx\bigg)^{\frac{1}{p}}
\end{equation}
holds for any $p>1$ satisfying $|\frac{1}{p}-\frac{1}{2}|\leq
\frac{1}{2d}+\theta$ with
$0<\theta\ll 1$ depending only on $d$ and $M_0$.
Also, for any
$|\frac{1}{q}-\frac{1}{2}|\leq \frac{1}{2d}+\theta$ and $\bar{p}>p$, we have the annealed Calder\'on-Zygmund estimate
\begin{equation}\label{B}
\bigg(\int_{\Omega}
\Big\langle\Big(\dashint_{U_{\varepsilon}(x)}|\nabla u_\varepsilon|^2
\Big)^{\frac{p}{2}}
\Big\rangle^{\frac{q}{p}}dx\bigg)^{\frac{1}{q}}
\lesssim_{\lambda,d,M_0,p,\bar{p},q}
\bigg(
\int_{\Omega}
\Big\langle\Big(\dashint_{U_{\varepsilon}(x)}
|f|^{2}\Big)^{\frac{\bar{p}}{2}}\Big\rangle^{\frac{q}{\bar{p}}}dx
\bigg)^{\frac{1}{q}},
\end{equation}
Particularly, if the equations $\eqref{pde:2}$ are scalar,
the estimates $\eqref{A}$ and $\eqref{B}$ then hold for
$(4/3)-\theta<p,q<4+\theta$ if $d=2$; and for $(3/2)-\theta<p,q<3+\theta$
if $d\geq 3$.
Moreover, if $\Omega$ is a regular SKT (or $C^1$)
domain
(in such the case the symmetry condition $a=a^*$
is not needed\footnote{
As mentioned in
Subsection $\ref{subsection:1.1}$,
if $\Omega$ is a $C^1$ (or regular SKT) domain,
the singular integral
$\mathcal{K}_{\bar{a}}$ (see footnote [\ref{footnote3}])
is compact on $L^q(\partial\Omega)$,
and therefore we can obtain the
solvability of the boundary value problems in
$\eqref{pde:DN}$ for any $q\in(1,\infty)$,
which is different from the argument based upon
the Rellich identity (see footnote [\ref{footnote2}]).}), then the estimates $\eqref{A}$ and $\eqref{B}$
hold for any $1<p,q<\infty$ with $\bar{p}>p$;
For any $\omega\in A_q$,
we also obtain weighted annealed Calder\'on-Zygmund estimate
\begin{equation}\label{C}
\bigg(\int_{\Omega}
\Big\langle\Big(\dashint_{U_{*,\varepsilon}(x)}|\nabla u_\varepsilon|^2 \Big)^{\frac{p}{2}}
\Big\rangle^{\frac{q}{p}}\omega(x)dx\bigg)^{\frac{1}{q}}
\lesssim_{\lambda,d,M_0,p,q,[\omega]_{A_q}}
\bigg(\int_{\Omega}
\Big\langle\Big(\dashint_{U_{*,\varepsilon}(x)}
|f|^2\Big)^{\frac{p}{2}}\Big\rangle^{\frac{q}{p}}\omega(x)dx\bigg)^{\frac{1}{q}}.
\end{equation}
In particular, there exists $\omega_\sigma\in A_q$ such that
\begin{equation}\label{D}
\bigg(\int_{\Omega}\Big\langle
\Big(\dashint_{U_\varepsilon(x)}|\nabla u_\varepsilon|^2 \Big)^{\frac{p}{2}}
\Big\rangle^{\frac{q}{p}}\omega_\sigma(x)dx\bigg)^{\frac{1}{q}}
\lesssim_{\lambda,d,M_0,p,\bar{p},q} \bigg(\int_{\Omega}
\Big\langle\Big(\dashint_{U_{\varepsilon}(x)}|f|^2 \Big)^{\frac{\bar{p}}{2}}
\Big\rangle^{\frac{q}{\bar{p}}}
\omega_\sigma(x)dx\bigg)^{\frac{1}{q}}.
\end{equation}
%where the multiplicative constant never depends on $\varepsilon$ and %$R_0$.
\end{theorem}

\begin{remark}
\emph{The assumptions $\eqref{a:2}$ and $\eqref{a:3}$
 are not crucial to Theorem $\ref{thm:1}$,
 but rather serve to derive a fluctuation estimate
 of the corrector and flux corrector (see Lemma $\ref{lemma:3}$),
 which ensures the existence of our newly defined
 minimum radius $\chi_*$ (distinct from its counterpart $r_*$ in
 the literature \cite{Gloria-Neukamm-Otto20}).
 In other words, this theorem applies to a broader
 range of random coefficients than presently assumed
 (including the models discussed in \cite{Gloria-Neukamm-Otto20,Gloria-Neukamm-Otto21},
 but also for different mixing condition relying on the results
 by S. Armstrong et. al \cite{Armstrong-Kuusi-Mourrat19,Armstrong-Kuusi22}),
 provided that the quantified sublinearity
 of the corrector (and flux corrector) has been established in
 the corresponding setting.} 
\end{remark}

In the case of Dirichlet boundary conditions for $\partial\Omega\in C^{1,\alpha}$ with $\alpha\in(0,1]$,
the result analogy to $\eqref{A}$ was established by
S. Armstrong, T. Kuusi, and J.-C. Mourrat \cite{Armstrong-Kuusi-Mourrat19}
under a finite range assumption, and the estimate similar to $\eqref{B}$ was recently stated by M. Duerinckx \cite{Duerinckx22}.
If $\Omega=\mathbb{R}^d$, for some weight
$\omega\in A_q$ the result analogy to $\eqref{C}$ was given by
A. Gloria, S. Neukamm, and F. Otto in their early version of \cite{Gloria-Neukamm-Otto20}.
In terms of Reifenberg flat domains (see Definition $\ref{def:Rfd}$) ,
the Calder\'on-Zygmund estimates were well studied by
S. Byun, L. Wang \cite{Byun-Wang04} for deterministic elliptic operators
with small $\text{BMO}$ coefficients, originally motivated by
L. Caffarelli and I. Peral's work \cite{Caffarelli-Peral98}.
Concerning general Lipschitz domains, Z. Shen \cite{Shen05}
developed a real method to simplify the related estimates to a reverse H\"older inequality, which was also the main philosophy behind Theorem $\ref{thm:1}$.
The range of $p$ in $\eqref{A}$ is optimal in the case of scalar (see \cite{Jerison-Kenig95}) and systems for $d=2,3$ (see \cite{Geng-Shen-Song12,Shen05} for periodic homogenization).
The estimate $\eqref{C}$ is new even for smooth domains, which is also sharp in the sense of Muckenhoupt’s weight class. When the large-scale average in \eqref{C} is changed into the small-scale one,
the weighted estimate \eqref{D} is not valid for all
of the related weight functions.

Although the small-scale regularity on coefficients is assumed,
the estimates $\eqref{B}$ can not be easily
upgraded into $\eqref{pri:c-z}$.
The possible way is to employ the robust non-perturbative argument
developed in \cite{Gloria-Neukamm-Otto20, Josien-Otto22}, and we will address it
in a separate work. Nevertheless, the assumption $\eqref{a:3}$ which
leads to pointwise estimates of correctors (see Lemma $\ref{lemma:3}$) still simplifies the demonstrations on the higher order moment estimates of
the newly introduced minimal radius $\chi_{*}$.
We also mention that the ensemble $\langle\cdot\rangle$ which satisfies assumptions $\eqref{a:1}$,
$\eqref{a:2}$ and $\eqref{a:3}$ can be found in the literature \cite[Section 3.2]{Josien-Otto22}, where the stationary centered Gaussian fields with an integrable covariance function were introduced.
In addition, the weighted
annealed estimate similar to $\eqref{D}$ can be applied to the fluctuation
estimates of second-order correctors (see \cite[Proposition 3]{Clozeau-Josien-Otto-Xu}),
and the weighted quenched Calder\'on-Zygmund estimate was used to study
the optimal error estimates for periodic homogenization problems on perforated domains (see \cite{Wang-Xu-Zhao21}).
Finally, since we avoid
imposing boundary correctors in the boundary estimates, our methods are  completely applicable to different types of boundary conditions.

\medskip

As an application of Theorem \ref{thm:1}, we therefore have the following results.

\begin{theorem}[homogenization errors]\label{thm:2}
Let $\Omega\subset\mathbb{R}^d$ with $d\geq 2$ be a bounded Lipschitz domain and $0<\varepsilon\ll 1$.
Suppose that $\langle\cdot\rangle$ is stationary and satisfies
the spectral gap condition $\eqref{a:2}$,
and the (admissible) coefficient additionally satisfies
$\eqref{a:3}$ with the symmetry condition $a=a^*$.
Let $u_\varepsilon,u_0\in H_0^1(\Omega)$ be associated with $f\in C_0^1(\Omega;\mathbb{R}^{d})$ by
\begin{equation}\label{pde:0}
(\emph{D}_\varepsilon)~\left\{\begin{aligned}
-\nabla\cdot a^{\varepsilon}\nabla u_\varepsilon
&= \nabla\cdot f &\quad&\text{in}\quad\Omega;\\
u_\varepsilon
&= 0 &\quad&\text{on}\quad\partial\Omega,
\end{aligned}\right.
\qquad\text{and}
\qquad
(\emph{D}_0)~\left\{\begin{aligned}
-\nabla\cdot \bar{a}\nabla \bar{u}
&= \nabla\cdot f &\quad&\text{in}\quad\Omega;\\
\bar{u}
&= 0 &\quad&\text{on}\quad\partial\Omega.
\end{aligned}\right.
\end{equation}
Then, for any $p<\infty$, we have
the strong norm estimate
\begin{equation}\label{error:1}
 \Big\langle \big(\int_{\Omega}|u_\varepsilon-\bar{u}|^2
 \big)^{\frac{p}{2}}\Big\rangle^{\frac{1}{p}}
 \lesssim_{\lambda,\lambda_1,
\lambda_2,d,M_0,p} \mu_d(R_0/\varepsilon)\varepsilon\ln(R_0/\varepsilon)
 \Big(\int_{\Omega}|\nabla\cdot f|^2\Big)^{\frac{1}{2}},
\end{equation}
where the definition of $\mu_d$ is given in $\eqref{mu_d}$.
Moreover, suppose that $\Omega$ is a bounded regular SKT (or $C^1$) domain
(in such the case, the symmetry condition $a=a^*$ is not required).
For any $h\in C_0^\infty(\Omega;\mathbb{R}^{d})$ and
some $\varphi_i\in H_0^1(\Omega)$ with $i=1,\cdots,d$, define the random variable as
\begin{equation}\label{eq:5.2}
H^\varepsilon : = \int_{\Omega}
h\cdot (a^\varepsilon-\bar{a})\big(\nabla u_\varepsilon-\nabla\bar{u}
-\nabla\phi^\varepsilon_i\varphi_i\big).
\end{equation}
Then, there holds the following
weak norm estimate
\begin{equation}\label{error:2}
\begin{aligned}
\varepsilon^{-\frac{d}{2}} \big\langle (H^\varepsilon -\langle H^\varepsilon\rangle)^{2p} \big\rangle^{\frac{1}{2p}}
\lesssim_{\lambda,\lambda_1,
\lambda_2,d,M_0,p,s}
\mu_d(R_0/\varepsilon)\varepsilon\ln^{\frac{1}{2s'}}(R_0/\varepsilon)
\Big(\int_{\Omega} |\nabla h|^{2s}\Big)^{\frac{1}{2s}}
\Big(\int_{\Omega}
 |R_0\nabla f|^{2s'}\Big)^{\frac{1}{2s'}}
\end{aligned}
\end{equation}
for all $p<\infty$,
where $s,s'>1$ are associated with $1/s'+1/s=1$.
\end{theorem}

In the case of periodic homogenization,
the non-smoothness of the boundary leads to the loss of
$O(\ln(R_0/\varepsilon))$ in the convergence rate, first observed by
C. Kenig, F. Lin, and Z. Shen \cite{Kenig-Lin-Shen12}.
However, their method relies on Rellich-type estimates
established for $\mathcal{L}_\varepsilon$, so the H\"older continuity
of coefficients at small scales is inevitably required at least.
Inspired by T. Suslina's work \cite{Suslina13},
the second author of this paper removed the smoothness assumption of coefficients to develop a result similar to the estimate $\eqref{error:1}$ in \cite{Xu16}.
In the random setting, M. Josien, C. Raithel, and M. Sch\"affner \cite{Josien-Raithel-Schaffner22} recently discovered the same phenomenon in a much stranger norm for a region with a corner, while in the present paper, the estimate $\eqref{error:1}$ is established for the more general class of non-smooth domains by a different argument.

As we have mentioned in the Abstract, the other source of the loss of $O(\ln(R_0/\varepsilon))$ in $\eqref{error:1}$ and $\eqref{error:2}$ is also related to the quantified sublinearity of the corrector in dimension two (see Lemma $\ref{lemma:3}$), which originally recognized by A. Gloria and F. Otto \cite{Gloria-Otto11} for a discrete elliptic differential operator with independently and identically distributed random coefficients. In the case of $\partial\Omega\in C^{1,1}$ or convex domains, the result analogy to $\eqref{error:1}$
with better random integrability was established by S. Armstrong, T. Kuusi, and J.-C. Mourrat \cite{Armstrong-Kuusi-Mourrat19} under a finite range assumption. Regarding stationary random coefficients with slowly
decaying correlations,
A. Gloria, S. Neukamm, and F. Otto \cite{Gloria-Neukamm-Otto21} obtained a sharp homogenization error
in the sense of quenched norm on $\mathbb{R}^d$.
In terms of the weak norm estimate,
the result analogy to $\eqref{error:2}$ was shown in \cite{Josien-Otto22} without considering boundary layers, and we additionally employ the weighted annealed Calder\'on-Zygmund estimate $\eqref{D}$ to get
the almost sharp result $\eqref{error:2}$.

% 把式子放到定理1 的证明当中，这一段可以删除。

\subsection{\centering Organization of the paper}

\noindent
In section 2, we show the H-convergence in the sense of a local way by introducing a new minimal radius $\chi_*$ in Lemma $\ref{lemma:1}$. Also, a Lipschitz continuity of $\chi_*$ with a local boundedness property
is shown in Remarks $\ref{remark:3}$, $\ref{remark:1}$. For the readers' convenience, we recall notations of the extended correctors, smoothing operator, Shen’s lemma of weighted version, and
some geometric properties of integrals.

\medskip

\noindent
In section $\ref{section:3}$, we prove Theorem $\ref{thm:1}$.
First, we show how to utilize Shen's real argument (i.e., Lemma $\ref{shen's lemma2}$)
to derive the (weighted) quenched Calder\'on-Zygmund estimate in
Subsection $\ref{subsection:3.0}$.
Then, the quenched estimates
and the annealed ones are discussed in Propositions $\ref{P:10}$ and $\ref{P:7}$, respectively. With the help of Shen’s real argument, our approach is still from quenched estimates to annealed ones.
To overcome the difficulties caused by the non-smoothness of the boundary,
Proposition $\ref{P:10}$ relies on interior quenched $W^{1,p}$ estimates of the correctors, and a reverse H\"older's inequality (see Lemmas $\ref{lemma:7}$ and $\ref{lemma:14}$).
As for Proposition $\ref{P:7}$, we first establish the weighted estimates for the regular SKT (or $C^1$) domains and then pass from the large-scale averages to small-scale ones via Lemma $\ref{lemma:12}$.

\medskip

\noindent
In Section $\ref{section:4}$, we provide a proof for
Theorem $\ref{thm:2}$. The homogenization error in the
strong norm and the weak norm is discussed in Propositions \ref{P:8}
and \ref{P:9}, respectively.
Proposition \ref{P:8} requires the idea of duality and weight functions to overcome the loss of the convergence rate, caused by boundary layers.
Proposition \ref{P:9} is an application of the weighted annealed Calder\'{o}n-Zygmund estimates in Theorem \ref{thm:1}. Besides, to coincide with the framework of sensitive estimates, a rescaling argument plays a key role in the whole proof.

\medskip

\noindent
Section $\ref{section:5}$ is an appendix of the paper,
which contains an improved Meyer's estimate and a
reverse H\"older’s inequality in terms of regular SKT domains.
Also, some properties of
$A_q$ class and the definition of the regular SKT domains
are given for the reader's convenience.

\subsection{\centering Notations}\label{notation}

\begin{enumerate}
  \item Notation for estimates.
\begin{enumerate}
  \item $\lesssim$ and $\gtrsim$ stand for $\leq$ and $\geq$
  up to a multiplicative constant,
  which may depend on some given parameters in the paper,
  but never on $\varepsilon$. The subscript form $\lesssim_{\lambda,\cdots,\lambda_n}$ means that the constant depends only on parameters $\lambda,\cdots,\lambda_n$.
  In addition, we also use superscripts like $\lesssim^{(2.1)}$ to indicate the formula or estimate referenced.
  We write $\sim$ when both $\lesssim$ and $\gtrsim$ hold.
  \item We use $\gg$ instead of $\gtrsim$ to indicate that the multiplicative constant is much larger than 1 (but still finite),
      and it's similarly for $\ll$.
\end{enumerate}
  \item Notation for derivatives.
  \begin{enumerate}
    \item Spatial derivatives:
  $\nabla v = (\partial_1 v, \cdots, \partial_d v)$ is the gradient of $v$, where
  $\partial_i v = \partial v /\partial x_i$ denotes the
  $i^{\text{th}}$ derivative of $v$.
  $\nabla^2 v$  denotes the Hessian matrix of $v$;
  $\nabla\cdot v=\sum_{i=1}^d \partial_i v_i$
  denotes the divergence of $v$, where
  $v = (v_1,\cdots,v_d)$ is a vector-valued function.
  Let $\nabla_{\text{tan}}$ be the tangential derivative with respect to $\partial\Omega$, and we
refer the reader to \cite[pp.213-214]{Shen18} for the concrete definition.
   \item Functional (or vertical) derivative:  the random tensor field $\frac{\partial F}{\partial a}$ (depending on $(a,x)$) is the functional derivative of $F$ with respect to $a$, defined by
       \begin{equation}\label{functional}
       \lim_{\varepsilon\to 0} \frac{F(a+\varepsilon\delta a)
       -F(a)}{\varepsilon}
       =\int_{\mathbb{R}^d} \frac{\partial F(a)}{\partial a_{ij}(x)}
       (\delta a)_{ij}(x)dx.
       \end{equation}
  \end{enumerate}

  \item Geometric notation.
  \begin{enumerate}
  \item $d\geq 2$ is the dimension, and $R_0$ represents
  the diameter of $\Omega$, and $M_0$ is the Lipschitz constant of $\Omega$.
  \item For a ball $B$, we set $B=B_{r_B}(x_B)$ and abusively write
  $\alpha B:=B_{\alpha r_B}(x_B)$.
\item For any  $x\in\partial\Omega$,
we introduce $D_r(x):= B_r(x)\cap\Omega$ and $\Delta_r(x):=
B_r(x)\cap\partial\Omega$.
We will omit the center point of
$B_r(x),D_r(x),\Delta_r(x)$ only when $x=0$.
For any $x\in\Omega$, we introduce $U_r(x):=B_r(x)\cap\Omega$,
and $U_{*,\varepsilon}(x):=
\Omega\cap B_{*,\varepsilon}(x)$ with
$B_{*,\varepsilon}(x):=B_{\varepsilon\chi_{*}(x/\varepsilon)}(x)$,
where $\chi_*$ is the minimum radius given
in Lemma $\ref{lemma:1}$ or Corollary $\ref{cor:1}$.
\end{enumerate}

\item Notation for functions.
\begin{enumerate}
    \item The function $I_{E}$ is the indicator function of region $E$.
\item We denote $F(\cdot/\varepsilon)$ by $F^{\varepsilon}(\cdot)$ for simplicity, and $(f)_r = \dashint_{B_r} f  = \frac{1}{|B_r|}\int_{B_r} f(x) dx$.
\item We denote the support of $f$ by $\text{supp}(f)$.
\item The nontangential maximal function of $u$ is defined by
\begin{equation*}
(u)^{*}(Q):=\sup\big\{|u(x)|:x\in\Gamma_{\beta}(Q)\big\}
\qquad\forall Q\in\partial\Omega,
\end{equation*}
where $\Gamma_\beta(Q):=\{x\in\Omega:|x-x_0|\leq \beta
\text{dist}(x,\partial\Omega)\}$ is the cone with vertex $Q$
and aperture $\beta$, and $\beta>1$ may be sufficiently large.
  \end{enumerate}
\end{enumerate}

Finally, we mention that:
(1)
when we say that the multiplicative constant depends on the character of the domain,
it means that the constant relies on $M_0$;
(2) the index on random integrability is usually nonnegative by default unless otherwise noted.

\section{Preliminaries}\label{section:2}

\begin{lemma}[extended correctors \cite{Gloria-Neukamm-Otto20}]\label{lemma:3}
Let $d\geq 1$. Assume that $\langle\cdot\rangle$ satisfies stationary and ergodic\footnote{In terms of ``ergodic'', we refer the reader to \cite[pp.105]{Gloria-Neukamm-Otto20} or \cite[pp.222-225]{Jikov-Kozlov-Oleinik94} for the definition.} conditions.
Then, there exist two random  fields $\{\phi_i\}_{i=1,\cdots,d}$
and $\{\sigma_{ijk}\}_{i,j,k=1,\cdots,d}$ with the following
properties: the gradient fields $\nabla\phi_i$ and
$\nabla\sigma_{ijk}$ are stationary in the sense of
$\nabla\phi_{i}(a;x+z)=\nabla\phi_i(a(\cdot+z);x)$ for any shift
vector $z\in\mathbb{R}^d$ (and likewise for $\nabla\sigma_{ijk}$).
Moreover, $(\phi,\sigma)$, which are called extended correctors, satisfy
\begin{equation}\label{eq:2.1}
\begin{aligned}
\nabla\cdot a(\nabla\phi_i+e_i) &= 0;\\
\nabla\cdot\sigma_i &= q_i:= a(\nabla\phi_i+e_i)-\bar{a}e_i;\\
-\Delta\sigma_{ijk} &= \partial_j q_{ik} - \partial_k q_{ij},
\end{aligned}
\end{equation}
in the distributional sense on $\mathbb{R}^d$ with
$\bar{a}e_i:=\big\langle a(\nabla\phi_i +e_i)\big\rangle$ and
$(\nabla\cdot\sigma_i)_j:=\sum_{k=1}^d\partial_k\sigma_{ijk}$. Also,
the field $\sigma$ is skew-symmetric in its last indices, that is
$\sigma_{ijk} = -\sigma_{ikj}$.
%\begin{equation}\label{}
% \sigma_{ijk} = -\sigma_{ikj}.
%\end{equation}
Furthermore, if $\langle\cdot\rangle$ additionally satisfies
the spectral gap condition $\eqref{a:2}$, as well as $\eqref{a:3}$.
Let $\mu_d$ be given in $\eqref{mu_d}$.
For any $p\in[1,\infty)$, there holds
\begin{equation}\label{pri:2}
\big\langle |\nabla (\phi,\sigma)|^p\big\rangle
\lesssim_{\lambda,\lambda_1,\lambda_2,d,p} 1,
\end{equation}
and
\begin{equation}\label{pri:3}
 \Big\langle\big|(\phi,\sigma)(x)
 -(\phi,\sigma)(0)\big|^p \Big\rangle^{\frac{1}{p}}
 \lesssim_{\lambda,\lambda_1,\lambda_2,d,p} \left\{\begin{aligned}
 &\sqrt{2+|x|}; &~&d=1;\\
 &\mu_d(x); &~& d\geq 2.
 \end{aligned}\right.
\end{equation}
\end{lemma}

\begin{proof}
See e.g. \cite[Proposition 4.1]{Josien-Otto22} coupled with
\cite[Lemma 1]{Gloria-Neukamm-Otto20}.
\end{proof}

\begin{lemma}[qualitative theory]\label{lemma:1}
Let $\Omega$ be a bounded Lipschitz domain and $\varepsilon\in(0,1]$.
Assume that the ensemble $\langle\cdot\rangle$
is stationary and satisfies
the spectral gap condition $\eqref{a:2}$,
as well as $\eqref{a:3}$.  Let $p,\gamma\in(1,\infty)$
and $p',\gamma'$ be the associated conjugate index satisfying
$0<p'-1\ll 1$ and $0<\gamma'-1\ll 1$. Then, for any \textcolor[rgb]{0.00,0.00,1.00}{$\nu>0$}, there exist
$\theta\in (0,1)$, depending on $\lambda,d,\gamma,p,M_0,\nu$,
and a  stationary random field
$\chi_*$, satisfying
\begin{equation}\label{m-1}
\langle
\chi_*^\beta\rangle\lesssim_{\lambda,\lambda_1,\lambda_2,d,\beta,\theta} 1
\end{equation}
for any $\beta\in[1,\infty)$, such that
\begin{equation}\label{key:1}
\Big(\frac{1}{R}\Big)^{\frac{1}{\gamma p'}}
 \Big(\dashint_{B_{2R}}
 |(\phi,\sigma)-(\phi,\sigma)_{2R}|^{2p}\Big)^{\frac{1}{p}}\leq \theta
 \quad\text{and}\quad
\Big(\frac{1}{R}\Big)^{\frac{1}{\gamma p'}}
 \Big(\dashint_{B_{2R}}
 |\nabla\phi|^{2p}\Big)^{\frac{1}{p}}\leq \theta
 \qquad \forall R\geq \chi_*,
\end{equation}
and for any
solution $u_\varepsilon$ of $\nabla\cdot a^\varepsilon\nabla u_\varepsilon = 0$
in $D_{8r}$
with $u_\varepsilon = 0$
on $\Delta_{8r}$, there exists $\bar{u}_r\in H^1(D_{2r})$ satisfying
$\nabla\cdot \bar{a} \nabla \bar{u}_r = 0$ in $D_{2r}$
with $\bar{u}_r = u_\varepsilon$ on $\partial D_{2r}$, such that
\begin{equation}\label{pri:2.0}
\dashint_{D_r}|\nabla u_\varepsilon -
(e_i+(\nabla\phi_i)^\varepsilon)\partial_i\bar{u}_r|^2 \leq \nu^2
 \dashint_{\textcolor[rgb]{1.00,0.00,0.00}{D_{4r}}}|\nabla u_\varepsilon|^2,
\end{equation}
and
\begin{equation}\label{pri:1}
\dashint_{D_r}|u_\varepsilon - \bar{u}_r|^2 \leq \nu^2
 \dashint_{D_{8r}}|u_\varepsilon|^2
\end{equation}
hold for any $r\geq \varepsilon\chi_*$.
\end{lemma}

\begin{remark}
\emph{If the (admissible) coefficients satisfy the symmetry condition
$a=a^*$, then we can choose $\gamma=1$ in $\eqref{key:1}$. Concerned with the Neumann boundary problem,
we only modify the equation which $\bar{u}_r$ satisfies by the new one:
$\nabla\cdot \bar{a} \nabla \bar{u}_r = 0$ in $D_{2r}$
with $n\cdot\bar{a}\nabla\bar{u}_r = n\cdot a^\varepsilon
\nabla u_\varepsilon$ on $\partial D_{2r}$,
where $n$ is the outward unit normal
vector associated with $\partial\Omega$.
Then the proof for Neumann boundary conditions follows
from the same arguments as those given for the Dirichlet one.
Besides, the above results trivially hold for
the interior case.}
\end{remark}

\begin{remark}
\emph{We mention that the starting point in
the proofs of Lemma \ref{lemma:1}, as well as Theorems \ref{thm:1} and \ref{thm:2},
is the error of the following two-scale expansions:
\begin{equation}\label{expansion:1}
w_\varepsilon
:= u_\varepsilon - u -\varepsilon \phi_i^\varepsilon \varphi_i,
\end{equation}
and $\varphi_i\in H_0^1(\Omega)$ with $u$ will be fixed according to the concrete environment. In general, we will introduce a cut-off function to make it
satisfy the homogenous boundary condition as
mentioned earlier in $\eqref{expansion:1-00}$.
%(also see Section $\ref{section:4}$).
%Therefore, the region will be
%divided into ``layer'' and ``co-layer'' parts,
%and we reduce the boundary layer problem to control the gradient of the effective solution on the layer part
%and second order derivatives on the co-layer one, respectively.
%That is where the nontangential maximal function plays a role, and
%classical results of boundary value problems on
%non-smooth domains enter.
}
\end{remark}

\begin{proof}
By a rescaling argument it is fine to assume $r=1$. We start from denoting
the ``layer'' region of $D_2$ with the width $r$ by  $O_r(D_2):=\{x\in D_2:\text{dist}(x,\partial D_2)\leq r\}$. Without confusion,
$O_r(D_2)$ is simply written as $O_r$ throughout the proof.
Let $\eta_{\varepsilon}\in C_0^1(D_2)$ be a cut-off function satisfying $\eta_{\varepsilon} = 1$ on $D_2\setminus O_{2\varepsilon}$,
$\eta_{\varepsilon} = 0$ outside $D_2\setminus O_{\varepsilon}$ and $|\nabla\eta_{\varepsilon}|\lesssim_d 1/\varepsilon$ on $O_{2\varepsilon}\setminus O_{\varepsilon}$. Let $\bar{u}_1$
be given by the equations: $\nabla\cdot \bar{a} \nabla \bar{u}_1 = 0$ in $D_{2}$ with $\bar{u}_1 = u_\varepsilon$ on $\partial D_{2}$. By choosing $\varphi_i = \eta_{\varepsilon}\partial_i\bar{u}_1$ and $u=\bar{u}_1$ in $\eqref{expansion:1}$, and employing the equations $\eqref{eq:2.1}$ together with the skew-symmetry of $\sigma_i$, we can derive
\begin{equation}\label{pde:3}
\begin{aligned}
-\nabla\cdot a^\varepsilon\nabla w_\varepsilon
= \nabla\cdot\Big[\varepsilon\big(a^\varepsilon\phi_i^\varepsilon
-\sigma_i^\varepsilon)\nabla \varphi_i
+(a^\varepsilon-\bar{a})(\nabla \bar{u}_1 -\varphi)\Big]
\quad \text{in}~~D_2
\end{aligned}
\end{equation}
with $w_\varepsilon= 0$ on $\partial D_2$, whose details
may be found in \cite[pp.8-10]{Josien-Otto22}. Then the energy estimate together with the definition of $\eta_{\varepsilon}$ leads to
\begin{equation}\label{f:1}
\begin{aligned}
 \int_{D_2}|\nabla w_\varepsilon|^2
& \lesssim_{\lambda,d}
 \int_{D_2} |\varepsilon(\phi_i^\varepsilon,\sigma_i^\varepsilon)|^2
 |\nabla (\eta_{\varepsilon}\partial_i\bar{u}_1)|^2
 + \int_{O_{2\varepsilon}}
 |\nabla \bar{u}_1|^2 \\
&\lesssim \underbrace{\int_{O_{2\varepsilon}}|(\phi_i^\varepsilon,\sigma_i^\varepsilon)|^2
 |\partial_i\bar{u}_1|^2}_{I_1}
 + \underbrace{\int_{D_2\setminus O_\varepsilon} |\varepsilon(\phi_i^\varepsilon,\sigma_i^\varepsilon)|^2
 |\nabla\partial_i\bar{u}_1|^2}_{I_2}
 +\varepsilon^{\frac{1}{p}}
 \underbrace{\Big(\int_{D_2}
 |\nabla \bar{u}_1|^{2p'}\Big)^{\frac{1}{p'}}}_{I_3},
\end{aligned}
\end{equation}
and the argument is therefore reduced to showing the estimates of
$I_1$, $I_2$, and $I_3$ above.

We start from the term $I_1$ (i.e., the layer part estimate), and it follows from H\"older's inequality
that
\begin{equation}\label{f:2.0}
\begin{aligned}
 I_1 \leq
 \Big(\int_{B_2}|(\phi^\varepsilon,\sigma^\varepsilon)|^{2p}\Big)^{\frac{1}{p}}
 \Big(\int_{O_{2\varepsilon}}
 |\nabla \bar{u}_1|^{2p'}\Big)^{\frac{1}{p'}}
\leq \varepsilon^{\frac{1}{\gamma p'}} \Big(\int_{B_2}|(\phi^\varepsilon,\sigma^\varepsilon)|^{2p}\Big)^{\frac{1}{p}}
 \Big(\int_{D_2}
 |\nabla \bar{u}_1|^{2p'\gamma'}\Big)^{\frac{1}{\gamma'p'}},
\end{aligned}
\end{equation}
where we mention that $p',\gamma'$ are the conjugate index of $p,\gamma$,  satisfying $0<p'-1\ll 1$ and $0<\gamma'-1\ll 1$, respectively.
By using $W^{1,q}$ estimates (see e.g. \cite{Giaquinta-Martinazzi12} with $0<q-2\ll 1$)
and Lemma $\ref{lemma:15}$, we have
\begin{equation}\label{f:2.2}
\begin{aligned}
 \Big(\int_{D_2}
 |\nabla \bar{u}_1|^{2p'\gamma'}\Big)^{\frac{1}{\gamma'p'}}
\lesssim
 \Big(\int_{D_2}
 |\nabla u_\varepsilon|^{2p'\gamma'}\Big)^{\frac{1}{\gamma'p'}}
\lesssim^{\eqref{pri:10}}
\int_{D_4}
 |\nabla u_\varepsilon|^{2}.
\end{aligned}
\end{equation}
This together with $\eqref{f:2.0}$ gives us
\begin{equation}\label{f:2.3}
I_1 \lesssim \varepsilon^{\frac{1}{\gamma p'}}
\Big(\dashint_{B_{2/\varepsilon}}|(\phi,\sigma)|^{2p}\Big)^{\frac{1}{p}}
\int_{D_4}
 |\nabla u_\varepsilon|^{2}.
\end{equation}

For the term $I_2$ (i.e., the co-layer part estimate), similarly, we begin from H\"older's inequality
\begin{equation}\label{f:2.29*}
I_2 \leq \Big(\int_{B_2} |\varepsilon(\phi_i^\varepsilon,\sigma_i^\varepsilon)|^{2p}
\Big)^{\frac{1}{p}}
\Big(\int_{D_2\setminus O_\varepsilon}
 |\nabla\partial_i\bar{u}_1)|^{2p'}\Big)^{\frac{1}{p'}},
\end{equation}
and then applying for the interior Lipschitz estimate,
H\"older' inequality, and the geometric properties of integrals, in the order, we have
\begin{equation}\label{f:2.29}
\begin{aligned}
\int_{D_2\setminus O_\varepsilon}|\nabla^2\bar{u}_1|^{2p'}
&\lesssim
\int_{D_2\setminus O_\varepsilon}
\frac{1}{[\delta(x)]^{2p'}}\dashint_{B(x,\delta(x)/4)}|\nabla\bar{u}_1|^{2p'}\\
&\lesssim
\varepsilon^{\frac{1}{\gamma}-2p'}\Big(
\int_{D_2\setminus O_\varepsilon}\dashint_{B(x,\delta(x)/4)}|\nabla\bar{u}_1|^{2p'\gamma'}
\Big)^{\frac{1}{\gamma'}}
\lesssim
\varepsilon^{\frac{1}{\gamma}-2p'}\Big(
\int_{D_2}|\nabla\bar{u}_1|^{2p'\gamma'}
\Big)^{\frac{1}{\gamma'}},
\end{aligned}
\end{equation}
where $\delta(x):=\text{dist}(x,\partial D_2)$.
Plugging this back into the estimate $\eqref{f:2.29*}$ leads to
\begin{equation}\label{f:2.4}
I_2 \lesssim
\varepsilon^{\frac{1}{\gamma p'}-2}\Big(\int_{B_2} |\varepsilon(\phi_i^\varepsilon,\sigma_i^\varepsilon)|^{2p}
\Big)^{\frac{1}{p}}
\Big(
\int_{D_2}|\nabla\bar{u}_1|^{2p'\gamma'}
\Big)^{\frac{1}{\gamma'p'}}
\lesssim^{\eqref{f:2.2}} \varepsilon^{\frac{1}{\gamma p'}}
\Big(\dashint_{B_{2/\varepsilon}}|(\phi,\sigma)|^{2p}\Big)^{\frac{1}{p}}
\int_{D_4}
 |\nabla u_\varepsilon|^{2}.
\end{equation}

Concerned with the last term $I_3$, it follows from the estimate $\eqref{f:2.2}$ that
\begin{equation}\label{f:2.5}
  I_3 \lesssim \varepsilon^{\frac{1}{p}} \int_{D_4}
 |\nabla u_\varepsilon|^{2}.
\end{equation}

As a result, plugging $\eqref{f:2.3}$, $\eqref{f:2.4}$, and $\eqref{f:2.5}$ back into $\eqref{f:1}$, we immediately obtain that
\begin{equation}\label{f:2.22}
\begin{aligned}
\int_{D_2}|\nabla w_\varepsilon|^2
&\lesssim \Big\{\varepsilon^{\frac{1}{\gamma p'}}\Big(\dashint_{B_{2/\varepsilon}} |(\phi,\sigma)|^{2p}
\Big)^{\frac{1}{p}}
+\varepsilon^{\frac{1}{p}}\Big\}\int_{D_4}|\nabla u_\varepsilon|^2.
\end{aligned}
\end{equation}

In view of $\eqref{expansion:1}$, we have
\begin{equation}\label{f:2.20}
\begin{aligned}
&\int_{D_1}|\nabla u_\varepsilon - (e_i+(\nabla\phi_i)^\varepsilon)\partial_i\bar{u}_1|^2
\lesssim  \int_{D_2}|\nabla w_\varepsilon|^2
+ \int_{D_2\setminus O_{\varepsilon}}|\varepsilon\phi_i^\varepsilon \nabla \partial_i\bar{u}_1|^2
+ \int_{O_{\varepsilon}}
|(\nabla\phi_i)^\varepsilon \partial_i\bar{u}_1|^2.
\end{aligned}
\end{equation}
Moreover, there holds
\begin{equation*}
\int_{O_\varepsilon}
|(\nabla\phi_i)^\varepsilon \partial_i\bar{u}_1|^2
\lesssim^{\eqref{f:2.2}}
\varepsilon^{\frac{1}{\gamma p'}}
\Big(\dashint_{B_{2/\varepsilon}}|\nabla\phi_i|^{2p}
\Big)^{\frac{1}{p}}
\int_{D_4}|\nabla u_\varepsilon|^{2},
\end{equation*}
and plugging this and the estimate $\eqref{f:2.4}$ back into $\eqref{f:2.20}$, we obtain
\begin{equation}\label{f:2.21}
\begin{aligned}
\int_{D_1}|\nabla u_\varepsilon - (e_i+(\nabla\phi_i)^\varepsilon)\partial_i\bar{u}_1|^2
&\lesssim \bigg\{\varepsilon^{\frac{1}{\gamma p'}}\Big(\dashint_{B_{2/\varepsilon}} |(\phi,\sigma,\nabla\phi)|^{2p}
\Big)^{\frac{1}{p}}
+\varepsilon^{\frac{1}{p}}\bigg\}\int_{D_4}|\nabla u_\varepsilon|^2.
\end{aligned}
\end{equation}
By using a triangle inequality, Poincar\'e's inequality, and Caccioppoli's inequality, we arrive at
\begin{equation}\label{f:2.6}
\begin{aligned}
\int_{D_1}|u_\varepsilon - \bar{u}_1|^2
&\leq  \int_{D_2}|w_\varepsilon|^2
+ \int_{D_2}|\varepsilon\phi_i^\varepsilon\eta_{\varepsilon}\partial_i\bar{u}_1|^2
\lesssim \bigg\{\varepsilon^{\frac{1}{\gamma p'}}\Big(\dashint_{B_{2/\varepsilon}} |(\phi,\sigma)|^{2p}
\Big)^{\frac{1}{p}}
+\varepsilon^{\frac{1}{p}}\bigg\}\int_{D_8}|u_\varepsilon|^2.
\end{aligned}
\end{equation}

If we replace the corrector $\phi_i$ in the error term $\eqref{expansion:1}$ by $\phi_i-\dashint_{B_{2/\varepsilon}}\phi_i$ (denoted by $\bar{\phi}_i$, and
likewise for $\sigma$), and
rescaling back, then the stated estimates $\eqref{f:2.21}$ can be rewritten as
\begin{equation*}
\dashint_{D_r}|\nabla u_\varepsilon - (e_i+(\nabla\phi_i)^\varepsilon)\partial_i\bar{u}_r|^2
\lesssim \bigg\{\Big(\frac{\varepsilon}{r}\Big)^{\frac{1}{\gamma p'}}\Big(\dashint_{B_{2r/\varepsilon}} |
(\bar{\phi},\bar{\sigma},\nabla\phi)|^{2p}
\Big)^{\frac{1}{p}}
+\Big(\frac{\varepsilon}{r}\Big)^{\frac{1}{p}}\bigg\}
\dashint_{D_{4r}}|\nabla u_\varepsilon|^2.
\end{equation*}
%\begin{equation}
%\dashint_{D_r}|u_\varepsilon - \bar{u}_r|^2
%\lesssim \Big\{\Big(\frac{\varepsilon}{r}\Big)^{\frac{1}{\gamma p'}}
%\Big(\dashint_{B_{2r/\varepsilon}} |(\phi,\sigma)-(\phi,\sigma)_{2r/\varepsilon}|^{2p}
%\Big)^{\frac{1}{p}}
%+\Big(\frac{\varepsilon}{r}\Big)^{\frac{1}{p}}\Big\}
%\dashint_{D_{8r}}|u_\varepsilon|^2.
%\end{equation}
In this regard, for any $\theta>0$ (to be determined later) we can
define the stationary field $\chi_*$ (we still call it minimal radius, which
plays the same role as the minimal radius $r_*$ defined in \cite{Gloria-Neukamm-Otto20}) as follows:
\begin{equation}\label{min_rad}
\chi_{*}(0):=
 \inf\Bigg\{l>0:
 \Big(\frac{1}{R}\Big)^{\frac{1}{\gamma p'}}
 \bigg[\Big(\dashint_{B_{2R}}
 |(\phi,\sigma)-(\phi,\sigma)_{2R}|^{2p}
\Big)^{\frac{1}{p}}
+
\Big(\dashint_{B_{2R}}
 |\nabla\phi|^{2p}
\Big)^{\frac{1}{p}}
\bigg]\leq \theta;~\forall R\geq l\Bigg\}\vee \theta^{-p}
\end{equation}
(we usually omit the center point, and denote it by $\chi_*$).
By the definition of $\chi_*$, one can straightforwardly derive that
\begin{equation*}
\begin{aligned}
\dashint_{D_r}|\nabla u_\varepsilon - (e_i+(\nabla\phi_i)^\varepsilon)\partial_i\bar{u}_r|^2
&\lesssim \theta\dashint_{D_{4r}}|\nabla u_\varepsilon|^2
\leq \nu^2\dashint_{D_{4r}}|\nabla u_\varepsilon|^2
\qquad \forall r\geq \varepsilon\chi_{*},
\end{aligned}
\end{equation*}
which proves the stated result $\eqref{pri:2.0}$ by preferring
$\theta = \nu^2/ C(\lambda,d,p,\gamma,M_0)$. By the same token, we have
\begin{equation*}
\dashint_{D_r}|u_\varepsilon - \bar{u}_r|^2
\lesssim \theta
\dashint_{D_{8r}}|u_\varepsilon|^2
\leq \nu^2 \dashint_{D_{8r}}|u_\varepsilon|^2
\qquad \forall r\geq \varepsilon\chi_*.
\end{equation*}

Finally, we show the arguments for $\eqref{m-1}$,
which is mainly based upon a dyadic decomposition of scales
and Lemma $\ref{lemma:3}$. To do so, for any $\alpha\in(0,\frac{1}{\gamma p'})$,
it suffices to define a stationary
random field as follows:
\begin{equation*}
 \chi_{**}(0):=\bigg[\frac{1}{\theta}\sup_{R\geq 1}\frac{1}{R^{\alpha}}
 \Big(\dashint_{B_{2R}}|(\bar{\phi}^{R},\bar{\sigma}^{R},\nabla\phi)|^{2p}
 \Big)^{\frac{1}{p}}\bigg]^{\frac{1}{\frac{1}{\gamma p'}-\alpha}}\vee 1,
\end{equation*}
where $(\bar{\phi}^{R},\bar{\sigma}^{R}):=(\phi,\sigma)
-(\phi,\sigma)_{2R}$. Thus, it is not hard to verify that for any
$R\geq\chi_{**}(0)$, there holds
\begin{equation*}
\Big(\frac{1}{R}\Big)^{\frac{1}{\gamma p'}}
 \Big(\dashint_{B_{2R}}
 |(\bar{\phi}^{R},\bar{\sigma}^{R},\nabla\phi)
 |^{2p}\Big)^{\frac{1}{p}}\leq \theta.
\end{equation*}
Moreover, recalling the definition of $\chi_*$, we have
$\chi_{**}(0)\geq \chi_{*}(0)$. Let $\beta_1:=\frac{1}{\gamma p'}-\alpha$,
and we derive that
\begin{equation*}
\begin{aligned}
\langle|\chi_{**}|^\beta\rangle^{\frac{1}{\beta}}
&\lesssim_{\theta} \Big\langle\Big(\sup_{R\geq 1}\frac{1}{R^{\alpha}}
\big(\dashint_{B_{2R}}|(\bar{\phi}^{R},\bar{\sigma}^{R},\nabla\phi)|^{2p}
\big)^{\frac{1}{p}}\Big)^{\frac{\beta}{\beta_1}}\Big\rangle^{\frac{1}{\beta}}\\
&\lesssim
\Big\langle\Big(\sum_{R\geq 1\atop\text{dyadic}}\frac{1}{R^{\alpha}}
\big(\dashint_{B_{2R}}|(\bar{\phi}^{R},\bar{\sigma}^{R},\nabla\phi)|^{2p}
\big)^{\frac{1}{p}}\Big)^{\frac{\beta}{\beta_1}}\Big\rangle^{\frac{1}{\beta}}
\lesssim \sum_{R\geq1\atop\text{dyadic}}\frac{1}{R^{\alpha}}\mu_d^2(R)
\lesssim_{\lambda,\lambda_1,\lambda_2,d,\beta,\theta}^{\eqref{pri:2},\eqref{pri:3}} 1,
\end{aligned}
\end{equation*}
which together with $\chi_{**}\geq \chi_{*}$ leads to the desired
estimate $\eqref{m-1}$. We complete the whole proof. 
\end{proof}

\begin{remark}\label{remark:3}
\emph{In \cite{Gloria-Neukamm-Otto20}, the authors gave a way to make the minimal radius   $\frac{1}{8}$-Lipschitz continuous. In terms of $\chi_*$
introduced in $\eqref{min_rad}$, one can similarly construct
$\frac{1}{L}$-Lipschitz continuous random fields, where
$L:=\max\{R_0, d-1, 8\}$.
Let $\theta_0$ be such that $\eqref{pri:2.0}$ holds
for any $r\geq \varepsilon\chi_*(0;\theta)$ with $0<\theta\leq\theta_0$.
We can define $\underline{\chi}_{*}(0)$ as
the largest function with the Lipschitz
constant \textcolor[rgb]{1.00,0.00,0.00}{$\frac{1}{L}$} less than
$\chi_*(0;\theta)$, and $\theta$ will be chosen later, i.e.,
\begin{equation*}
  \underline{\chi}_*(x):=\inf_{y\in\mathbb{R}^d}
  \big(\chi_*(y;\theta)+|x-y|/L\big).
\end{equation*}
By setting $\theta_1:= 2(L+1)^{\frac{d}{2}+1}\theta_0$, it concludes that
$\underline{\chi}_{*}$ is $\frac{1}{L}$-Lipschitz continuous satisfying
$\chi_*(0;\theta_0)\leq \underline{\chi}_{*}(0)\leq \chi_*(0;\theta_1)$. For the ease of statement,  we still use the original notation to represent
the minimal radius with the $\frac{1}{L}$-Lipschitz continuity.}
\end{remark}

\begin{remark}\label{remark:1}
\emph{By the $\frac{1}{L}$-Lipschitz continuity of $\chi_{*}$ and $\langle\chi_*^\beta\rangle\lesssim 1$ with $\beta<\infty$, it is not hard to see that for any $x\in\mathbb{R}^d$, there holds
\begin{equation}\label{f:2.25}
 \Big\langle\big(\sup_{y\in B_1(x)}|\chi_*(y)|\big)^\beta
 \Big\rangle^{\frac{1}{\beta}}
 \lesssim 1,
\end{equation}
where the multiplicative constant is independent of $x$. Moreover, for any $\kappa>0$, $\beta_1<\infty$ and $x_0\in\mathbb{R}^d$, let $X_R(x_0):= R^{-\kappa}\sup_{y\in B_R(x_0)}|\chi_*(y)|$, where $R\geq 1$.
Then, it follows from a covering argument\footnote{In terms of $B_R(x_0)$,
there exists a family of points $\{x_i\}_{i=1}^{R^d}$
such that $B_{R}(x_0)\subset\cup_i B_1(x_i)$.}, subadditivity
of $\langle\cdot\rangle$, and Markov's inequality that
\begin{equation*}
 \big\langle
 |X_R(x_0)|>\rho\big\rangle
\leq \sum_{i=1}^{R^d}\big\langle \sup_{y\in B_1(x_i)}|\chi_*(y)|
>R^{\kappa}\rho \big\rangle
\leq \sum_{i=1}^{R^d}\frac{\big\langle
(\sup_{y\in B_1(x_i)}|\chi_*(y)|)^{\beta} \big\rangle}{\rho^{\beta}R^{\kappa\beta}}
\lesssim^{\eqref{f:2.25}} \rho^{-\beta} R^{d-\kappa\beta}
\lesssim \rho^{-\beta},
\end{equation*}
provided $\beta\geq d/\kappa$. By choosing $\beta>\max\{d/\kappa,\beta_1\}$,  there uniformly holds
\begin{equation}\label{pri:7-1}
\big\langle
 |X_R(x_0)|^{\beta_1}\big\rangle \lesssim 1 + \int_{1}^{\infty}\rho^{\beta_1-1-\beta}d\rho\lesssim 1,
\end{equation}
with respective to $x_0\in\mathbb{R}^d$. One can further derive that
\begin{equation}\label{pri:7}
\sup_{y\in B_R(x_0)}|\chi_*(y)|
\leq X_{R}(x_0)R^{\kappa}.
\end{equation}
%If $B_R(x)\subset B_{\bar{R}}(y)$, we have
%$X_R(x)\leq (\frac{\bar{R}}{R})^{\kappa} X_{\bar{R}}(y)$.
}
\end{remark}

\begin{lemma}[random shift]\label{lemma:2}
Let $0<\varepsilon\leq 1$.  Fix
$\zeta\in C_0^\infty(B_{1/2}(0))$ with $\int_{\mathbb{R}^d}\zeta =1$,
and $\zeta_r(x) := r^{-d}\zeta(x/r)$.
Given $\chi_*$ as in Lemma $\ref{lemma:1}$ with the $\frac{1}{L}$-Lipschitz continuity stated in Remark $\ref{remark:3}$ and
for any fixed $x\in\mathbb{R}^d$, we define the smoothing operator as
\begin{equation}\label{def1}
 S_{*,\varepsilon}(f)(x)
 := \int_{\mathbb{R}^d}
 \zeta_{\varepsilon\chi_*(x/\varepsilon)}(x-y)f(y)dy.
\end{equation}
Let $f\in W^{1,p}(\mathbb{R}^d)$ with $1\leq p<\infty$, then we obtain
\begin{equation}\label{pri:4}
\int_{\mathbb{R}^d} |f - S_{*,\varepsilon}(f)|^p
\lesssim_{d,p} \varepsilon^p \int_{\mathbb{R}^d}
\chi_*^p(\cdot/\varepsilon)|\nabla f|^p.
\end{equation}
\end{lemma}

\begin{proof}
Although $\chi_*$ involves randomness, the proof is totally deterministic like that given for periodic homogenization (see \cite{Shen16}), inspired by Steklov smoothing operator which is originally applied to homogenization problem by V. Zhikov and S. Pastukhova \cite{Zhikov-Pastukhova05}. We provide a proof
for the sake of completeness.

We start from changing a variable, H\"older's inequality and Fubini's theorem that
\begin{equation}\label{f:2.24}
\begin{aligned}
\int_{\mathbb{R}^d}
\big|f(x)-S_{*,\varepsilon}(f)(x)\big|^p dx
&= \int_{\mathbb{R}^d}
\Big|\int_{|z|\leq 1}
\zeta(z)\big(f(x)-f(x-\varepsilon\chi_{*}(x/\varepsilon)z)\big)dz\Big|^p dx\\
&\lesssim
\int_{|z|\leq 1}
\zeta(z)\int_{\mathbb{R}^d}
\big|f(x)-f(x-\varepsilon\chi_*(x/\varepsilon)z)\big|^pdx dz.
\end{aligned}
\end{equation}
For any $x\in\mathbb{R}^d$ and $|z|\leq 1$, we obtain that
\begin{equation}\label{f:2.23}
\begin{aligned}
\big|f(x)-f(x-\varepsilon\chi_*(x/\varepsilon)z)\big|^{p}
\lesssim \varepsilon^p
\int_{0}^{1}|\chi_*(x/\varepsilon)|^p
|\nabla f(x-t\varepsilon\chi_*(x/\varepsilon)z)|^p dt.
\end{aligned}
\end{equation}

Let $\tilde{x}:=x-t\varepsilon\chi_*(x/\varepsilon)z$, and we have
$|\tilde{x}-x|\leq \varepsilon\chi_*(x/\varepsilon)$. In terms of
the $\frac{1}{L}$-Lipschitz continuity of $\chi_{*}$, there holds
$\chi_*(\tilde{x}/\varepsilon)\sim \chi_*(x/\varepsilon)$. Moreover,
we have $d\tilde{x} = \frac{\partial\tilde{x}}{\partial x} dx$
(where the Jacobian matrix $\frac{\partial\tilde{x}}{\partial x}$ is diagonally dominant) and
the Jacobian determinant $|\frac{\partial\tilde{x}}{\partial x}|\sim 1$.
Integrating both sides of $\eqref{f:2.23}$ with respect to $x\in\mathbb{R}^d$ and then changing variable, we have
\begin{equation*}
\begin{aligned}
\int_{\mathbb{R}^d}
\big|f(x)-f(x-\varepsilon\chi_*(x/\varepsilon)z)\big|^{p} dx
&\lesssim \varepsilon^p\int_{\mathbb{R}^d}
\int_{0}^{1}|\chi_*(x/\varepsilon)|^p
|\nabla f(x-t\varepsilon\chi_*(x/\varepsilon)z)|^p dtdx\\
&\lesssim \varepsilon^p\int_{\mathbb{R}^d}
|\chi_*(\tilde{x}/\varepsilon)|^p
|\nabla f(\tilde{x})|^p d\tilde{x}.
\end{aligned}
\end{equation*}
Plugging the above estimate back into
$\eqref{f:2.24}$, we have already proved the stated estimate $\eqref{pri:4}$.
\end{proof}

\begin{remark}
\emph{By rescaling, it is not hard to get that $\tilde{u}_{\varepsilon'}(y):=u_\varepsilon(ry)=u_\varepsilon(x)$ satisfies
$-\nabla\cdot a^{\varepsilon'}\nabla \tilde{u}_{\varepsilon'} = 0$ in $D_4$
with $\tilde{u}_{\varepsilon'} = 0$ on $\Delta_4$, in which $\varepsilon'=\varepsilon/r$.
It follows from the estimates $\eqref{f:2.22}$, $\eqref{f:2.3}$, and
$\eqref{f:2.4}$ that
\begin{equation}\label{f:2.26}
\begin{aligned}
\int_{D_1}|\nabla \tilde{u}_{\varepsilon'} - (e_i+(\nabla\phi_i)^{\varepsilon'})\eta_{\varepsilon'}\partial_i\tilde{\bar{u}}_1|^2
&\lesssim \Big\{{\varepsilon'}^{\frac{1}{\gamma p'}}\Big(\dashint_{B_{2/{\varepsilon'}}} |(\phi,\sigma)|^{2p}
\Big)^{\frac{1}{p}}
+{\varepsilon'}^{\frac{1}{p}}\Big\}\int_{D_4}|\nabla \tilde{u}_{\varepsilon'}|^2,
\end{aligned}
\end{equation}
where $\eta_{\varepsilon'}$ is the cut-off function given as in the proof of Lemma $\ref{lemma:1}$,
which together with the introduced minimal radius
$\chi_*$ leads to
\begin{equation}\label{pri:2.1}
\begin{aligned}
\dashint_{D_r}|\nabla u_\varepsilon - (e_i+(\nabla\phi_i)^\varepsilon)
\eta_\varepsilon\partial_i\bar{u}_r|^2
&
\leq \nu^2\dashint_{D_{4r}}|\nabla u_\varepsilon|^2
\qquad \forall r\geq \varepsilon\chi_{*},
\end{aligned}
\end{equation}
where $\eta_\varepsilon:
=\eta_{\varepsilon'}(\cdot/r)$.}
\end{remark}

\begin{corollary}\label{cor:1}
Let $\chi_{*}$, $u_\varepsilon$, and $\bar{u}_r$ be introduced as in Lemma $\ref{lemma:1}$, and we also
assume the same conditions hold as in Lemma $\ref{lemma:1}$.
Let $\varphi_i:=S_{*,\varepsilon}(\eta_{\varepsilon}
\partial_i\bar{u}_r)$ and $u:=\bar{u}_r$ in $\eqref{expansion:1}$, where the definition of $S_{*,\varepsilon}$ is given in Lemma $\ref{lemma:2}$ and
$\eta_\varepsilon$ is a cut-off function similar to
that given in $\eqref{pri:2.1}$. Then, for any $\nu>0$, there exists
%\textcolor[rgb]{0.00,0.00,1.00}{$\theta\in (0,1)$}, depending on $\lambda,d,\gamma,p,M_0,\epsilon$, and
a stationary random field
$\tilde{\chi}_*\geq \chi_*$, satisfying
$\langle \tilde{\chi}_*^\beta\rangle\lesssim 1$ for any $\beta\in[1,\infty)$,
such that
\begin{equation}\label{pri:2.1-1}
\dashint_{D_r}|\nabla u_\varepsilon - (e_i+(\nabla\phi_i)^\varepsilon)
\varphi_i|^2 \leq \nu^2
 \dashint_{D_{4r}}|\nabla u_\varepsilon|^2,
\end{equation}
holds for any $r\geq \varepsilon\tilde{\chi}_*$.
\end{corollary}

\begin{proof}
By the definition of $\varphi_i$, it follows from
the estimate $\eqref{f:2.26}$ that
\begin{equation}\label{f:2.27}
\begin{aligned}
&\dashint_{D_r} |\nabla u_\varepsilon
-(e_i+(\nabla\phi_i)^{\varepsilon})\varphi_i|^{2}
\leq \dashint_{D_r} |\nabla u_\varepsilon
-(e_i+(\nabla\phi_i)^{\varepsilon})\eta_\varepsilon
\partial_i{\bar{u}_r}|^{2}
+
\dashint_{D_r} |(e_i+(\nabla\phi_i)^{\varepsilon})
\big(\eta_\varepsilon\partial_i{\bar{u}_r}-\varphi_i\big)|^{2}\\
&\lesssim \bigg\{\Big(\frac{\varepsilon}{r}\Big)^{\frac{1}{\gamma p'}}\Big(\dashint_{B_{2r/\varepsilon}} |
(\bar{\phi},\bar{\sigma})|^{2p}
\Big)^{\frac{1}{p}}
+\Big(\frac{\varepsilon}{r}\Big)^{\frac{1}{p}}\bigg\}
\dashint_{D_{4r}}|\nabla u_\varepsilon|^2
 +
\dashint_{D_r} |(e_i+(\nabla\phi_i)^{\varepsilon})
\big(\eta_\varepsilon
\partial_i{\bar{u}_r}-\varphi_i\big)|^{2},
\end{aligned}
\end{equation}
where the notation $(\bar{\phi},\bar{\sigma})$ is taken from the proof of Lemma $\ref{lemma:1}$.
We continue to estimate the last term
above by H\"older's inequality, Lemma $\ref{lemma:2}$ and
Remark $\ref{remark:1}$ as follows:
\begin{equation}\label{f:2.28}
\begin{aligned}
&\dashint_{D_r}|(e_i+(\nabla\phi_i)^{\varepsilon})
(\eta_\varepsilon\partial_i{\bar{u}_r}
-\varphi_i)|^{2}
\leq \Big(\dashint_{D_r}|(e_i+(\nabla\phi_i)^{\varepsilon})
|^{2p}\Big)^{\frac{1}{p}}
\Big(\dashint_{D_r}
|(\eta_\varepsilon\partial_i{\bar{u}_r}
-\varphi_i)|^{2p'}\Big)^{\frac{1}{p'}} \\
&\lesssim^{\eqref{pri:4}} \varepsilon^2
\Big\{1+\dashint_{B_{\frac{r}{\varepsilon}}}|\nabla \phi|^{2p}\Big\}^{\frac{1}{p}}
\bigg\{
\dashint_{D_r}|\nabla \eta_\varepsilon\nabla{\bar{u}_r}
+ \eta_\varepsilon\nabla^2{\bar{u}_r}|^{2p'}|\chi_*(\cdot/\varepsilon)|^{2p'}
\bigg\}^{\frac{1}{p'}}\\
&\lesssim^{\eqref{pri:7}}
\Big(\frac{\varepsilon}{r}\Big)^{\frac{1}{\gamma p'}-\kappa}
X_{\frac{r}{\varepsilon}}\Big\{1+\dashint_{B_{\frac{r}{\varepsilon}}}|\nabla \phi|^{2p}\Big\}^{\frac{1}{p}}\dashint_{D_{4r}}|\nabla u_\varepsilon|^2,
\end{aligned}
\end{equation}
where we also employ the layer and co-layer type estimates
$\eqref{f:2.2}$, $\eqref{f:2.29}$ in the last step.
Then, plugging the estimate $\eqref{f:2.28}$ back into $\eqref{f:2.27}$
we have
\begin{equation*}
\begin{aligned}
\dashint_{D_r} |\nabla u_\varepsilon
-(e_i+(\nabla\phi_i)^{\varepsilon})\varphi_i|^{2}
\lesssim
\Big(\frac{\varepsilon}{r}\Big)^{\frac{1}{\gamma p'}-\kappa}
X_{\frac{r}{\varepsilon}}\bigg\{\Big(\dashint_{B_{2r/\varepsilon}} |
(\bar{\phi},\bar{\sigma},\nabla\phi)|^{2p}
\Big)^{\frac{1}{p}}
+1\bigg\}\dashint_{D_{4r}}|\nabla u_\varepsilon|^2,
\end{aligned}
\end{equation*}
where one can prefer $0<\kappa<\frac{1}{\gamma p'}$ in the above estimate.
As a result, similar to the definition of $\chi_*$
given in Lemma $\ref{lemma:1}$, for any
$\theta>0$, we can define the minimal radius as
\begin{equation*}
\tilde{\chi}_{*}(0):=
 \inf\Bigg\{l>0:
 \Big(\frac{1}{R}\Big)^{\frac{1}{\gamma p'}-\kappa}
 X_R\bigg[\Big(\dashint_{B_{2R}}
 |(\phi,\sigma)-(\phi,\sigma)_{2R}|^{2p}
\Big)^{\frac{1}{p}}
+
\Big(\dashint_{B_{2R}}
 |\nabla\phi|^{2p}
\Big)^{\frac{1}{p}}
+1
\bigg]\leq \theta;\forall R\geq l\Bigg\}.
\end{equation*}
By definition, it is not hard to observe that $\tilde{\chi}_*\geq \chi_*$.
Also, by a similar argument given for $\eqref{m-1}$,
the $\beta$-th moment of
$\tilde{\chi}_*$ follows from $\eqref{pri:2}$, $\eqref{pri:3}$,
$\eqref{pri:7-1}$, and the dyadic decomposition of scales.
\end{proof}

\begin{remark}\label{remark:4}
\emph{The minimal radius $\tilde{\chi}_*$ introduced in Corollary $\ref{cor:1}$ can be verified to have
the $\frac{1}{L}$-Lipschitz continuity by the same argument stated in Remark $\ref{remark:3}$.
Therefore, we merely use the notation $\chi_*$ to represent the minimal radius in later sections, although
it indeed comes from Corollary $\ref{cor:1}$ sometimes.}
\end{remark}

%\begin{lemma}[Shen's lemma]\label{shen's lemma1}
%Let $q>2$ and $\Omega$ be a bounded Lipschitz domain.
%Let $F\in L^2(\Omega)$ and $f\in L^p(\Omega)$
%for some $2<p<q$. Suppose that for each ball $B$
%with the property that $|B|\leq  c_0|\Omega|$
%and either $4B\subset\Omega$ or $B$ is centered on
%$\partial\Omega$, there exist two measurable functions
%$F_B$ and $R_B$ on $\Omega\cap2B$,
%such that $|F|\leq |F_B| + |R_B|$ on $\Omega\cap2B$,
%\begin{subequations}
%\begin{align}
%\Big(\dashint_{2B\cap\Omega}|R_B|^q\Big)^{\frac{1}{q}}
%&\leq N_1 \bigg\{
%\Big(\dashint_{4B\cap\Omega}|F|^2\Big)^{\frac{1}{2}}
%+\sup_{4B_0\supseteq B^\prime\supseteq B}
%\Big(\dashint_{B^\prime\cap\Omega}|f|^2\Big)^{\frac{1}{2}}
%\bigg\}
%\label{pri:3.2a}\\
%\Big(\dashint_{2B\cap\Omega}|F_B|^2\Big)^{\frac{1}{2}}
%&\leq N_2
%\sup_{4B_0\supseteq B^\prime\supseteq B}
%\Big(\dashint_{B^\prime\cap\Omega}|f|^2\Big)^{\frac{1}{2}},
%\label{pri:3.2a}
%\end{align}
%\end{subequations}
%where $N_1, N_2>0$
%and $0<c_0<1$. Then $F\in L^p(\Omega)$
%and
%\begin{equation}\label{pri:3.1}
%\Big(\dashint_{B_0\cap\Omega}|F|^p\Big)^{\frac{1}{p}}
%\leq C\bigg\{\Big(\dashint_{4B_0\cap\Omega}|F|^2\Big)^{\frac{1}{2}}
%+\Big(\dashint_{4B_0\cap\Omega}|f|^p\Big)^{\frac{1}{p}}\bigg\},
%\end{equation}
%where $C$ depends at most on $N_1$,
% $N_2$, $c_0$, $p$, $q$ and the Lipschitz character of $\Omega$.
%\end{lemma}

Shen's real argument\footnote{
``Real'' is referred to the real variable theory,
which includes the notion of maximal functions,
certain covering lemmas (e.g. Vitali covering lemma,
Calder\'on-Zygmund decomposition),
the splitting of functions into
their large and small parts, etc.} was developed by Z. Shen \cite{Shen05,Shen07}, originally
inspired by L. Caffarelli and I. Peral \cite{Caffarelli-Peral98}.
Recently, he provided a weighted version (stated below)
which played an important role in the present work.

\begin{lemma}[Shen's real argument of weighted version \cite{Shen20}]\label{shen's lemma2}
Let $0<p_0<p<p_1$, $\omega\in A_{\frac{p}{p_0}}$weight
and $\Omega$ be a bounded Lipschitz domain,
$F\in L^{p_0}(\Omega)$ and $\textcolor[rgb]{0.00,0.00,1.00}{g}\in L^{p_1}(\Omega)$.
Let $B_0:=B_{r_0}(x_0)$, where $x_0\in\partial\Omega$ and
$0<r_0<c_0\text{diam}(\Omega)$. Suppose that for each ball $B$
with the property that $|B|\leq  c_1|B_0|$
and either $4B\subset 2B_0\cap\Omega$ or $x_B\in\partial\Omega\cap 2B_0$, there exist two measurable functions
$F_B$ and $R_B$ on $\Omega\cap2B$,
such that $|F|\leq |W_B| +
|V_B|$ on $\Omega\cap2B$, and
\begin{subequations}
\begin{align}
\Big(\dashint_{2B\cap\Omega}|V_B|^{p_1}\omega
\Big)^{\frac{1}{p_1}}
&\lesssim_{N_1} \bigg\{
\Big(\dashint_{\alpha B\cap\Omega}|F|^{p_0}\Big)^{\frac{1}{p_0}}
+\sup_{4B_0\supseteq B^\prime\supseteq B}
\Big(\dashint_{B^\prime\cap\Omega}|g|^{p_0}
\Big)^{\frac{1}{p_0}}
\bigg\}\Big(\dashint_{B}\omega\Big)^{1/{p_1}},
\label{pri:3.3a}\\
\Big(\dashint_{2B\cap\Omega}|W_B|^{p_0}
\Big)^{\frac{1}{p_0}}
&\lesssim_{N_2}
\sup_{4B_0\supseteq B^\prime\supseteq B}
\Big(\dashint_{B^\prime\cap\Omega}|g|^{p_0}
\Big)^{\frac{1}{p_0}}
+\eta\Big(\dashint_{\alpha B\cap\Omega}
|F|^{p_0}\Big)^{\frac{1}{p_0}},
\label{pri:3.3b}
\end{align}
\end{subequations}
where $N_1, N_2>0$, $\eta\geq0$
and $0<c_1<1$ with $\alpha>2$. Then there exists $\theta_{0}>0$, depending on
$d,p_0,p_1,p,N_1,[\omega]_{A_{p/p_0}}$ and the
Lipschitz character of $\Omega$, with the property that if $0\leq \eta\leq\eta_0$, then
\begin{equation}\label{pri:3.4}
\Big(\dashint_{B_0\cap\Omega}|F|^p\omega\Big)^{\frac{1}{p}}
\lesssim \bigg\{\Big(\dashint_{4B_0\cap\Omega}|F|^{p_0}\Big)^{\frac{1}{p_0}}
\Big(\dashint_{B_0}\omega\Big)^{1/{p}}
+\Big(\dashint_{4B_0\cap\Omega}|\textcolor[rgb]{0.00,0.00,1.00}{g}|^p\omega\Big)^{\frac{1}{p}}\bigg\},
\end{equation}
where the multiplicative constant relies on
$d,p_0,p_1,p,c_1,\textcolor[rgb]{0.00,0.00,1.00}{\alpha},N_1,N_2,[\omega]_{A_{p/p_0}}$ and the Lipschitz
character of $\Omega$.
\end{lemma}

\begin{proof}
See \cite[Theorem 4.1]{Shen20}.
In order to ensure the consistency on symbols in later sections,
we have replaced the original symbols as follows:
$F_B$, $R_B$ and $f$ in the original context are replaced by
$W_B$, $V_B$ and $g$, respectively.
Also, $\alpha=8$ was taken in the original statement, whereas
according to the given proof (see \cite[pp.16]{Shen20})
the concrete value of $\alpha$ is not vital
for the conclusion $\eqref{pri:3.4}$, since
its proof relies on \cite[Theorem 2.1]{Shen20} and
``$100B$'' appeared in the corresponding right-hand side therein is not essential
due to the real variable theory employed there (see
\cite[Theorems 3.1,3.3]{Shen05} or \cite[Theorem 3.2]{Shen07}).
\end{proof}

%\begin{remark}
%\emph{The estimated constant $C$ in $\eqref{pri:3.4}$ polynomially depends on $N_1$ and $N_2$. We refer the reader to \cite{Shen20} for the proof of
%Lemma $\ref{shen's lemma2}$.}
%\end{remark}

\begin{lemma}[primary geometry on integrals]\label{lemma:9}
Let $\chi_{*}$ be given as in Lemma $\ref{lemma:1}$ or Corollary $\ref{cor:1}$ with the $\frac{1}{L}$-Lipschitz continuity, and $\varepsilon\in(0,1]$.
Let $f\in L^1_{\emph{loc}}(\mathbb{R}^d)$, and
$\Omega\subset\mathbb{R}^d$ be a bounded Lipschitz domain with
$X\in\partial\Omega$. Then
there hold the following inequalities:
\begin{itemize}
  \item For all $B_r(X)
  \subset\mathbb{R}^d$ and $x_0\in B_r(X)\cap\Omega$ with
   $r<\frac{\varepsilon}{4}\chi_*(x_0/\varepsilon)$, we have
  \begin{equation}\label{pri:3.5-1}
  \Big(\dashint_{U_{*,\varepsilon}(x_0)}|f|\Big)^{\frac{1}{s}}
  \lesssim
  \dashint_{D_{5r}(X)}
  \Big(
  \dashint_{U_{*,\varepsilon}(x)}
  |f|\Big)^{\frac{1}{s}} dx,
  \end{equation}
  where $s>0$.
  \item For all balls $B_r(X)
  \subset\mathbb{R}^d$ with $r\geq \frac{\varepsilon}{4}
  \chi_*(X/\varepsilon)$,
  we have
\begin{subequations}
\begin{align}
  &\dashint_{D_r(X)} |f|
  \lesssim \dashint_{D_{2r}(X)}
  \Big(\dashint_{U_{*,\varepsilon}(x)}|f|\Big)dx
  \lesssim \dashint_{D_{7r}(X)} |f|;
  \label{pri:3.6-1}\\
  &
\int_{\Omega}|f|\sim\int_{\Omega}
  \Big(\dashint_{U_{*,\varepsilon}(x)}|f|\Big)dx,
  \label{pri:3.6-1*}
\end{align}
\end{subequations}
\end{itemize}
where the multiplicative constant depends only on $d$ and $M_0$.
\end{lemma}

\begin{proof}
The proof is standard and we provide a proof for the sake of completeness. For the case of $\Omega=\mathbb{R}^d$ we
refer the reader to \cite[Lemma 6.5]{Duerinckx-Otto20} and \cite[Lemma 6.2]{Wang-Xu-Zhao21} for a deterministic case. Throughout the proof,
the key ingredient is the $\frac{1}{L}$-Lipschitz
continuity of $\chi_{*}$. Since
the stated estimates $\eqref{pri:3.5-1}$ and
$\eqref{pri:3.6-1}$ are translate-invariant alone
the boundary $\partial\Omega$, it's fine to assume $X=0$,
and we omit the center of $B_r(0)$ in the notation.

We first show the estimate $\eqref{pri:3.5-1}$ for the case of $s\geq 1$.
For any $x\in B_r\cap\Omega$, noting that
$B_{*,\varepsilon}(x_0)\cap\Omega \subset
B_{\varepsilon\chi_*(x_0/\varepsilon)+2r}(x)\cap\Omega$, we have
\begin{equation*}
 \Big(\dashint_{B_{*,\varepsilon}(x_0)\cap\Omega} |f|\Big)^{\frac{1}{s}}
 \lesssim \dashint_{B_r\cap\Omega} \Big(\dashint_{B_{\varepsilon\chi_{*}(x_0/\varepsilon)+2r}(x)\cap\Omega} |f|\Big)^{\frac{1}{s}}.
\end{equation*}
It's known that
there exists a finite integer $N>0$, depending only on $d$,
and a family of points $\{z_i\}_{i=1}^{N}\subset B_{4r}$,
such that $B_{\varepsilon\chi_{*}(x_0/\varepsilon)+2r}\subset
\cup_{i=1}^N B_{\varepsilon\chi_{*}(x_0/\varepsilon)-r}(z_i)$.
Therefore, by noting $z_i+B_r\subset B_{5r}$ and $\varepsilon\chi_*(x_0/\varepsilon)-r
\leq \varepsilon\chi_{*}(x/\varepsilon)
\leq 2\varepsilon\chi_{*}(x_0/\varepsilon)$ holding for any $x\in B_{5r}$
(due to the $\frac{1}{L}$-Lipschitz continuity of
$\chi_{*}$),  we can derive that
\begin{equation}\label{f:3.30}
\begin{aligned}
 \Big(\dashint_{B_{*,\varepsilon}(x_0)\cap\Omega} |f|\Big)^{\frac{1}{s}}
& \lesssim \dashint_{B_r\cap\Omega} \Big(\sum_{i=1}^{N} \dashint_{B_{\varepsilon\chi_{*}(x_0/\varepsilon)-r}(x+
 z_i)\cap\Omega} |f|\Big)^{\frac{1}{s}}\\
& \lesssim \sum_{i=1}^{N} \dashint_{B_r\cap\Omega} \Big( \dashint_{B_{\varepsilon\chi_{*}(x_0/\varepsilon)-r}(x+
 z_i)\cap\Omega} |f|\Big)^{\frac{1}{s}}
  \lesssim
 \dashint_{B_{5r}\cap\Omega}
 \Big(\dashint_{B_{*,\varepsilon}(x)\cap\Omega} |f|\Big)^{\frac{1}{s}}.
\end{aligned}
\end{equation}
Concerned with the case $0<s<1$, we appeal to the special case $s=1$ proved above and the desired result follows from H\"older's inequality, i.e.,
\begin{equation*}
\Big(\dashint_{B_{*,\varepsilon}(x_0)\cap\Omega} |f|\Big)^{\frac{1}{s}}
\lesssim^{\eqref{f:3.30}}
\Big(\dashint_{B_{5r}\cap\Omega}
 \dashint_{B_{*,\varepsilon}(x)\cap\Omega} |f|\Big)^{\frac{1}{s}}
\leq \dashint_{B_{5r}\cap\Omega}
 \Big(\dashint_{B_{*,\varepsilon}(x)\cap\Omega} |f|\Big)^{\frac{1}{s}}.
\end{equation*}
This coupled with $\eqref{f:3.30}$ yields the whole proof of $\eqref{pri:3.5-1}$.

\medskip

Then, we turn to the estimate $\eqref{pri:3.6-1}$.
For any $|x-y|\leq \varepsilon\chi_*(x/\varepsilon)$, it follows from the $\frac{1}{L}$-Lipschitz continuity of
$\chi_{*}$ that $\varepsilon\chi_{*}(x/\varepsilon)\sim
\varepsilon\chi_{*}(y/\varepsilon)$. Then we have
\begin{equation}\label{f:3.2-1}
\begin{aligned}
\dashint_{B_{2r}\cap\Omega}
  \dashint_{B_{*,\varepsilon}(x)\cap\Omega}|f|dx
\sim \frac{1}{|B_{2r}\cap\Omega|}
\int_{\mathbb{R}^d}|f(y)|
\frac{|B_{2r}\cap\Omega\cap B_{*,\varepsilon}(y)|}{|B_{*,\varepsilon}(y)\cap\Omega|}dy.
\end{aligned}
\end{equation}
On the one hand,
for all $y\in B_r\cap\Omega$ we have $\frac{1}{5}B_{*,\varepsilon}(y)\subset B_{r+\frac{1}{5}\varepsilon\chi_{*}(y/\varepsilon)}\subset B_{2r}$ (due to
the $\frac{1}{L}$-Lipschitz continuity of
$\chi_{*}$), and this together with $\eqref{f:3.2-1}$ leads to
\begin{equation}\label{f:3.3-1}
\dashint_{B_{2r}\cap\Omega}
  \dashint_{B_{*,\varepsilon}(x)\cap\Omega}|f|dx
\geq \frac{C_d}{|D_{2r}|}
\int_{\mathbb{R}^d}|f(y)|
\frac{|B_{2r}\cap\Omega\cap B_{*,\varepsilon}(y)|}{|B_{*,\varepsilon}(y)\cap\Omega|}dy
\gtrsim \dashint_{B_{r}\cap\Omega}|f(y)|dy.
\end{equation}
On the other hand, we collect $y\in\mathbb{R}^d$ such that $B_{2r}\cap B_{*,\varepsilon}(y)\not=\emptyset$, and it follows that
$|y|\leq 2r+\varepsilon\chi_{*}(y/\varepsilon)\leq 2r+\varepsilon\chi_{*}(0)+\frac{|y|}{L}
$. Therefore we have $|y|\leq 7r$. In this regard, from $\eqref{f:3.2-1}$ we obtain
\begin{equation}\label{f:3.4-1}
\dashint_{B_{2r}\cap\Omega}
  \dashint_{B_{*,\varepsilon}(x)\cap\Omega}|f|dx
\leq \frac{C_d}{|D_{2r}|}
\int_{\mathbb{R}^d}|f(y)|
\frac{|B_{2r}\cap\Omega\cap B_{*,\varepsilon}(y)|}{|B_{*,\varepsilon}(y)\cap\Omega|}dy
\lesssim \dashint_{B_{7r}\cap\Omega}|f(y)|dy.
\end{equation}
Consequently, combining the estimates $\eqref{f:3.3-1}$ and $\eqref{f:3.4-1}$
yields the stated estimate $\eqref{pri:3.6-1}$.

Finally, we show the estimate $\eqref{pri:3.6-1*}$.
By the same reason as given for $\eqref{f:3.2-1}$, we can derive that
\begin{equation*}
\begin{aligned}
\int_{\Omega}\Big(\dashint_{U_{*,\varepsilon}(x)}|f|\Big)dx
&=\int_{\mathbb{R}^d}
\int_{\mathbb{R}^d}\frac{1}{|U_{*,\varepsilon}(x)|}
I_{\Omega}(x)I_{\{y\in\Omega:|y-x|<\varepsilon\chi_{*}(x/\varepsilon)\}}(y)
|f(y)|dy dx\\
&\sim
\int_{\mathbb{R}^d}
\int_{\mathbb{R}^d}\frac{1}{|U_{*,\varepsilon}(y)|}
I_{\Omega}(x)I_{\{y\in\Omega:|y-x|<\theta_2^{-1}\varepsilon\chi_{*}(y/\varepsilon)\}}(y)
|f(y)|dy dx =:J,
\end{aligned}
\end{equation*}
where $\theta_2:=(1-1/L)\in[7/8,1)$ (see Remark $\ref{remark:3}$ for $L$).
By Fubini's theorem, we observe that
\begin{equation*}
\int_{\mathbb{R}^d} I_{\Omega}(x)
\Big(\int_{\mathbb{R}^d} I_{\{y\in\Omega:|y-x|<\theta_2^{-1}
\varepsilon\chi_{*}(y/\varepsilon)\}}(y)dy\Big) dx
=\int_{\mathbb{R}^d} I_{\Omega}(y)
\Big(\int_{\mathbb{R}^d} I_{\{x\in\Omega:|x-y|<\theta_2^{-1}
\varepsilon\chi_{*}(y/\varepsilon)\}}(x)
dx\Big) dy
\end{equation*}
This implies that $J=\theta_2^{-d}\int_{\Omega}|f(y)|dy$,
which consequently leads to the stated estimate $\eqref{pri:3.6-1*}$.
\end{proof}

\section{Calder\'on-Zygmund estimates}\label{section:3}

We first introduce some notations used throughout this section.
Let the index $p_*$ be fixed as follows:
\begin{equation}\label{upindex-p}
p_* := \left\{\begin{aligned}
 &\infty, &\quad&\text{if}~ \Omega \text{~is~the regular SKT~}
 (\text{or}~C^1);\\
 &2d/(d-1)+\theta
 &\quad&\text{if}~
 \Omega \text{~is the Lipschitz domain with}~M_0~
 \text{large}\footnote{If the equations under consideration
are scalar valued,
let $p_1 =3+\theta$ if $d\geq3$; and $p_1=4+\theta$ if $d=2$.},
 \end{aligned}\right.
\end{equation}
where $0<\theta\ll 1$ depends only on
$d$ and $M_0$.
The above specific value of $p_*$
is determined by Lemma $\ref{lemma:16}$.
\footnotetext[11]{If the equations under consideration
are scalar valued,
let $p_* =3+\theta$ if $d\geq3$; and $p_*=4+\theta$ if $d=2$.}

\subsection{\centering 
Outline some basic ingredients}
\label{subsection:3.0}

Let $\chi_{*}$ be given as in Corollary $\ref{cor:1}$
with the $\frac{1}{L}$-Lipschitz continuity.
Then, in order to be consistent with the notation in Lemma $\ref{shen's lemma2}$,
we can define
\begin{equation}\label{f:3.6x}
F(x):=\Big(\dashint_{U_{*,\varepsilon}(x)}
 |\nabla u_\varepsilon|^2\Big)^{\frac{1}{2}};
 \qquad
 g(x):=\Big(\dashint_{U_{*,\varepsilon}(x)}
 |f|^2\Big)^{\frac{1}{2}},
\end{equation}
where $u_\varepsilon$ and $f$ are associated with certain equations.

How to utilize Shen’s real methods to address the Calder\'on-Zygmund estimates
is crucial for understanding the proof of
Theorem $\ref{thm:1}$.
We take the demonstration of the estimate $\eqref{A}$
as an example to explain it. Let $p_1\in(2,p_*)$ be arbitrarily fixed.

\noindent
\textbf{Step A.} \emph{Decompose the domain $\Omega$ and reduce
the desired estimate $\eqref{A}$ to some local estimates.}
Let $B_0\subset\mathbb{R}^d$ be a fixed ball with zero center
and the radius satisfying $r_{B_0}\sim R_0$. For any $x\in\mathbb{R}^d$,
we can denote by $B_0(x):=B_0+x$. Then, there exist two
integers $n_0, m_0$ with $n_0>m_0>0$ such that
\begin{equation}\label{f:decom-3}
\Omega\subset \Big(\cup_{i=1}^{m_0}4B_0(x_i)\Big)
  \cup \Big(\cup_{j=m_0+1}^{n_0}B_0(x_j)\Big) =: C_{\text{bdy}}
  \cup C_{\text{int}},
\end{equation}
where
$\{x_i\}_{i=1}^{m_0}\subset\partial\Omega$; and $\{x_j\}_{j=m_0+1}^{n_0}
\subset\Omega$ is such that $4B_0(x_j)\subset\Omega$ for each $j\in[m_0+1,n_0]$,
and we denote the boundary covering of $\Omega$ by $C_{\text{bdy}}$
and the interior covering of $\Omega$ by $C_{\text{int}}$ (see Figure \ref{figure}).
For each $i\in[1,n_0]$ and any $2\leq p<p_1$,
we manage to establish the following local estimates
\begin{equation}\label{f:3.1x}
\dashint_{\tilde{B}_0(x_i)\cap\Omega} F^p
\lesssim  \Big(\dashint_{4\tilde{B}_0(x_i)\cap\Omega} F^2\Big)^{\frac{p}{2}}
+ \dashint_{4\tilde{B}_0(x_i)\cap\Omega} g^p,
\end{equation}
where
$\tilde{B}_0(x_i) := 4B_0(x_i)$ if $i\in[1,m_0]$;
and $\tilde{B}_0(x_i) := B_0(x_i)$ if $i\in[m_0+1,n_0]$.
\emph{A covering argument} leads to
\begin{equation}\label{f:3.2x}
\int_{\Omega}F^p\lesssim \int_{\Omega}g^p.
\end{equation}
Then, by \emph{a duality argument}, the above estimate $\eqref{f:3.2x}$
also holds for $p_1'<p<2$, where $p_1'$ is the conjugate number of
$p_1$ (i.e., $1/p_1'+1/p_1=1$).
Thus, the desired estimate
$\eqref{A}$ has been reduced to $\eqref{f:3.1x}$.
\begin{figure}[htbp] %H为当前位置，!htb为忽略美学标准，htbp为浮动图形
\centering %图片居中
\includegraphics[width=1\textwidth]{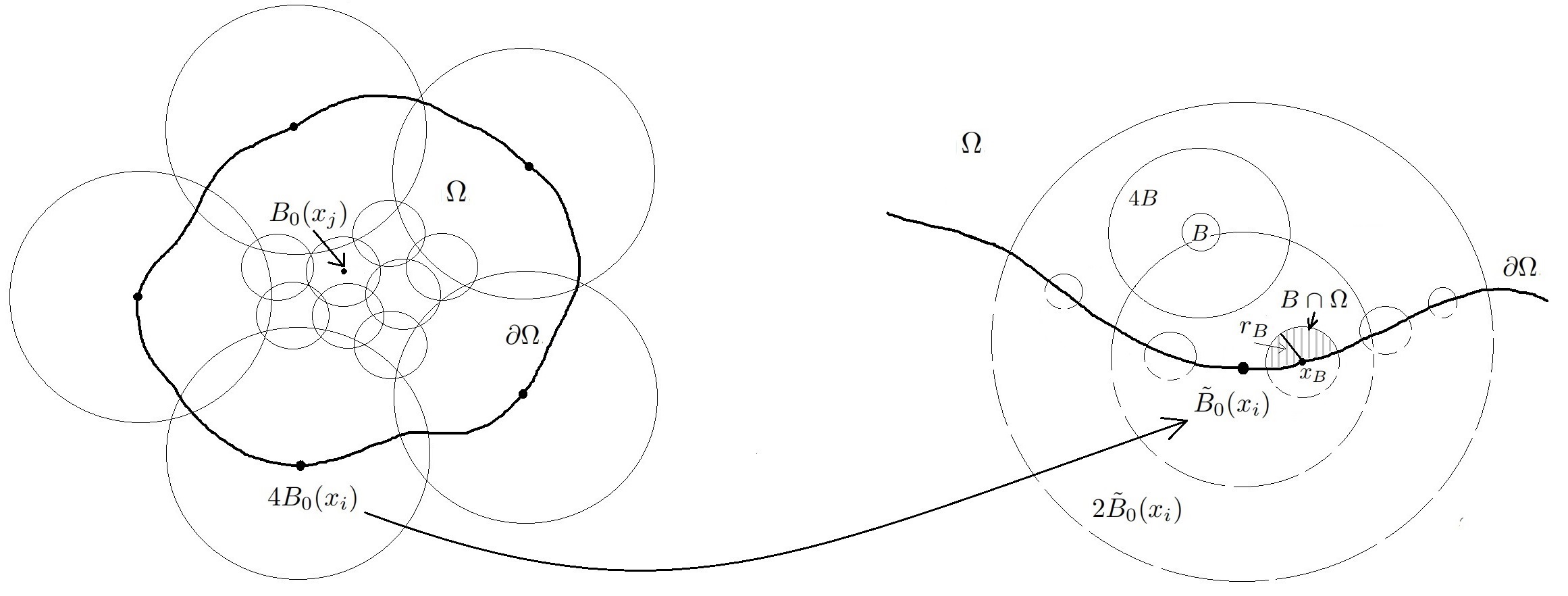} %插入图片，[]中设置图片大小，{}中是图片文件名
\caption{A covering of the domain $\Omega$} %最终文档中希望显示的图片标题
\label{figure} %用于文内引用的标签
\end{figure}

\noindent
\textbf{Step B.} \emph{Employ Shen's real methods
(presented in Lemma $\ref{shen's lemma2}$ as a weighted version,
simply taking $\omega =1$ therein for the case without weights) to obtain
the stated local estimates $\eqref{f:3.1x}$.} By the decomposition $\eqref{f:decom-3}$,
the demonstration of $\eqref{f:3.1x}$ should be divided into two parts
according to $C_{\text{bdy}}$ and $C_{\text{int}}$, respectively.
Here we merely introduce the arguments for the part of $C_{\text{bdy}}$,
since the part of $C_{\text{int}}$ can be similarly obtained.

In view of Lemma $\ref{shen's lemma2}$,
for the given $\tilde{B}_0(x_i)$ with $x_i\in\partial\Omega$, we concern any ball
$B$ with the properties: $|B|\leq c_1|\tilde{B}_0|$ and either
$x_B\in 2\tilde{B}_0(x_i)\cap\partial\Omega$ or $4B\subset 2\tilde{B}_0(x_i)\cap\Omega$,
where $c_1\in(0,1)$ (see Figure \ref{figure}).
Then, we need to
split $F$ according to $B$.
Recalling the decomposition $\eqref{f:decom-2}$,
it heuristically suggests that
\begin{equation*}
\dashint_{U_{*,\varepsilon}(x)}\eqref{f:decom-2}
\qquad \forall x\in 2B\cap\Omega,
\end{equation*}
and the right-hand side of $\eqref{f:decom-2}$
provides us a concrete form of $W_B$ and $V_B$
stated in Lemma $\ref{shen's lemma2}$ as follows:
\begin{equation}\label{f:3.5x}
W_B(x):= \Big(\dashint_{U_{*,\varepsilon}(x)}|W_\varepsilon|^2\Big)^{\frac{1}{2}}
\quad \text{and} \quad
V_B(x):= \Big(\dashint_{U_{*,\varepsilon}(x)}|V|^2\Big)^{\frac{1}{2}},
\end{equation}
by which we immediately have $F\leq W_B + V_B$ on $2B\cap\Omega$ in our hand.

However, the decomposition of $F$ is subtle because
geometry properties of integrals make
the case complicated under the multi-scale action
(see Lemma $\ref{lemma:9}$). Therefore, the decomposition of $F$
is based on the scale of this ball $B$, and
the following two cases should be considered\footnote{We
impose the notation $W_B$ (and $V_B$) to stress that
the decomposition of $F$ is indeed subjected to the ball $B$.}:
\begin{itemize}
  \item [(I).] $0<r_B<\frac{\varepsilon}{4}\chi_{*}(x_B/\varepsilon)$
  and, in such the case, we merely prefer
  $W_B = 0$ and $V_B = F$;
  \item [(II).] $r_B\geq \frac{\varepsilon}{4}\chi_{*}(x_B/\varepsilon)$
  and, in such the case, $W_B$ and $V_B$ are given by $\eqref{f:3.5x}$.
\end{itemize}

Then,
the remaining work is to verify
the estimates $\eqref{pri:3.3a}$
and $\eqref{pri:3.3b}$
satisfied by $W_B$ and $V_B$,
respectively, for each of the above two cases.
We will address them with details in the next subsection, and
finally obtain the desired estimate $\eqref{f:3.1x}$.

\begin{remark}\label{remark:3.1}
\emph{The form of $W_B$ and $V_B$ in $\eqref{f:3.5x}$
is still a heuristic statement.
Therefore,
unlike the symbols $F$ and $g$ introduced
in $\eqref{f:3.6x}$, $W_B$ and $V_B$
will own different forms according to the concrete context
in Subsections $\ref{subsection:3.1}, \ref{subsection:3.2}$.}
\end{remark}

The covering argument and duality one mentioned in \textbf{Step A}
will be silently used several times later in Subsections
$\ref{subsection:3.1},\ref{subsection:3.2}$ and,
for the reader's convenience, we provide some details below.

1.\emph{ Show the covering argument for $\eqref{f:3.2x}$.}
We start from the energy type estimate
\begin{equation}\label{f:3.8x}
\int_{\Omega} F^2
 \lesssim^{\eqref{pri:3.6-1*}} \int_{\Omega}|\nabla u_\varepsilon|^2
\lesssim_{d,\lambda,M_0}\int_{\Omega}|f|^2
\lesssim^{\eqref{pri:3.6-1*}} \int_{\Omega}|g|^2,
\end{equation}
where the second inequality follows from the energy estimate of
$u_\varepsilon$. We now acquire that
\begin{equation}\label{f:3.7x}
\begin{aligned}
\int_{\Omega} F^p \leq
\sum_{i=1}^{n_0}\int_{\tilde{B}_0(x_i)\cap\Omega} F^p
&\lesssim^{\eqref{f:3.1x}}|B_0|^{1-\frac{p}{2}}
\sum_{i=1}^{n_0}
\Big(\int_{4\tilde{B}_0(x_i)\cap\Omega} F^2\Big)^{\frac{p}{2}}
+ \sum_{i=1}^{n_0}\int_{4\tilde{B}_0(x_i)\cap\Omega} g^p\\
&\lesssim |B_0|^{1-\frac{p}{2}}
\Big(\int_{\Omega}F^2\Big)^{\frac{p}{2}}
+ \int_{\Omega}g^p
\lesssim^{\eqref{f:3.8x}}
|\Omega|^{1-\frac{p}{2}}\Big(\int_{\Omega}g^2\Big)^{\frac{p}{2}}
+ \int_{\Omega}g^p,
\end{aligned}
\end{equation}
and then the above estimate implies
\begin{equation*}
\dashint_{\Omega} F^p
\lesssim \Big(\dashint_{\Omega} g^2\Big)^{\frac{p}{2}}
+ \dashint_{\Omega}g^p
\lesssim \dashint_{\Omega}g^p,
\end{equation*}
which consequently leads to the stated estimate $\eqref{f:3.2x}$.

2.\emph{ Show the duality argument.} For $p_1'<p< 2$,
we have the conjugate number $p'\in(2,p_1)$.
Let $h$ be any test function
subject to the certain space (related to $p'$), and
$v_\varepsilon$
is associated with $h$ by the adjoint equations:
\begin{equation*}
(\text{D}_\varepsilon^*)~\left\{\begin{aligned}
-\nabla\cdot a^{*\varepsilon}\nabla v_\varepsilon
&= \nabla\cdot h
&\quad&\text{in}~~\Omega;\\
v_\varepsilon &= 0
&\quad&\text{on}~~\partial\Omega,
\end{aligned}\right.
\quad\text{or}\quad
(\text{N}_\varepsilon^{*})~\left\{\begin{aligned}
-\nabla\cdot a^{*\varepsilon}\nabla v_\varepsilon
&= \nabla\cdot h
&\quad&\text{in}~~\Omega;\\
n\cdot a^{*\varepsilon} \nabla v_\varepsilon &= -n\cdot h
&\quad&\text{on}~~\partial\Omega.
\end{aligned}\right.
\end{equation*}
Then, integrating by parts leads to the following identity:
\begin{equation}\label{f:3.4x}
\int_{\Omega}\nabla u_\varepsilon\cdot h
=\int_{\Omega}\nabla v_\varepsilon\cdot f.
\end{equation}
Since we have the estimate $\eqref{f:3.2x}$ for $v_\varepsilon$ in hand,
it follows from Lemma $\ref{lemma:9}$ and H\"older's inequality that
\begin{equation}\label{f:3.3x}
\begin{aligned}
\int_{\Omega}\nabla v_\varepsilon\cdot f
&\lesssim_{d,M_0}^{\eqref{pri:3.6-1*}} \int_{\Omega}
\Big(\dashint_{U_{*,\varepsilon}(x)}|\nabla v_\varepsilon\cdot f|\Big) dx\\
&\lesssim \bigg(
\int_{\Omega}
\Big(\dashint_{U_{*,\varepsilon}(x)}
|\nabla v_\varepsilon|^2\Big)^{\frac{p'}{2}}dx\bigg)^{\frac{1}{p'}}
\Big(\int_{\Omega} g^{p}\Big)^{\frac{1}{p}}
\lesssim
\bigg(
\int_{\Omega}
\Big(\dashint_{U_{*,\varepsilon}(x)}
|h|^2\Big)^{\frac{p'}{2}}dx\bigg)^{\frac{1}{p'}}
\Big(\int_{\Omega} g^{p}\Big)^{\frac{1}{p}}.
\end{aligned}
\end{equation}
Plugging the estimate $\eqref{f:3.3x}$ back into $\eqref{f:3.4x}$,
we can infer from the definition of the dual space that
the estimate $\eqref{f:3.2x}$ also holds for $p_1'<p< 2$.
This together with $\eqref{f:3.2x}$ subsequently implies the
desired estimate $\eqref{A}$.

\subsection{\centering Quenched estimates}\label{subsection:3.1}

\noindent
In order to overcome the difficulties caused by the non-smoothness of
the boundary, we first make use of the interior quenched $W^{1,p}$ estimates of the correctors, whose proof essentially belongs to F. Otto (see Lemma $\ref{lemma:7}$). After that, we establish the reverse H\"older's  inequality in Lemma $\ref{lemma:14}$, which will play an important role not only in the quenched estimates but also in the annealed ones in our demonstration.
We show the proof of Proposition $\ref{P:10}$ with details, the novelty of which can be easily recognized compared to the methods directly appealing to large-scale Lipschitz estimates.

\begin{proposition}\label{P:10}
Let $\Omega\subset\mathbb{R}^d$ be a bounded Lipschitz domain with
$d\geq 2$ and $\varepsilon\in(0,1]$.
Suppose that  $\langle\cdot\rangle$ is stationary,
satisfying the spectral gap condition $\eqref{a:2}$,
and the (admissible) coefficient additionally satisfies $\eqref{a:3}$ and
the symmetry condition $a=a^*$. Let $u_\varepsilon$ and $f$ be associated with the equations:
\begin{equation}\label{pde:2-D}
\left\{\begin{aligned}
\mathcal{L}_\varepsilon (u_\varepsilon)
&= \nabla\cdot f
&\quad&\text{in}~~\Omega;\\
u_\varepsilon &= 0
&\quad&\text{on}~~\partial\Omega.
\end{aligned}\right.
\end{equation}
Then, there exists a stationary random field $\chi_*$
given as in Corollary $\ref{cor:1}$ with the
$\frac{1}{L}$-Lipschitz continuity (see Remark $\ref{remark:4}$) such that
the estimate $\eqref{A}$
holds for any $p>1$ satisfying $|\frac{1}{p}-\frac{1}{2}|\leq \frac{1}{2d}+\theta$
with $0<\theta\ll 1$ depending only on $d$ and $M_0$. Also,
for some $\omega\in A_1$ we have
\begin{equation}\label{pri:8-1}
\int_{\Omega}\Big(\dashint_{U_{*,\varepsilon}(x)}
|\nabla u_\varepsilon|^2 \Big)
[\omega(x)]^{\pm1}dx
\lesssim_{\lambda,d,M_0,p}
\int_{\Omega}\Big(\dashint_{U_{*,\varepsilon}(x)}|f|^2 \Big)
[\omega(x)]^{\pm1}dx,
\end{equation}
%where we recall the notations $B_{*,\varepsilon}(x):=
%B_{\varepsilon\chi_{*}(x/\varepsilon)}(x)$
%and $D_{*,\varepsilon}(x):=\Omega\cap B_{*,\varepsilon}(x)$.
Moreover, if $\Omega$ is a regular SKT (or $C^1$)
domain (in such the case,
the symmetry condition $a=a^*$ is not needed),
then the estimate $\eqref{A}$ holds for any $1<p<\infty$;
For any $\omega\in A_p$ with $1<p<\infty$,
we have the following weighted estimate
%\begin{equation}
%\int_{\Omega}|\nabla \bar{u}|^p\omega
%\lesssim \int_{\Omega}|f|^p
%\omega.
%\end{equation}
%for the effective solution of $-\nabla\cdot\bar{a}\nabla \bar{u}=\nabla\cdot f$ in $\Omega$ with $\bar{u}=0$ on $\partial\Omega$,
%we then have
\begin{equation}\label{f:12}
\int_{\Omega}\Big(
\dashint_{U_{*,\varepsilon}(x)}|\nabla u_\varepsilon|^2
\Big)^{\frac{p}{2}}\omega(x)dx
\lesssim_{\lambda,d,M_0,p,[\omega]_{A_p}}
\int_{\Omega}\Big(\dashint_{U_{*,\varepsilon}(x)}|f|^2\Big)^{\frac{p}{2}}
\omega(x)dx.
\end{equation}
\end{proposition}

\begin{remark}
\emph{The notation $[\omega(x)]^{\pm1}$ appeared in $\eqref{pri:8-1}$
represents that the estimate $\eqref{pri:8-1}$ holds for
the weight function $\omega$ and $\omega^{-1}$, respectively.
The same results are also valid for the elliptic systems with Neumann boundary conditions if we
correspondingly modify Lemma $\ref{lemma:14}$ below. Besides, in terms of the scalar case of equation $\eqref{pde:2-D}$,
the range of $p$ satisfies $(4/3)-\theta<p<4+\theta$ if $d=2$; and $(3/2)-\theta<p<3+\theta$ if $d\geq 3$, which is due to
Lemma $\ref{lemma:14}$.}
\end{remark}

\begin{lemma}[interior $W^{1,p}$ estimates]\label{lemma:7}
Let $2\leq q<\infty$ and $\varepsilon\in(0,1]$.
Suppose that $\langle\cdot\rangle$ is stationary and
satisfies $\eqref{a:2}$ and $\eqref{a:3}$. Let $u_\varepsilon$ and $f$ be
related by
$-\nabla \cdot a^\varepsilon\nabla u_\varepsilon = \nabla \cdot f$
in $\mathbb{R}^d$. Then, there exists a stationary
random fields $\chi_*$ given as in Lemma $\ref{lemma:1}$
with the $\frac{1}{L}$-Lipschitz continuity such that
\begin{equation}\label{pri:3.7}
 \dashint_{B_R}
 \Big(\dashint_{B_{*,\varepsilon}(x)}
 |\nabla u_\varepsilon|^2\Big)^{\frac{q}{2}}dx
 \lesssim_{\lambda,d,q}
 \bigg(\dashint_{B_{4R}}
 \Big(\dashint_{B_{*,\varepsilon}(x)}|\nabla u_\varepsilon|^2
 \Big) dx\bigg)^{\frac{q}{2}}
  +
 \dashint_{B_{4R}}
 \Big(\dashint_{B_{*,\varepsilon}(x)}|f|^2\Big)^{\frac{q}{2}}
 dx
\end{equation}
holds for any $R>0$.
Moreover, we additionally assume that $f$ is $L^2$-integrable
on $\mathbb{R}^d$.
Then, for $1<p<\infty$, one can rebuild the global estimate
\begin{equation}\label{pri:3.7-1}
\int_{\mathbb{R}^d}
\Big(\dashint_{B_{*,\varepsilon}(x)}
|\nabla u_\varepsilon|^2\Big)^{\frac{p}{2}}dx
\lesssim_{\lambda,d,p}
\int_{\mathbb{R}^d}
\Big(\dashint_{B_{*,\varepsilon}(x)}|f|^2\Big)^{\frac{p}{2}}dx.
\end{equation}
\end{lemma}

\begin{proof}
Based upon large-scale Lipschitz estimates, the original proof is due to
\cite{Armstrong-Daniel16} and \cite{Gloria-Neukamm-Otto20},
while the present work starts from the newly introduced minimal
radius $\chi_*$ in Lemma $\ref{lemma:1}$.
The main idea of the proof is taken from Felix Otto's
lectures in Toulouse 2019, and we modified it for our purpose.
%By rescaling arguments, it suffices to consider the case $R=1$.
The proof is divided into four steps.
Let $B_R$ be fixed with $R>0$. For any ball $B$ with $B\subset 2B_R$
and $r_B\leq (1/100)R$,
according to Remark $\ref{remark:3.1}$, the first task is to find the concrete
forms of $W_B$ and $V_B$ if
$r_B\geq \frac{\varepsilon}{4}\chi_{*}(x_B/\varepsilon)$.

% and set
%$\textcolor[rgb]{1.00,0.00,0.00}{V_B} = U$ and $\textcolor[rgb]{1.00,0.00,0.00}{W_B} = 0$
%for the case of $0<r_B<\frac{\varepsilon}{4}\chi_{*}(x_B/\varepsilon)$.
\noindent
\textbf{Step 1.} Decompose the field $\nabla u_\varepsilon$ and
make a reduction for the stated estimate $\eqref{pri:3.7}$.
Let $w_\varepsilon\in H_0^1(30B)$ and
$v_\varepsilon:=u_\varepsilon-w_\varepsilon$ satisfy the following equations:
\begin{equation}\label{pde:19}
 -\nabla\cdot a^\varepsilon\nabla w_\varepsilon = \nabla \cdot fI_{30B},
 \quad\text{and}\quad
 \nabla\cdot a^\varepsilon\nabla v_\varepsilon
 =0 \quad\text{in}\quad 30B.
\end{equation}
Then we consider the approximating problem:
$\nabla \cdot \bar{a}\nabla\bar{v} = 0$ in $9B$ with
$\bar{v} = v_\varepsilon$ on $\partial (9B)$. Thus, we have
\begin{equation}\label{}
\begin{aligned}
 \nabla u_\varepsilon &= \nabla w_\varepsilon + \nabla v_\varepsilon\\
 & = \underbrace{\nabla w_\varepsilon + \nabla v_\varepsilon -
 (e_i+(\nabla\phi)^\varepsilon)\partial_i\bar{v}}_{\text{approximating~part}}
 + (e_i+(\nabla\phi)^\varepsilon)\partial_i\bar{v}
\quad\text{on}~9B.
\end{aligned}
\end{equation}
Thus, the concrete form of
$W_B$ and $V_B$ in the case of (II) is introduced as follows:
\begin{equation}\label{f:3.7*}
\begin{aligned}
&V_{B}(x):=\Big(\dashint_{B_{*,\varepsilon}(x)}
|(e_i+(\nabla\phi_i)^\varepsilon)
\partial_i\bar{v}|^2\Big)^{1/2};
\qquad
W_{B}(x):=\Big(\dashint_{B_{*,\varepsilon}(x)}
|\nabla u_\varepsilon-(e_i+(\nabla\phi_i)^\varepsilon)
\partial_i\bar{v}|^2\Big)^{1/2};\\
&W^{(1)}_{B}(x):=\Big(\dashint_{B_{*,\varepsilon}(x)}
|\nabla v_\varepsilon - (e_i+(\nabla\phi_i)^\varepsilon)
\partial_i\bar{v}|^2\Big)^{1/2};\qquad
W^{(2)}_{B}(x)
:=\Big(\dashint_{B_{*,\varepsilon}(x)}|\nabla w_\varepsilon|^2\Big)^{1/2}.
\end{aligned}
\end{equation}
For $\Omega = \mathbb{R}^d$, we need to
rewrite $\eqref{f:3.6x}$ as
\begin{equation}%\label{f:3.7}
F(x):=\Big(\dashint_{B_{*,\varepsilon}(x)}|\nabla u_\varepsilon|^2\Big)^{1/2};
\qquad
g(x):=\Big(\dashint_{B_{*,\varepsilon}(x)}|f|^2\Big)^{1/2}.
\end{equation}
%where %$B_{*,\varepsilon}(x):=B_{\varepsilon\chi_{*}(x/\varepsilon)}(x)$.

Now, combining the two cases (I) and (II),
one can immediately observe the following relationships:
\begin{equation}\label{f:3.8}
F\leq W_{B}+V_{B};\quad \text{and}
\quad
W_{B}\leq W_{B}^{(1)} + W_{B}^{(2)}
\quad\text{on}~2B.
\end{equation}
By Lemma $\ref{shen's lemma2}$ (we simply take $\omega=1$ and $\Omega=\mathbb{R}^d$
therein),
the desired estimate $\eqref{pri:3.7}$ therefore
follows from the next two estimates:
\begin{equation}\label{f:3.5}
\sup_{2B}
V_B \lesssim
\Big(\dashint_{60B}F^2
+g^2\Big)^{\frac{1}{2}};
\end{equation}
and (for \textcolor[rgb]{0.00,0.00,1.00}{$0<\nu\ll 1$}, there holds)
\begin{equation}\label{f:3.6}
\dashint_{2B}
W_B^2
\lesssim \dashint_{60B}
(g^2+\nu^2
F^2).
\end{equation}

%Let $x_B$ and $r_B$ represent the center and radius of $B$, respectively.
As mentioned in Subsection $\ref{subsection:3.0}$,
the proof on the above estimates should be divided into
two cases: (I) $0<r_B<\frac{\varepsilon}{4}\chi_{*}(x_B/\varepsilon)$;
(II) $r_B\geq \frac{\varepsilon}{4}\chi_{*}(x_B/\varepsilon)$.
Since $W_B = 0$ and $V_B = F$ are taken for the case (I),
the estimate $\eqref{f:3.6}$ is trivial while the estimate
$\eqref{f:3.5}$ is reduced to show
\begin{equation*}
\sup_{2B}F
\lesssim
\Big(\dashint_{20B}F^2\Big)^{\frac{1}{2}}.
\end{equation*}
This relies on
the $\frac{1}{L}$-Lipschitz continuity of $\chi_{*}$ and
the estimate $\eqref{pri:3.5-1}$,
and related details can be found in Step 2 in
the proof of Lemma $\ref{lemma:14}$.
Next, we put more effort
into verifying $\eqref{f:3.5}$ and $\eqref{f:3.6}$
for the case (II).

\noindent
\textbf{Step 2.} Show the estimate $\eqref{f:3.5}$
in the case of (II).
For any $x\in 2B$, recalling the notation
$V_B(x)$ in $\eqref{f:3.7*}$,
and noting that $B_{*,\varepsilon}(x)\subset 6B$ for
all $x\in 2B$ (due to the $\frac{1}{L}$-Lipschitz continuity of
$\chi_{*}$), we have
\begin{equation*}
\begin{aligned}
  \Big(\dashint_{B_{*,\varepsilon}(x)} |(e_i+(\nabla\phi_i)^\varepsilon)\partial_i\bar{v}|^2\Big)^{\frac{1}{2}}
  &\leq  \Big(\dashint_{B_{*,\varepsilon}(x)}
  |e+(\nabla\phi)^\varepsilon|^2\Big)^{\frac{1}{2}}
  \sup_{6B}|\nabla\bar{v}| \\
  &\lesssim \bigg\{\Big(\dashint_{B_{\chi_*(x/\varepsilon)}(x/\varepsilon)}
  |\nabla\phi|^2\Big)^{\frac{1}{2}}+1\bigg\}
  \Big(\dashint_{8B}|\nabla\bar{v}|^2\Big)^{1/2}.
\end{aligned}
\end{equation*}
By using Caccioppoli's inequality and energy estimates (see e.g.
\cite{Giaquinta-Martinazzi12}),
as well as the definition of $\chi_*$ in Lemma $\ref{lemma:1}$, we consequently derive that
\begin{equation}\label{f:2.30}
\begin{aligned}
\sup_{2B}  V_B
&\lesssim \sup_{\varepsilon z\in B}
\bigg\{\frac{1}{\chi_*(z)}\Big(\dashint_{B_{2\chi_*(z)}(z)}
|\phi-(\phi)_{2\chi_{*}(z)}|^2\Big)^{\frac{1}{2}}+1\bigg\}
  \Big(\dashint_{9B}|\nabla v_\varepsilon|^2\Big)^{\frac{1}{2}}\\
&\lesssim
  \Big(\dashint_{9B}|\nabla u_\varepsilon|^2
  \Big)^{\frac{1}{2}} +
  \Big(\dashint_{30B}|f|^2\Big)^{\frac{1}{2}}
\lesssim^{\eqref{pri:3.6-1}}
\Big(\dashint_{60B}
F^2 + g^2\Big)^{\frac{1}{2}}.
\end{aligned}
\end{equation}

\noindent
\textbf{Step 3.} Show the estimate $\eqref{f:3.6}$
in the case of (II). In view of
the second inequality of $\eqref{f:3.8}$ and Lemmas $\ref{lemma:1}$ and
$\ref{lemma:9}$,
we have
\begin{equation*}
\begin{aligned}
\dashint_{2B} W_B^2
&\lesssim^{\eqref{pri:3.6-1}}
\dashint_{7B}|\nabla v_\varepsilon
-(e_i+(\nabla\phi_i)^\varepsilon)\partial_i \bar{v}|^2
 + \dashint_{7B}|\nabla w_\varepsilon|^2\\
&\lesssim^{\eqref{pri:2.0}}
\nu^2
\dashint_{28B}|\nabla v_\varepsilon|^2
+ \dashint_{30B}|f|^2
\lesssim^{\eqref{pri:3.6-1}} \dashint_{60B}
(g^2 +
\nu^2 F^2),
\end{aligned}
\end{equation*}
where we also employ the energy estimate for the
first equation of $\eqref{pde:19}$ in the last two inequalities.

\noindent
\textbf{Step 4.}
\emph{Show the estimate $\eqref{pri:3.7-1}$}.
In the case of $p=2$, we can derive
the desired estimate $\eqref{pri:3.7-1}$ by
a similar argument given by $\eqref{f:3.8x}$.
Then, for the case $p>2$, it follows from
$\eqref{pri:3.7}$ and the previous result ($p=2$)
that
\begin{equation*}
\int_{B_R}
 \Big(\dashint_{B_{*,\varepsilon}(x)}
 |\nabla u_\varepsilon|^2\Big)^{\frac{p}{2}}dx
 \lesssim_{\lambda,d,p}
 R^{d(1-\frac{p}{2})}
 \bigg(\int_{\mathbb{R}^d}
 |f|^2\bigg)^{\frac{p}{2}}
  +
 \int_{B_{4R}}
 \Big(\dashint_{B_{*,\varepsilon}(x)}|f|^2\Big)^{\frac{p}{2}}
 dx.
\end{equation*}
By letting $R\to\infty$, the stated estimate $\eqref{pri:3.7-1}$
holds for any $p\in(2,\infty)$. Finally, the duality argument
as shown in Subsection $\ref{subsection:3.0}$ leads to
the desired estimate $\eqref{pri:3.7-1}$ for any $p\in(1,\infty)$.
This ends the whole proof.
\end{proof}

\begin{remark}
\emph{The estimate $\eqref{pri:3.7}$ is not only
translation invariant but also scaling invariant. Thus, by the
covering argument similar to that shown in Subsection $\ref{subsection:3.0}$,
for any $\alpha>1$,
one can derive that
\begin{equation}\label{pri:3.7-2}
\dashint_{B_R}
 \Big(\dashint_{B_{*,\varepsilon}(x)}
 |\nabla u_\varepsilon|^2\Big)^{\frac{q}{2}}dx
 \lesssim_{\lambda,d,q,\alpha}
 \bigg(\dashint_{B_{\alpha R}}
 \Big(\dashint_{B_{*,\varepsilon}(x)}|\nabla u_\varepsilon|^2
 \Big)dx\bigg)^{\frac{q}{2}}
  +
 \dashint_{B_{\alpha R}}
 \Big(\dashint_{B_{*,\varepsilon}(x)}|f|^2\Big)^{\frac{q}{2}}
 dx.
\end{equation}
Moreover, if $f=0$, it follows
from the estimates $\eqref{pri:3.7-2}$,
$\eqref{pri:10-a}$, and
Lemma $\ref{lemma:9}$ (merely taking $\Omega=\mathbb{R}^d$) that
\begin{equation}\label{pri:3.7-3}
\dashint_{B_R}
 \Big(\dashint_{B_{*,\varepsilon}(x)}
 |\nabla u_\varepsilon|^2\Big)^{\frac{q}{2}}dx
 \lesssim_{\lambda,d,q,p_0,\alpha}
 \bigg(\dashint_{B_{\alpha R}}
 \Big(\dashint_{B_{*,\varepsilon}(x)}|\nabla u_\varepsilon|^2\Big)^{\frac{p_0}{2}}
 dx\bigg)^{\frac{q}{p_0}}
\end{equation}
holds for any $0<p_0<q<\infty$.
The detail is left to the reader.}
\end{remark}

%\footnote{If the equations $\eqref{pde:1}$ are scalar,
%let $p_1 =3+\theta$ as $d\geq3$; and $p_1=4+\theta$ as $d=2$.}.

\begin{lemma}[quenched reverse H\"older's inequality]\label{lemma:14}
Let $p_*$ be given as in $\eqref{upindex-p}$ and $0\in\partial\Omega$.
The ensemble $\langle\cdot\rangle$
is assumed as in Proposition $\ref{P:10}$.
Let $0<\varepsilon\leq 1$ and $R>0$.
Suppose that $u_\varepsilon$ is the solution of
\begin{equation}\label{pde:1}
\left\{\begin{aligned}
\nabla\cdot a^\varepsilon\nabla u_\varepsilon &= 0
&\quad&\text{in}~~D_{2R};\\
u_\varepsilon &= 0
&\quad&\text{on}~~\Delta_{2R}.
\end{aligned}\right.
\end{equation}
Then, there exists a stationary random field $\chi_*$
given as in Corollary $\ref{cor:1}$ with the
$\frac{1}{L}$-Lipschitz continuity such that,
for $0<p_0<q<p_*$, there holds
\begin{equation}\label{pri:3.8}
\bigg(\dashint_{D_{R/2}}
\Big(\dashint_{U_{*,\varepsilon}(x)}|\nabla u_\varepsilon|^2
\Big)^{\frac{q}{2}}dx\bigg)^{\frac{1}{q}}
\lesssim_{\lambda,d,p_0,q,M_0}
\bigg(\dashint_{D_{2R}}
\Big(\dashint_{U_{*,\varepsilon}(x)}|\nabla u_\varepsilon|^2
\Big)^{\frac{p_0}{2}}dx\bigg)^{\frac{1}{p_0}}.
\end{equation}
\end{lemma}

\begin{proof}
The main idea is similar to that given in Lemma $\ref{lemma:7}$.
Before giving the concrete proof, we mention that in the case of $p_0>2$,
it suffices to show
\begin{equation*}
  \Big(\dashint_{\textcolor[rgb]{1.00,0.00,0.00}{D_{R/2}}}
   F^q\bigg)^{\frac{1}{q}}
   \lesssim \bigg(\dashint_{\textcolor[rgb]{1.00,0.00,0.00}{D_{2R}}}
   F^2\Big)^{\frac{1}{2}},
\end{equation*}
where we use the notation $F$ stated in $\eqref{f:3.6x}$, and then
the desired estimate directly follows from the H\"older inequality.
Thus, we merely focus on the demonstration in the case of $p_0\in(0,2]$,
and the whole proof will be divided into four steps.
Let $B$ be any ball that owns the properties:
$r_B\leq (1/100)R$ and either
$x_B\in \partial\Omega\cap B_{R}$ or $4B\subset D_{R}$.
According to Remark $\ref{remark:3.1}$,
we need to find the concrete
forms of $W_B$ and $V_B$ if
$r_B\geq \frac{\varepsilon}{4}\chi_{*}(x_B/\varepsilon)$.

\noindent
\textbf{Step 1.}
Decompose the field $\nabla u_\varepsilon$ and
make a reduction for the stated estimate $\eqref{pri:3.8}$.
In terms of the decomposition,
the case $x_B\in\partial\Omega\cap B_{R}$ is more important
than the case of $4B\subset D_{R}$ here\footnote{
To see the case of $4B\subset D_{R}\subset\Omega$,
we can take $V_B = F$ and $W_B = 0$.
On account of $\eqref{pri:3.7-3}$, it is not hard to verify
the condition $\eqref{f:3.31}$, while $\eqref{f:3.32}$ is trivial in such the case.}.
Thus, we merely focus on the case $x_B\in\partial\Omega\cap B_R$,
and set $D:=B\cap\Omega$.
We now consider the approximating equation:
$\nabla \cdot \bar{a}\nabla\bar{u} = 0$ in $9D$ with
$\bar{u} = u_\varepsilon$ on $\partial (9D)$,
and introduce
the following notation:
\begin{equation*}
\begin{aligned}
&V_{B}(x):=
\Big(\dashint_{U_{*,\varepsilon}(x)}
|(e_i+(\nabla\phi_i)^\varepsilon)
\varphi_i|^2\Big)^{\frac{1}{2}};\\
&W_{B}(x):=
\Big(\dashint_{U_{*,\varepsilon}(x)}
|\nabla u_\varepsilon-(e_i+(\nabla\phi_i)^\varepsilon)
\varphi_i|^2\Big)^{\frac{1}{2}},
\end{aligned}
\end{equation*}
where $\varphi_i:=S_{*,\varepsilon}(\eta_{\varepsilon}
\partial_i\bar{u})$ (see Lemma $\ref{lemma:2}$ for
the definition of $S_{*,\varepsilon}$), and we mention that the random field $\chi_*$ is taken from Corollary $\ref{cor:1}$.
Then, combining the two cases (I)
and (II) presented in Subsection $\ref{subsection:3.0}$, we observe that
$F\leq W_{B}+V_{B}$ on $2B\cap\Omega$.
Let $p_1\in(2,p_*)$ be arbitrarily fixed.
By Lemma $\ref{shen's lemma2}$ (simply take $\omega=1$ therein), the desired estimate $\eqref{pri:3.8}$ immediately follows from the next two estimates:
\begin{equation}\label{f:3.31}
\Big(\dashint_{2B\cap\Omega}
V_B^{p_1}\Big)^{\frac{1}{p_1}}
\lesssim \Big(\dashint_{20B\cap\Omega}F^{p_0}\Big)^{\frac{1}{p_0}};
\end{equation}
and (for $0<\nu\ll 1$, there holds)
\begin{equation}\label{f:3.32}
\dashint_{2B\cap\Omega}
W_B^{p_0}
\lesssim \nu^{p_0}
\dashint_{58B\cap\Omega}
F^{p_0}.
\end{equation}
As mentioned in Subsection $\ref{subsection:3.0}$,
the proof of the above estimates will be divided into two cases:
(I). $0<r_B<\frac{\varepsilon}{4}\chi_{*}(x_B/\varepsilon)$;
(II). $r_B\geq \frac{\varepsilon}{4}\chi_{*}(x_B/\varepsilon)$.
We consider the first case in Step 2, and the second case in Steps 3 and 4
(note that it suffices to address the situation of $x_B\in\partial\Omega$
for the second case).

\medskip
\noindent
\textbf{Step 2.}
Establish the
estimates $\eqref{f:3.31}$ and $\eqref{f:3.32}$
in the case of (I). In such the case, we merely choose
$\textcolor[rgb]{1.00,0.00,0.00}{V_B} = F$
and $\textcolor[rgb]{1.00,0.00,0.00}{W_B} = 0$, and
the estimate $\eqref{f:3.32}$ is trivial while the
estimate $\eqref{f:3.31}$ is reduced to show
\begin{equation}\label{f:3.31*}
\Big(\dashint_{2D}
F^{p_1}\Big)^{\frac{1}{p_1}}
\lesssim \Big(\dashint_{15D}F^{p_0}\Big)^{\frac{1}{p_0}}.
\end{equation}
To see the estimate $\eqref{f:3.31*}$,
on account of the $\frac{1}{L}$-Lipschitz continuity of
$\chi_{*}$, for any
$x_0\in 3D$, we have $0< \vartheta
r_B<\frac{\varepsilon}{4}\chi_{*}(x_0/\varepsilon)$ with
$\vartheta:=(1-\frac{3}{4L})\in[7/8,1)$ by recalling the definition of
$L$ in Remark $\ref{remark:3}$.
Let $\tilde{B}$ be the concentric ball with
the radius $r_{\tilde{B}}=\vartheta r_B$
. We can infer that
$2B\subset(5/2)\tilde{B}\subset 3B$.
Then, for any $x_0\in (5/2)\tilde{B}$,
appealing to Lemma $\ref{lemma:9}$ we can derive that
\begin{equation*}
 \Big(\dashint_{U_{*,\varepsilon}(x_0)}
 |\nabla u_\varepsilon|^2\Big)^{\frac{p_0}{2}}
 \lesssim^{\eqref{pri:3.5-1}}
 \dashint_{15D}\Big(
 \dashint_{U_{*,\varepsilon}(x)}|\nabla u_\varepsilon|^2
 \Big)^{\frac{p_0}{2}}dx,
\end{equation*}
which further implies that for any $p_1\geq p_0$,
\begin{equation*}
\Big(\dashint_{U_{*,\varepsilon}(x_0)}
|\nabla u_\varepsilon|^2\Big)^{\frac{p_1}{2}}
\lesssim \bigg(\dashint_{15D}\Big(
 \dashint_{U_{*,\varepsilon}(x)}
 |\nabla u_\varepsilon|^2
 \Big)^{\frac{p_0}{2}}dx\bigg)^{\frac{p_1}{p_0}}.
\end{equation*}
Integrating both sides above with respective to
$x_0\in 2D$, we arrive at
\begin{equation}\label{f:3.33}
\bigg(\dashint_{\textcolor[rgb]{1.00,0.00,0.00}{2D}}
\Big(\dashint_{\textcolor[rgb]{1.00,0.00,0.00}{U_{*,\varepsilon}(x)}}
|\nabla u_\varepsilon|^2\Big)^{\frac{p_1}{2}}
dx\bigg)^{\frac{1}{p_1}}
\lesssim \bigg(\dashint_{\textcolor[rgb]{1.00,0.00,0.00}{15D}}\Big(
 \dashint_{\textcolor[rgb]{1.00,0.00,0.00}{U_{*,\varepsilon}(x)}}|\nabla u_\varepsilon|^2
 \Big)^{\frac{p_0}{2}}dx\bigg)^{\frac{1}{p_0}},
\end{equation}
which indicates the stated estimate $\eqref{f:3.31*}$ and therefore
has verified $\eqref{f:3.31}$.

\medskip
\noindent
\textbf{Step 3.} Show the estimate
$\eqref{f:3.31}$ in the case of (II).
According to the definition of $S_{*,\varepsilon}$ in $\eqref{def1}$,
for any $x\in 2D$,
it follows from H\"older's inequality and Fubini's theorem that
\begin{equation*}
\begin{aligned}
\dashint_{U_{*,\varepsilon}(x)} \big|(e_i+(\nabla\phi_i)^\varepsilon)
S_{*,\varepsilon}(\eta_\varepsilon\partial_i\bar{u})\big|^2
&\lesssim \dashint_{U_{*,\varepsilon}(x)} |(e_i+(\nabla\phi_i)^\varepsilon)(z)|^2
\int_{\mathbb{R}^d}\zeta_{\varepsilon\chi_{*}(z/\varepsilon)}(z-y)
|\eta_\varepsilon\partial_i\bar{u}(y)|^2 dy dz\\
&\lesssim \dashint_{2U_{*,\varepsilon}(x)} |\partial_i\bar{u}(y)|^2
\int_{\mathbb{R}^d}\zeta_{\varepsilon\chi_{*}(y/\varepsilon)}(y-z)
|(e_i+(\nabla\phi_i)^\varepsilon)(z)|^2 dzdy.
\end{aligned}
\end{equation*}
In the second inequality above, we also need to analyze the
compact support under the convolution, and it relies
on two facts (1): $\chi_*(y/\varepsilon)\sim \chi_*(z/\varepsilon)$
provided $|z-y|\leq \varepsilon\chi_*(z/\varepsilon)$;
(2): For any $z\in U_{*,\varepsilon}(x)$ and
$y\in\mathbb{R}^d$ satisfying $|y-z|\leq \varepsilon\chi_*(z/\varepsilon)$
we can infer that $|y-x|\leq 2\varepsilon\chi_*(x/\varepsilon)$. Then we
continue to have
\begin{equation*}
\begin{aligned}
\dashint_{U_{*,\varepsilon}(x)} \big|(e_i+(\nabla\phi_i)^\varepsilon)
S_{*,\varepsilon}(\eta_\varepsilon\partial_i\bar{u})\big|^2
&\lesssim \dashint_{2U_{*,\varepsilon}(x)} |\partial_i\bar{u}(y)|^2
\dashint_{B_{*,\varepsilon}(y)}
|(e_i+(\nabla\phi_i)^\varepsilon)|^2dy.
\end{aligned}
\end{equation*}
Suppose that $\alpha'$ is the conjugate index of $\alpha$ satisfying $0<\alpha'-1\ll 1$.
Thus, by H\"older's inequality we have
\begin{equation}\label{f:3.34}
\begin{aligned}
\Big(\dashint_{2D}
 V_B^{p_1}\Big)^{\frac{1}{p_1}}
\lesssim  \Bigg(\dashint_{2D}
\Big(\dashint_{2U_{*,\varepsilon}(x)}
|\partial_i\bar{u}|^{2\alpha'}\Big)^{\frac{p_1}{2\alpha'}}
\Big(\dashint_{2U_{*,\varepsilon}(x)}
\Big(\dashint_{B_{*,\varepsilon}(y)}
|(\nabla\phi_{i})^\varepsilon+e_i|^{2}\Big)^{\alpha}
dy\Big)^{\frac{p_1}{2\alpha}}dx
\Bigg)^{\frac{1}{p_1}}.
\end{aligned}
\end{equation}
In terms of the right-hand side of $\eqref{f:3.34}$, we claim that
\begin{subequations}
\begin{align}
&\sup_{x\in 2D\atop i=1,\cdots,d}
\Big(\dashint_{2B_{*,\varepsilon}(x)}
\Big(\dashint_{B_{*,\varepsilon}(y)}
|(\nabla\phi_i)^\varepsilon+e_i|^{2}\Big)^{\alpha}
dy\Big)^{\frac{p_1}{2\alpha}}
\lesssim 1; \label{f:3.24}\\
&\bigg(\dashint_{2D}
\Big(\dashint_{2U_{*,\varepsilon}(x)}
|\nabla\bar{u}|^{2\alpha'}\Big)^{\frac{p_1}{2\alpha'}}
dx\bigg)^{\frac{1}{p_1}}
\lesssim \Big(\dashint_{10D}
|\nabla u_\varepsilon|^{p_0}\Big)^{\frac{1}{p_0}}.
\label{f:3.25}
\end{align}
\end{subequations}
Admitting them for a while, plugging them back into
$\eqref{f:3.34}$ and applying Lemma $\ref{lemma:9}$ we derive that
\begin{equation*}
\Big(\dashint_{2D}
 V_B^{p_1}\Big)^{\frac{1}{p_1}}
\lesssim^{\eqref{f:3.34},\eqref{f:3.24},\eqref{f:3.25}}
\Big(\dashint_{10D}
|\nabla u_\varepsilon|^{p_0}\Big)^{\frac{1}{p_0}}
\lesssim^{\eqref{pri:3.6-1}} \Big(\dashint_{20D}
F^{p_0}\Big)^{\frac{1}{p_0}},
\end{equation*}
which is the stated estimate $\eqref{f:3.31}$ in such the case.

We now turn to address the estimate $\eqref{f:3.24}$.
For any $x\in2D$,
it follows from Lemma $\ref{lemma:7}$ that,
\begin{equation*}
\begin{aligned}
\bigg(\dashint_{2B_{*,\varepsilon}(x)}
\Big(\dashint_{B_{*,\varepsilon}(y)}
|(\nabla\phi_i)^\varepsilon+e_i|^{2}\Big)^{\alpha}
dy\bigg)^{\frac{1}{2\alpha}}
\lesssim^{\eqref{pri:3.7}}
\bigg(\dashint_{8B_{*,\varepsilon}(x)}
\Big(\dashint_{B_{*,\varepsilon}(y)}
|(\nabla\phi_i)^\varepsilon+e_i|^{2}\Big)dy
\bigg)^{\frac{1}{2}}
+ 1
\lesssim 1,
\end{aligned}
\end{equation*}
where we also employ the same computation as those given in $\eqref{f:2.30}$ for the last inequality.
We proceed to show the estimate $\eqref{f:3.25}$.
By H\"older's inequality, we have
\begin{equation}\label{f:3.36}
\bigg(\dashint_{2D}
\Big(\dashint_{2U_{*,\varepsilon}(x)}
|\nabla\bar{u}|^{2\alpha'}\Big)^{\frac{p_1}{2\alpha'}}
dx\bigg)^{\frac{1}{p_1}}
\leq
\bigg(\dashint_{2D}
\Big(\dashint_{2U_{*,\varepsilon}(x)}
|\nabla\bar{u}|^{p_1}\Big)
dx\bigg)^{\frac{1}{p_1}}
=\bigg(\dashint_{D}
\Big(\dashint_{U_{*,\varepsilon}(\tilde{x})}
|\widetilde{\nabla\bar{u}}|^{p_1}\Big)
d\tilde{x}\bigg)^{\frac{1}{p_1}},
\end{equation}
where $\widetilde{\nabla\bar{u}} := (\nabla\bar{u})(2\cdot)$ on $D$, and
the last equality is due to the variable substitution.
Then, we can apply Lemmas $\ref{lemma:9}$ to the right-hand side of
$\eqref{f:3.36}$, and obtain
\begin{equation}\label{f:3.37}
\bigg(\dashint_{D}
\Big(\dashint_{U_{*,\varepsilon}(\tilde{x})}
|\widetilde{\nabla\bar{u}}|^{p_1}\Big)
d\tilde{x}\bigg)^{\frac{1}{p_1}}
\lesssim^{\eqref{pri:3.6-1}}
\Big(\dashint_{2D}
|\widetilde{\nabla \bar{u}}|^{p_1}\Big)^{\frac{1}{p_1}}
= \Big(\dashint_{4D}|\nabla \bar{u}|^{p_1}\Big)^{\frac{1}{p_1}}.
\end{equation}
In the sequel, by using Lemmas $\ref{lemma:16}$,
energy estimates and Lemma $\ref{lemma:15}$ in the order, we derive that
\begin{equation*}
\begin{aligned}
\bigg(\dashint_{2D}
\Big(\dashint_{2U_{*,\varepsilon}(x)}
|\nabla\bar{u}|^{2\alpha'}&\Big)^{\frac{p_1}{2\alpha'}}
dx\bigg)^{\frac{1}{p_1}}
\lesssim^{\eqref{f:3.36},\eqref{f:3.37}} \Big(\dashint_{4D}
|\nabla \bar{u}|^{p_1}\Big)^{\frac{1}{p_1}}\\
&\lesssim^{\eqref{pri:12}}  \Big(\dashint_{5D}
|\nabla \bar{u}|^{2}\Big)^{\frac{1}{2}}
\lesssim \Big(\dashint_{9D}
|\nabla u_\varepsilon|^{2}\Big)^{\frac{1}{2}}
\lesssim^{\eqref{pri:10-b}} \Big(\dashint_{10D}
|\nabla u_\varepsilon|^{p_0}\Big)^{\frac{1}{p_0}}.
\end{aligned}
\end{equation*}

\medskip
\noindent
\textbf{Step 4.} Show the
estimate $\eqref{f:3.32}$ in the case of (II).
By using H\"older's inequality,
Corollary $\ref{cor:1}$ and Lemma $\ref{lemma:9}$ in the order, there holds
\begin{equation*}
\begin{aligned}
\dashint_{2D} W_B^{p_0}
&\lesssim^{\eqref{pri:3.6-1}}
\Big(\dashint_{7D} |\nabla u_\varepsilon
-(e_i+(\nabla\phi_i)^{\varepsilon})\varphi_i|^{2}
\Big)^{\frac{p_0}{2}}
\lesssim^{\eqref{pri:2.1-1},\textcolor[rgb]{1.00,0.00,0.00}{\eqref{pri:10-b}}}
\nu^{p_0}
\dashint_{29D}
|\nabla u_\varepsilon|^{p_0}
\lesssim^{\eqref{pri:3.6-1}} \nu^{p_0}
\dashint_{58D}
U^{p_0}.
\end{aligned}
\end{equation*}
We have the stated estimate $\eqref{f:3.32}$, and
this ends the whole proof.
\end{proof}

\medskip
\noindent
\textbf{Proof of Proposition $\ref{P:10}$.}
%Basically, we should prove quenched Calder\'on-Zygmund estimates
%for general Lipschitz domains and regular SKT (or $C^1$) domains, separately.
The basic idea has been shown
in Subsection $\ref{subsection:3.0}$.
First of all,
let us fix the range of values of parameters $p_1,p_0$ as follows:
\begin{itemize}
\item [(a)] To handle the estimate $\eqref{A}$, we take $p_1\in(2,p_*)$
with $p_0=2$ and, in such the case, we take $\omega=1$;
\item [(b)] To handle the estimate $\eqref{pri:8-1}$,
we take $p_1=2+\theta$ with $p_0\in(p_1',2]$, where
$1/p_1+1/p_1'=1$ and
$0<\theta\ll 1$ depends only on $d$ and $M_0$;
\item [(c)] To handle the estimate $\eqref{f:12}$,
  we take $p_1\in(2,p_*)$ and $p_0\in(1,2]$.
\end{itemize}

\noindent
\textbf{Step 1.} Outline the proof of
$\eqref{f:12}$. It is consistent with the idea given for
the estimate $\eqref{A}$
stated in subsection $\ref{subsection:3.0}$.
Let $B_0$ and $\{B_{0}(x_i)\}_{i=1}^{n_0}$ be given as in $\eqref{f:decom-3}$
and $\omega\in A_{p/p_0}$ with $p\in(p_0,p_1)$.
By the same decomposition as in \textbf{Step A},
the desired estimate $\eqref{f:12}$ is rooted in the following local estimates:
\begin{equation}\label{f:2.8}
\Big(\dashint_{\tilde{B}_0(x_i)\cap\Omega}F^p\omega\Big)^{\frac{1}{p}}
\lesssim\bigg\{
\Big(\dashint_{4\tilde{B}_0(x_i)\cap\Omega}F^{p_0}\Big)^{\frac{1}{p_0}}
\Big(\dashint_{\tilde{B}_0(x_i)}\omega\Big)^{1/{p}}
+\Big(\dashint_{4\tilde{B}_0(x_i)\cap\Omega}|g|^p\omega\Big)^{\frac{1}{p}}\bigg\},
\end{equation}
where we recall that $\tilde{B}_0(x_i) = 4B_0(x_i)$ if $i\in[1,m_0]$;
$\tilde{B}_0(x_i) = B_0(x_i)$ if $i\in[m_0+1,n_0]$,
which played a similar role as $\eqref{f:3.1x}$
did in the estimate $\eqref{A}$\footnote{
\textcolor[rgb]{1.00,0.00,0.00}{If we take $\omega=1$ and
$p_0=2$ in $\eqref{f:2.8}$, then the estimate
$\eqref{f:2.8}$ leads to $\eqref{f:3.1x}$.}\label{footnote1}}.
To see this,  the covering argument (shown in Subsection $\ref{subsection:3.0}$)
provides us with
\begin{equation}\label{f:3.18}
\begin{aligned}
\dashint_{\Omega}F^p\omega
&\lesssim
%\big(\frac{|\Omega|}{|B_0|}\big)^{p-1}
\Big(\dashint_{\Omega}F^{p_0}\Big)^{\frac{p}{p_0}}
\max_{1\leq i\leq n_0}\Big(\dashint_{\tilde{B}_0(x_i)}\omega\Big)
+\dashint_{\Omega}g^p\omega\\
&\lesssim^{\eqref{A}} \Big(\dashint_{\Omega}g^{p_0}\Big)^{\frac{p}{p_0}}
\Big(\dashint_{2\Omega}\omega\Big)
+\dashint_{\Omega}g^p\omega
\lesssim^{\eqref{f:18}}
\Big(\dashint_{\Omega}g^{p_0}\Big)^{\frac{p}{p_0}}
\Big(\dashint_{\Omega}\omega\Big)
+\dashint_{\Omega}g^p\omega.
\end{aligned}
\end{equation}

We further set
$0<p_0-1\ll1$ and, by using H\"older's inequality and the property of
$A_{p/p_0}$ class (it is known from $\eqref{f:19-1}$ that
$\omega\in A_p$ implies $\omega\in A_{p/p_0}$), we can derive that
\begin{equation}\label{f:3.38}
\Big(\dashint_{\Omega}g^{p_0}\Big)^{\frac{1}{p_0}}
\leq \Big(\dashint_{\Omega}g^{p}\omega\Big)^{\frac{1}{p}}
\Big(\dashint_{\Omega}\omega^{-\frac{p_0}{p-p_0}}\Big)^{\frac{p-p_0}{p_0p}}
\lesssim^{\eqref{f:w-1}}_{d,p,[\omega]_{A_p}} \Big(\dashint_{\Omega}g^{p}\omega\Big)^{\frac{1}{p}}
\Big(\dashint_{\Omega}\omega\Big)^{-\frac{1}{p}}.
\end{equation}
Plugging the estimate $\eqref{f:3.38}$ back into the right-hand side of $\eqref{f:3.18}$,
we establish the desired estimate $\eqref{f:12}$.

%(if $\omega=1$, we take $p_1=\frac{2d}{d-1}+\theta$ with
%$0<\theta\ll 1$, $0<p_0-2\ll 1$; while for some $\omega\in A_1$, we take $p_1=2+\theta$ and $0<2-p_0\ll 1$ in the later computations.)
%and regular-SKT (or $C^1$) domains (we choose $2<p_1<\infty$ to be arbitrarily fixed and $1<p_0\leq 2$), .

\noindent
\textbf{Step 2.} Outline the proof of
$\eqref{pri:8-1}$. As the Lipschitz domain is concerned, the estimate $\eqref{f:3.18}$ will be demonstrated
to be true for $p=p_0=2$ and some $\omega\in A_1$.
This inspires us to
modify the estimate $\eqref{f:3.38}$. Thus, by using the property of $A_1$ class,
one can show that
\begin{equation}\label{f:3.35}
\Big(\dashint_{\Omega}g^{2}\Big)^{\frac{1}{2}}
\leq \Big(\dashint_{\Omega}g^{2}\omega\Big)^{\frac{1}{2}}
\Big(\esssup_{\Omega}\omega^{-1}\Big)^{\frac{1}{2}}
\lesssim_{d,[\omega]_{A_1}}^{\eqref{f:w-2}} \Big(\dashint_{\Omega}g^{2}\omega\Big)^{\frac{1}{2}}
\Big(\dashint_{\Omega}\omega\Big)^{-\frac{1}{2}}.
\end{equation}

As a result, combining the estimates $\eqref{f:3.18}$ and $\eqref{f:3.35}$ leads to
\begin{equation*}
\int_{\Omega} F^2 \omega \lesssim \int_{\Omega} g^2\omega,
\end{equation*}
and then the stated estimate $\eqref{pri:8-1}$ finally follows from
the duality argument similar to that in Subsection $\ref{subsection:3.0}$.

\textbf{Note that} we indeed have the following flow chart:
\begin{equation*}
\boxed{~\eqref{f:2.8}~\overset{\omega=1}{\Longrightarrow}~\eqref{f:3.1x}
~\Longrightarrow~\eqref{A}~}
\quad
\overset{+\eqref{f:2.8}}{\Longrightarrow}\quad \eqref{f:3.18}
~+\quad
\left\{\begin{aligned}
&\eqref{f:3.38}~\Longrightarrow~ \eqref{f:12};\\
&\eqref{f:3.35}~\Longrightarrow~ \eqref{pri:8-1}.
\end{aligned}\right.
\end{equation*}
Hence,
the following is devoted to studying how to use Lemma $\ref{shen's lemma2}$
to get the required estimate $\eqref{f:2.8}$,
analogous to that stated in \textbf{Step B} and Remark $\ref{remark:3.1}$.
We remind the reader to pay more attention to
the range of values for $p_1$ and $p_0$ based on (a), (b), and (c),
respectively.

\noindent
\textbf{Step 3.} Make a reduction for
$\eqref{f:2.8}$. As the reduction of Step B shown in
Subsection $\ref{subsection:3.0}$,
it suffices to show $\eqref{f:2.8}$ for the case of $x_i\in\partial\Omega$,
since the interior estimate $\eqref{pri:3.7}$ makes the interior part
$C_{\text{int}}$ easier (see Step 10).
Moreover, it is fine to assume that $0\in\partial\Omega$ and,
by translation invariance of
$\eqref{f:2.8}$ alone $\partial\Omega$, it can be further
reduced to illustrate
$\eqref{f:2.8}$ for $x_i=0$.
Given $\tilde{B}_0$ (which is equal to
4$B_0$ by definition), let $B$ be arbitrary ball
with the properties: $r_B\leq (1/100)r_{\tilde{B}_0}$ and
either $x_B\in 2\tilde{B}_0\cap\partial\Omega$ or
$4B\subset 2\tilde{B}_0\cap\Omega$.
Proceeding as in the proof of Lemma $\ref{lemma:14}$, we merely
fucus on the case $x_B\in 2\tilde{B}_0\cap \partial\Omega$ and
let $D:=B\cap\Omega$.
Thus, we need to find the concrete
forms of $W_B$ and $V_B$ in the case of
$r_B\geq \frac{\varepsilon}{4}\chi_{*}(x_B/\varepsilon)$ and
$x_B\in 2\tilde{B}_0\cap\partial\Omega$.

\medskip
\noindent
 \textbf{Step 4.} Decompose the field $\nabla u_\varepsilon$ and
make a further reduction for $\eqref{f:2.8}$ based upon Step 3.
Let $B$ be given as in the previous step.
We suppose that $w_\varepsilon\in H_0^1(\Omega)$ and
$v_\varepsilon:=u_\varepsilon-w_\varepsilon$ satisfy the following equations:
\begin{equation*}
 -\nabla\cdot a^\varepsilon\nabla w_\varepsilon =
 \nabla \cdot fI_{30D},
 \quad\text{and}\quad
 \nabla\cdot a^\varepsilon\nabla v_\varepsilon
 =0 \quad\text{in}\quad 30D.
\end{equation*}
Then we consider the approximating function $\bar{v}$ satisfying
$\nabla \cdot \bar{a}\nabla\bar{v} = 0$ in $\textcolor[rgb]{1.00,0.00,0.00}{9D}$
with $\bar{v} = v_\varepsilon$ on $\partial (\textcolor[rgb]{1.00,0.00,0.00}{9D})$.
We also introduce
the following notation:
\begin{equation*}
\begin{aligned}
%&F(x):=\Big(\dashint_{U_{*,\varepsilon}(x)}
%|\nabla u_\varepsilon|^{2}\Big)^{1/2}; \qquad
%g(x):=\Big(\dashint_{U_{*,\varepsilon}(x)}|f|^{2}\Big)^{1/2};\\
&V_{B}(x)
:=\Big(\dashint_{U_{*,\varepsilon}(x)}
|(e_i+(\nabla\phi_i)^\varepsilon)
\varphi_i|^{2}\Big)^{1/2};
\qquad
W_{B}(x)
:=\Big(\dashint_{U_{*,\varepsilon}(x)}
|\nabla u_\varepsilon-(e_i+(\nabla\phi_i)^\varepsilon)
\varphi_i|^{2}\Big)^{1/2};\\
&W^{(1)}_{B}(x)
:=\Big(\dashint_{U_{*,\varepsilon}(x)}
|\nabla v_\varepsilon - (e_i+(\nabla\phi_i)^\varepsilon)
\varphi_i|^{2}\Big)^{1/2};\qquad
W^{(2)}_{B}(x)
:=\Big(\dashint_{U_{*,\varepsilon}(x)}|\nabla w_\varepsilon|^{2}\Big)^{1/2},
\end{aligned}
\end{equation*}
where $\varphi_i = S_{*,\varepsilon}(\eta_\varepsilon\partial_i\bar{v})$ similar to that defined in the proof of Lemma $\ref{lemma:14}$.

Then, combining the cases (I) and (II),
one can observe that
$F\leq W_{B}+V_{B}$ and $W_{B}\leq W_{B}^{(1)} + W_{B}^{(2)}$
on $2D$.
Let $p_0, p_1$ be fixed as in (a), (b), and (c) above, and $p_0<p<p_1$ with
$\omega\in A_{p/p_0}$ being a weight function. We claim that:
\begin{equation}\label{f:3.16}
\Big(\dashint_{2D}
V_B^{p_1}\omega
\Big)^{\frac{1}{p_1}} \lesssim
\Big(\dashint_{45D}
(F^{p_0}+g^{p_0})\Big)^{\frac{1}{p_0}}
\Big(\dashint_{B}\omega\Big)^{\frac{1}{p_1}};
\end{equation}
and
\begin{equation}\label{f:3.17}
\dashint_{2D}
W_B^{p_0}
\lesssim \dashint_{60D}
(g^{p_0}
+\nu^{p_0}
F^{p_0}).
\end{equation}

By Lemma $\ref{shen's lemma2}$, we have
\begin{equation*}
\Big(\dashint_{\tilde{B}_0\cap\Omega}F^p\omega\Big)^{\frac{1}{p}}
\lesssim^{\eqref{pri:3.4}}\bigg\{
\Big(\dashint_{4\tilde{B}_0\cap\Omega}F^{p_0}\Big)^{\frac{1}{p_0}}
\Big(\dashint_{\tilde{B}_0}\omega\Big)^{1/{p}}
+\Big(\dashint_{4\tilde{B}_0\cap\Omega}g^p\omega\Big)^{\frac{1}{p}}\bigg\},
\end{equation*}
which indeed leads to the desired estimate $\eqref{f:2.8}$ by
translations alone $\partial\Omega$.
Thus, the crucial job is to verify the claims $\eqref{f:3.16}$
and $\eqref{f:3.17}$.
Similar to those given for Lemmas $\ref{lemma:7}$ and $\ref{lemma:14}$,
the proof of the estimates $\eqref{f:3.16}$ and $\eqref{f:3.17}$ should be divided into
two cases: (I) $0<r_B<\frac{\varepsilon}{4}\chi_{*}(x_B/\varepsilon)$;
(II) $r_B\geq \frac{\varepsilon}{4}\chi_{*}(x_B/\varepsilon)$.
We plan to address the first case in Step 6,
and the second case in Steps 7a (together with Step 9), 7b and 8.

\noindent
\textbf{Step 5.} Show the estimate $\eqref{A}$.
Let $p_0,p_1$ be given as in (a), and one
can verify $\eqref{f:3.16}$ and $\eqref{f:3.17}$
in the cases (I) and (II),
which has already been done by Steps 6, 7a (together with Step 9), 7b, and 8.
(Note the explanation in the corresponding proof
for the specific case of $\omega=1$.)
Then, by Lemma $\ref{shen's lemma2}$, we obtain
\begin{equation*}
\Big(\dashint_{\tilde{B}_0\cap\Omega}F^p\Big)^{\frac{1}{p}}
\lesssim^{\eqref{pri:3.4}}\bigg\{
\Big(\dashint_{4\tilde{B}_0\cap\Omega}F^{2}\Big)^{\frac{1}{2}}
+\Big(\dashint_{4\tilde{B}_0\cap\Omega}g^p\Big)^{\frac{1}{p}}\bigg\},
\end{equation*}
which gives the stated estimate $\eqref{f:3.1x}$ for the boundary part
$C_{\text{bdy}}$ by translations alone $\partial\Omega$.
On the other hand,
by taking $R=r_{B_{0}}$, the estimate $\eqref{pri:3.7}$ indeed provides us
with the stated estimate $\eqref{f:3.1x}$ for the interior part
$C_{\text{int}}$ by translations. Therefore, by
Step A in Subsection $\ref{subsection:3.0}$,
we have proved the estimate $\eqref{A}$.

\medskip
\noindent
\textbf{Step 6. } Establish the estimates $\eqref{f:3.16}$ and
$\eqref{f:3.17}$ in the case of \textcolor[rgb]{1.00,0.00,0.00}{(I)}.
In such the case, the following uniform proof can be derived
under the assumption that $p_1$ and $p_0$
satisfy (a), (b), and (c), respectively.
Here, we adopt the same strategy as that given in Lemma $\ref{lemma:14}$.
We merely recall that $W_B = F$ and $V_B = 0$ in such the case,
and therefore focus on the estimate $\eqref{f:3.16}$.
Similar to that given for $\eqref{f:3.33}$,
for any $\alpha\geq 1$, we have
\begin{equation}\label{f:3.19}
\bigg(\dashint_{2D}
\Big(\dashint_{U_{*,\varepsilon}(x)}
|\nabla u_\varepsilon|^2\Big)^{\frac{\alpha}{2}}
dx\bigg)^{\frac{1}{\alpha}}
\lesssim \dashint_{15D}\Big(
 \dashint_{U_{*,\varepsilon}(x)}|\nabla u_\varepsilon|^2\Big)^{\frac{1}{2}}dx.
\end{equation}
Moreover, from H\"older's inequality and the
reverse H\"older's property of $A_p$ class, it follows that for any $0<\beta'-1\ll 1$ with $1/\beta+1/\beta'=1$,
\begin{equation}\label{f:3.20}
\begin{aligned}
\bigg(\dashint_{2D}
\Big(\dashint_{U_{*,\varepsilon}(x)}
|\nabla u_\varepsilon|^2\Big)^{\frac{p_1}{2}}
\omega(x)dx\bigg)^{\frac{1}{p_1}}
&\leq
\bigg(\dashint_{2D}
\Big(\dashint_{U_{*,\varepsilon}(x)}
|\nabla u_\varepsilon|^2\Big)^{\frac{\beta p_1}{2}} \bigg)^{\frac{1}{\beta p_1}}
\Big(\dashint_{2B}\omega^{\beta'}\Big)^{\frac{1}{\beta'p_1}}\\
&\lesssim^{\eqref{pri:R-1},\eqref{f:18}}
\bigg(\dashint_{2D}
\Big(\dashint_{U_{*,\varepsilon}(x)}|\nabla u_\varepsilon|^2\Big)^{\frac{\beta p_1}{2}} \bigg)^{\frac{1}{\beta p_1}}
\Big(\dashint_{B}\omega\Big)^{\frac{1}{p_1}}.
\end{aligned}
\end{equation}

Combining the estimates $\eqref{f:3.19}$ and $\eqref{f:3.20}$ we obtain
\begin{equation*}
\bigg(\dashint_{2D}
F^{p_1}\omega\bigg)^{\frac{1}{p_1}}
\lesssim
\bigg(\dashint_{15D}
F^{p_0}\bigg)^{\frac{1}{p_0}}
\Big(\dashint_{B}\omega\Big)^{\frac{1}{p_1}},
\end{equation*}
which in fact proves the stated estimate $\eqref{f:3.16}$ by noting
that we choose $W_B =F$ in such the case.

\medskip
\noindent
\textbf{Step 7a.} Show the estimate $\eqref{f:3.16}$ in the case of
(II) for the general
Lipschitz domain.
In such the case, the following uniform proof can be derived
under the assumption that $p_1$ with $p_0$
satisfies (a) and (b), respectively.
Let $\alpha'$ be the conjugate index of $\alpha$ satisfying
$0<\alpha'-1\ll 1$.
As a preparation,
we claim that for some $\omega\in A_1$, there holds that
\begin{equation}\label{f:3.9}
\bigg(\dashint_{4D}
|\nabla\bar{v}|^{p_1}\omega
\bigg)^{\frac{1}{p_1}}
\lesssim \bigg(\dashint_{8D}
|\nabla\bar{v}|^{2}
\bigg)^{\frac{1}{2}}
\Big(\dashint_{B}\omega\Big)^{\frac{1}{p_1}}
\end{equation}
(and we give a short proof in Step 9).
%In particular, if $\omega=1$, one can choose $p_1 = \frac{2d}{d-1}+\theta$ with $0<\theta\ll 1$ in $\eqref{f:3.9}$
%(see Lemma $\ref{lemma:16}$).
In view of the estimates $\eqref{f:3.34}$ and $\eqref{f:3.24}$,
we have
\begin{equation}\label{f:3.10}
\begin{aligned}
\Big(\dashint_{2D}
 V_B^{p_1}\omega\Big)^{\frac{1}{p_1}}
\lesssim  \bigg(\dashint_{2D}
\Big(\dashint_{2U_{*,\varepsilon}(x)}
|\nabla\bar{v}|^{2\alpha'}\Big)^{\frac{p_1}{2\alpha'}}
\omega(x)dx\bigg)^{\frac{1}{p_1}}.
\end{aligned}
\end{equation}
Then, by Fubini's theorem and the definition of $A_1$
weight, we can further derive that
\begin{equation*}
\begin{aligned}
\Big(\dashint_{2D}
 V_B^{p_1}\omega\Big)^{\frac{1}{p_1}}
&\lesssim
\bigg(\dashint_{4D}
|\nabla\bar{v}(x)|^{p_1}\Big(\dashint_{2B_{*,\varepsilon}(x)}\omega\Big)
dx\bigg)^{\frac{1}{p_1}}\\
&\lesssim^{\eqref{f:w-2}} \bigg(\dashint_{4D}
|\nabla\bar{v}|^{p_1}\omega
\bigg)^{\frac{1}{p_1}}
\lesssim^{\eqref{f:3.9}} \bigg(\dashint_{8D}
|\nabla\bar{v}|^{2}
\bigg)^{\frac{1}{2}}
\Big(\dashint_{B}\omega\Big)^{\frac{1}{p_1}}.
\end{aligned}
\end{equation*}
It follows that
\begin{equation}\label{f:3.23}
\begin{aligned}
\Big(\dashint_{2D}
 V_B^{p_1}\omega\Big)^{\frac{1}{p_1}}
&\lesssim
\bigg(\dashint_{
8D}
|\nabla\bar{v}|^{2}
\bigg)^{\frac{1}{2}}
\Big(\dashint_{B}
\omega \Big)^{\frac{1}{p_1}}
\lesssim
\bigg(\dashint_{9D}
|\nabla v_\varepsilon|^{2}
\bigg)^{\frac{1}{2}}
\Big(\dashint_{B}
\omega \Big)^{\frac{1}{p_1}}\\
&\lesssim^{\eqref{pri:10}} \Big(\dashint_{10D}
|\nabla v_\varepsilon|^{p_0}\Big)^{\frac{1}{p_0}}
\Big(\dashint_{B}
\omega\Big)^{\frac{1}{p_1}}
\lesssim^{\eqref{pri:3.6-1}} \Big(\dashint_{20D}
F^{p_0}
+{W_B^{(2)}}^{p_0}\Big)^{\frac{1}{p_0}}
\Big(\dashint_{B}
\omega\Big)^{\frac{1}{p_1}}\\
&\lesssim^{\eqref{A}}
\Big(\dashint_{45D}
F^{p_0}
+g^{p_0}\Big)^{\frac{1}{p_0}}
\Big(\dashint_{B}
\omega\Big)^{\frac{1}{p_1}},
\end{aligned}
\end{equation}
which consequently leads to the stated estimate $\eqref{f:3.16}$
for the general Lipschitz domain.
Also, we mention that if $\omega=1$, the last inequality in $\eqref{f:3.23}$
is due to the energy estimate $\eqref{f:3.8x}$ instead of $\eqref{A}$.

\medskip
\noindent
\textbf{Step 7b.} Show the estimate $\eqref{f:3.16}$ in the case of
(II) for the regular SKT domain.
In such the case, the following uniform proof can be derived
under the assumption that $p_1$ with $p_0$
satisfies (a) and (c), respectively.
We start from the estimate $\eqref{f:3.10}$
and, by H\"older's inequality, the reverse H\"older's property
of $A_p$ class, and Lemmas $\ref{lemma:15},\ref{lemma:16}$, we obtain
\begin{equation}\label{f:3.28}
\begin{aligned}
\Big(\dashint_{2D}
 V_B^{p_1}\omega\Big)^{\frac{1}{p_1}}
&\lesssim_d
\bigg(\dashint_{2D}
|\nabla\bar{v}|^{p_1\beta}
\bigg)^{\frac{1}{p_1\beta}}
\Big(\dashint_{4B}
\omega^{\beta'}\Big)^{\frac{1}{p_1\beta'}}\\
&\lesssim^{\eqref{pri:12},\eqref{pri:10},\eqref{pri:R-1}}
\Big(\dashint_{10D}
|\nabla v_\varepsilon|^{p_0}\Big)^{\frac{1}{p_0}}
\Big(\dashint_{B}
\omega\Big)^{\frac{1}{p_1}}
\lesssim^{\eqref{pri:3.6-1}} \Big(\dashint_{20D}
F^{p_0}
+{W_B^{(2)}}^{p_0}\Big)^{\frac{1}{p_0}}
\Big(\dashint_{B}
\omega\Big)^{\frac{1}{p_1}}\\
&\lesssim^{\eqref{A}}
\Big(\dashint_{45D}
F^{p_0}
+g^{p_0}\Big)^{\frac{1}{p_0}}
\Big(\dashint_{B}
\omega\Big)^{\frac{1}{p_1}}.
\end{aligned}
\end{equation}
This consequently leads to the stated estimate $\eqref{f:3.16}$
for the regular SKT (or $C^1$) domain. Similar
as in Step 7a, if $\omega =1$,
the last inequality in $\eqref{f:3.28}$ follows from the energy estimate
$\eqref{f:3.8x}$ instead of $\eqref{A}$.

\medskip
\noindent
\textbf{Step 8.} Show the estimate $\eqref{f:3.17}$ in the case of (II).
The argument is similar to that given for $\eqref{f:3.32}$.
In view of $W_{B}\leq W_{B}^{(1)} + W_{B}^{(2)}$,
we first have
\begin{equation}\label{f:3.27}
\begin{aligned}
\dashint_{2D}
W_B^{p_0}
&\leq
\dashint_{2D}
\Big(\dashint_{U_{*,\varepsilon}(x)}
|\nabla v_\varepsilon
-(e_i+(\nabla\phi_i)^\varepsilon)\varphi_i|^2\Big)^{\frac{p_0}{2}}
dx
 + \dashint_{2D}
 \Big(\dashint_{U_{*,\varepsilon}(x)}|\nabla w_\varepsilon
 |^2\Big)^{\frac{p_0}{2}}
 dx\\
&\lesssim^{\eqref{pri:3.6-1},\eqref{A}}
\Big(\dashint_{7D}
|\nabla v_\varepsilon
-(e_i+(\nabla\phi_i)^\varepsilon)\varphi_i|^2\Big)^{\frac{p_0}{2}}
 + \dashint_{45D}
 \Big(\dashint_{U_{*,\varepsilon}(x)}|f|^2\Big)^{\frac{p_0}{2}}.
\end{aligned}
\end{equation}

If $p_1$ with $p_0$ satisfies (a),
then it follows from the first line of $\eqref{f:3.27}$ that
\begin{equation*}
\begin{aligned}
\dashint_{2D} W_B^{2}
&\lesssim^{\eqref{pri:3.6-1}}
\dashint_{7D}
|\nabla v_\varepsilon
-(e_i+(\nabla\phi_i)^\varepsilon)\varphi_i|^2
 + \dashint_{7D}|\nabla w_\varepsilon|^2
\lesssim^{\eqref{pri:2.1-1}}
\textcolor[rgb]{0.00,0.00,1.00}{\nu^2}\dashint_{28D}
|\nabla v_\varepsilon|^2
 + \dashint_{7D}|\nabla w_\varepsilon|^2\\
 &\lesssim
\textcolor[rgb]{0.00,0.00,1.00}{\nu^2}\dashint_{28D}
|\nabla u_\varepsilon|^2
 + \dashint_{30D}|f|^2
\lesssim^{\eqref{pri:3.6-1}} \textcolor[rgb]{0.00,0.00,1.00}{\nu^2}
\dashint_{56D}F^2
+\dashint_{60D}g^2
\end{aligned}
\end{equation*}
(which leads to $\eqref{f:3.17}$ in such the case),
where the third inequality above is due to the energy estimate
and the fact that $\nabla v_\varepsilon =
\nabla u_\varepsilon-\nabla w_\varepsilon$ in $30B\cap\Omega$.

If $p_1$ with $p_0$ satisfies (b) and (c), respectively,
it suffices to address the first term in the second line of $\eqref{f:3.27}$.
With the help of Corollary $\ref{cor:1}$ together with Lemma $\ref{lemma:15}$, there holds
\begin{equation*}
\begin{aligned}
\Big(\dashint_{7D}
|\nabla v_\varepsilon
-(e_i+(\nabla\phi_i)^\varepsilon)\varphi_i|^2\Big)^{\frac{1}{2}}
\lesssim^{\eqref{pri:2.1-1}} \nu
\Big(\dashint_{28D}
|\nabla v_\varepsilon|^2\Big)^{\frac{1}{2}}
\lesssim^{\eqref{pri:10}}
\nu
\dashint_{29D}
|\nabla v_\varepsilon|,
\end{aligned}
\end{equation*}
and then
a similar computation as that given for the second and
third lines of $\eqref{f:3.28}$ leads to
\begin{equation*}
\Big(\dashint_{7D}
|\nabla v_\varepsilon
-(e_i+(\nabla\phi_i)^\varepsilon)\varphi_i|^2\Big)^{\frac{p_0}{2}}
\lesssim \dashint_{45D}g^{p_0}+
\nu^{p_0}\dashint_{60D} F^{p_0}.
\end{equation*}
Plugging this back into $\eqref{f:3.27}$, we have already derived the stated
estimate $\eqref{f:3.17}$.

\medskip
\noindent
\textbf{Step 9.} Show the details on the estimate $\eqref{f:3.9}$.
According to (a) and (b), we recall that
if $\omega=1$ we take $p_1\in(2,p_*)$; if $\omega\in A_1$
we take $p_1 = 2+\theta$ here. Let $\omega(x):=[\text{dist}(x,\partial\Omega)]^{\sigma}$
with $\sigma\in(-1,0)$. It concludes that $\omega\in A_1$ (see \cite[pp.3]{Shen20}).
In terms of the equation that $\bar{v}$ satisfies,
by decomposing the region $4D$ into ``layer'' part  and ``co-layer'' part, and then
applying nontangential maximal function estimates to the layer part\footnote{In this regard, we refer the reader to the proof of $\eqref{f:6.15}$ for
the details.},
we obtain
\begin{equation*}
\begin{aligned}
\dashint_{4D}|\nabla\bar{v}|^{p_1}\omega
&\lesssim r^{\sigma}\Big(\dashint_{4B\cap\partial\Omega}
|(\nabla\bar{v})^{*}|^{p_1}  dS +
\dashint_{4D}|\nabla\bar{v}|^{p_1} \Big)\\
&\lesssim r^{\sigma}
\dashint_{6D}|\nabla\bar{v}|^{p_1}
\lesssim \Big(\dashint_{8D}|\nabla\bar{v}|^{2}\Big)^{\frac{p_1}{2}}
\Big(\dashint_{B}\omega\Big),
\end{aligned}
\end{equation*}
which implies the stated estimate $\eqref{f:3.9}$.
Note that the above estimate may not be true
if $p_1>2+\theta$ (see Remark $\ref{remark:5}$).
However, if $\omega=1$, the above estimate will follow
from Lemma $\ref{lemma:16}$, and it allows $p_1\in(2,p_*)$ in
such the case.

\noindent
\textbf{Step 10.} Show the estimate $\eqref{f:2.8}$ for the interior part
$C_{\text{int}}$. In such the case,
we have $\tilde{B}_0(x_i) = B_0(x_i)$ with $i\in[m_0+1,n_0]$
and $4B_0(x_i)\subset\Omega$.
Analogous to the reduction for the boundary part $C_{\text{bdy}}$,
it is fine to assume $x_i =0$. Given $B_0$, let $B$ be
arbitrary ball with $B\subset 2B_0$
and $r_B\lesssim (1/20)r_{B_0}$. In terms of the case (II), we make
the following decomposition.
We suppose that $w_\varepsilon\in H_0^1(\mathbb{R}^d)$ and
$v_\varepsilon:=u_\varepsilon-w_\varepsilon$ satisfy the equations:
\begin{equation*}
-\nabla\cdot a^\varepsilon\nabla w_\varepsilon =
 \nabla \cdot fI_{18B},
 \quad\text{and}\quad
 \nabla\cdot a^\varepsilon\nabla v_\varepsilon
 =0 \quad\text{in}\quad 18B.
\end{equation*}
We also introduce
the following notation:
\begin{equation*}
V_{B}(x):=\Big(\dashint_{B_{*,\varepsilon}(x)}
|\nabla v_\varepsilon|^{2}\Big)^{1/2};
\quad
W_{B}(x):=\Big(\dashint_{B_{*,\varepsilon}(x)}
|\nabla w_\varepsilon|^{2}\Big)^{1/2};
\quad
\tilde{g}:=\Big(\dashint_{B_{*,\varepsilon}(x)}
|fI_{18B}|^{2}\Big)^{1/2}.
\end{equation*}

Then, combining the cases (I) and (II)
(recall the case (I) in Subsection $\ref{subsection:3.0}$),
one can observe that
$F\leq W_{B}+V_{B}$
on $2B$.
Let $p_0, p_1$ be fixed as (c) above, $p_0<p<p_1$ with
$\omega\in A_{p/p_0}$ being a weight function. We claim that:
\begin{subequations}
\begin{align}
&\Big(\dashint_{2B}W_B^{p_0}\Big)^{\frac{1}{p_0}}
\lesssim \Big(\dashint_{25B}
g^{p_0}\Big)^{\frac{1}{p_0}}; \label{f:3.17-*}\\
& \Big(\dashint_{2B}
V_B^{p_1}\omega
\Big)^{\frac{1}{p_1}} \lesssim
\Big(\dashint_{36B}(F^{p_0}+g^{p_0})\Big)^{\frac{1}{p_0}}
\Big(\dashint_{B}\omega\Big)^{\frac{1}{p_1}}. \label{f:3.16-*}
\end{align}
\end{subequations}

First, we can use the same argument shown in Step 6
to illustrate that $\eqref{f:3.17-*}$ and
$\eqref{f:3.16-*}$ hold for the case of (I).
Then, we handle the case (II). To see $\eqref{f:3.17-*}$, we start from
\begin{equation}\label{f:3.39}
\dashint_{2B}W_B^{p_0}
\leq \frac{1}{|2B|}\int_{\mathbb{R}^d}
W_B^{p_0}
\lesssim^{\eqref{pri:3.7-1}}
\frac{1}{|2B|}\int_{\mathbb{R}^d}
\tilde{g}^{p_0}
\lesssim \dashint_{25B}g^{p_0},
\end{equation}
where we employ the fact that
$\text{supp}(\tilde{g})\subset 25B$ due to the condition
$r_B\geq \frac{\varepsilon}{4}\chi_{*}(x_B/\varepsilon)$
for the last inequality. This indeed gives the stated estimate $\eqref{f:3.17-*}$.
Now, we proceed to deal with $\eqref{f:3.16-*}$, and obtain 
\begin{equation}\label{f:3.40}
\begin{aligned}
\Big(\dashint_{2B}V_B^{p_1}\omega\Big)^{\frac{1}{p_1}}
&\leq \Big(\dashint_{2B}V_B^{p_1\beta}\Big)^{\frac{1}{p_1\beta}}
\Big(\dashint_{2B}\omega^{\beta'}\Big)^{\frac{1}{p_1\beta'}}\\
&\lesssim^{\eqref{pri:3.7},\eqref{pri:R-1}}
\Big(\dashint_{8B}V_B^2\Big)^{\frac{1}{2}}
\Big(\dashint_{2B}\omega\Big)^{\frac{1}{p_1}}
\lesssim^{\eqref{pri:3.6-1},\eqref{f:18}}
\Big(\dashint_{16B}|\nabla v_\varepsilon|^2\Big)^{\frac{1}{2}}
\Big(\dashint_{B}\omega\Big)^{\frac{1}{p_1}}.
\end{aligned}
\end{equation}
Moreover, we can infer that
\begin{equation*}
\begin{aligned}
\Big(\dashint_{16B}|\nabla v_\varepsilon|^2\Big)^{\frac{1}{2}}
&\lesssim^{\eqref{pri:10-a}} \dashint_{18B}|\nabla v_\varepsilon|
\lesssim  \dashint_{18B}|\nabla u_\varepsilon| + |\nabla w_\varepsilon|\\
&\lesssim^{\eqref{pri:3.6-1}}
 \Big(\dashint_{36B} \big(F^{p_0} + W_B^{p_0}\big)\Big)^{\frac{1}{p_0}}
 \lesssim^{\eqref{f:3.39}}
 \Big(\dashint_{36B} \big(F^{p_0} + g^{p_0}\big)\Big)^{\frac{1}{p_0}}.
\end{aligned}
\end{equation*}
Plugging this back into $\eqref{f:3.40}$, we have derived
the stated estimate $\eqref{f:3.16-*}$. Therefore, by using
Lemma $\ref{shen's lemma2}$, we obtain
\begin{equation*}
\Big(\dashint_{B_0}F^p\omega\Big)^{\frac{1}{p}}
\lesssim^{\eqref{pri:3.4}}\bigg\{
\Big(\dashint_{4B_0}F^{p_0}\Big)^{\frac{1}{p_0}}
\Big(\dashint_{B_0}\omega\Big)^{1/{p}}
+\Big(\dashint_{4B_0}g^p\omega\Big)^{\frac{1}{p}}\bigg\},
\end{equation*}
which implies the estimate $\eqref{f:2.8}$ for the interior part
$C_{\text{int}}$ by translations.
This ends the whole proof.
\qed

%\begin{remark}
%\emph{Due to the estimate $\eqref{f:3.9}$ proved in Step 6 above, one can derive reverse H\"older's inequality
%of weighted version on general Lipschitz domains as follows.
%Let $p_1=2+\theta$ with $0<\theta\ll 1$ and $0<p_0<q<p_1$. Assume $u_\varepsilon$ is the solution of $\eqref{pde:1}$.
%Then, for some weight $\omega\in A_1$ given in Step 6 above, we have
%\begin{equation}\label{AA}
%\bigg(\dashint_{B_{\frac{R}{2}}\cap\Omega}
%\Big(\dashint_{D_{*,\varepsilon}(x)}|\nabla u_\varepsilon|^2\Big)^{\frac{q}{2}}\omega\bigg)^{\frac{p_0}{q}}
%\lesssim
%\Big(\dashint_{B}\omega\Big)^{\frac{p_0}{q}}\dashint_{B_{2R}\cap\Omega}
%\Big(\dashint_{D_{*,\varepsilon}(x)}|\nabla u_\varepsilon|^2 \Big)^{\frac{p_0}{2}}.
%\end{equation}
%It can be derived by modifying Step 2,3 in the proof of Lemma $\ref{lemma:14}$ and the details can be found in Step 3, 4a in the proof
%of Proposition $\ref{P:10}$, respectively.}
%\end{remark}

\subsection{\centering Annealed estimates}\label{subsection:3.2}

\noindent
Because the integrable index of spatial variables is
determined by the regularity of the boundaries, weighted
Calder\'on-Zygmund estimates are unlikely to hold for all the
members in the Muckenhoupt’s weight class (otherwise
there will be a contradiction by extrapolation\footnote{
Concretely, if the estimate $\eqref{pri:9}$ is true
for some $q\in(1,\infty)$,  then the extrapolation
argument (see e.g. \cite[Theorem 7.8]{Duoandikoetxea01}) implies that
the estimate $\eqref{pri:9}$ will hold
for each $q\in(1,\infty)$. However, it is a contradiction to
the established estimate $\eqref{pri:9-3*}$ for
the case of Lipschitz domains, when
we trivially take $\omega=1$ in $\eqref{pri:9}$.}). Therefore,
we first establish the weighted estimates for the regular SKT (or $C^1$) domains (see Lemma $\ref{lemma:11}$).
Here, similar to an extrapolation argument, we merely show
the desired estimate for any $A_1$-weight at first,
and then pass to any $A_q$-weight.
In addition, as shown in Duerinckx and Otto's work
\cite[Theorem 6.1]{Duerinckx-Otto20}, when
we transfer the corresponding Calder\'on-Zygmund estimates from
averages at the scale $\chi_*$ to the unit scale,
%whose method
%is independent of PDEs,
it destroyed the
previous equilibrium of the stochastic integrability on both sides
(see the difference between $\eqref{pri:9-3}$ and $\eqref{pri:9-3*}$,
as well as, $\eqref{pri:9}$ and $\eqref{pri:11}$).  Finally, we mention
that by modifying the proof in a few places, the
corresponding estimates still hold true for the Neumann boundary value problems.

% we need to
%we employ some convex inequalities (such as Minkowski's inequality and H\"older’s
%inequality of $l^p$-version and $L^p$-version), as well as Chebyshev’s inequality on minimal radius $\chi_{*}$, to reduce the large-scale averages to the small-scale one\footnote{Here, small-scale average means unite average.} with the loss of the stochastic integrability at the right-hand side (see Lemma $\ref{lemma:12}$).

\begin{proposition}\label{P:7}
Let $\Omega$ be a general Lipschitz domain with $\varepsilon\in(0,1]$,
and $p_{*}$ is given by $\eqref{upindex-p}$.
Suppose that $\langle\cdot\rangle$ is stationary, satisfying the spectral gap condition $\eqref{a:2}$, and the (admissible) coefficient additionally satisfies $\eqref{a:3}$ and
the symmetry condition $a=a^*$. Let $u_\varepsilon$ be a weak solution to $\eqref{pde:2-D}$.
%\begin{equation}
%\left\{\begin{aligned}
%-\nabla\cdot a^\varepsilon\nabla u_\varepsilon &= \nabla\cdot f
%&\quad&\text{in}~~\Omega;\\
%u_\varepsilon &= 0
%&\quad&\text{on}~~\partial\Omega.
%\end{aligned}\right.
%\end{equation}
Then, there exists a stationary random field $\chi_*$
given as in Corollary $\ref{cor:1}$ with the $\frac{1}{L}$-Lipschitz continuity such that
\begin{equation}\label{pri:9-3}
\int_{\Omega}
\Big\langle\Big(\dashint_{
U_{*,\varepsilon}(x)}
|\nabla u_\varepsilon|^2 \Big)^{\frac{p}{2}}
\Big\rangle^{\frac{q}{p}}
\lesssim_{\lambda,d,M_0,p,q}
\int_{\Omega}
\Big\langle\Big(\dashint_{U_{*,\varepsilon}(x)}
|f|^{2}\Big)^{\frac{p}{2}}\Big\rangle^{\frac{q}{p}},
\end{equation}
holds for any $p_{*}'<p,q<p_*$ (where $p_{*}'$ is the conjugate number of
$p_{*}$),
and for any $p<\bar{p}<\infty$ we have
\begin{equation}\label{pri:9-3*}
\int_{\Omega}
\Big\langle\Big(\dashint_{
U_{\varepsilon}(x)}|\nabla u_\varepsilon|^2 \Big)^{\frac{p}{2}}
\Big\rangle^{\frac{q}{p}}
\lesssim_{\lambda,\lambda_1,
\lambda_2,d,M_0,p,\bar{p},q}
\int_{\Omega}
\Big\langle\Big(\dashint_{
U_{\varepsilon}(x)}
|f|^{2}\Big)^{\frac{\bar{p}}{2}}\Big\rangle^{\frac{q}{\bar{p}}}.
\end{equation}
Moreover, if $\Omega$ is a bounded regular SKT (or $C^1$) domain
(in such the case the symmetry condition $a=a^*$ is not needed), and $1<p,q<\infty$.
Then,
for any $\omega\in A_q$, we have
\begin{equation}\label{pri:9}
\bigg(\int_{\Omega}
\Big\langle\Big(\dashint_{
U_{*,\varepsilon}(x)}|\nabla u_\varepsilon|^2 \Big)^{\frac{p}{2}}
\Big\rangle^{\frac{q}{p}}\omega\bigg)^{\frac{1}{q}}
\lesssim_{\lambda,d,M_0,p,q,[\omega]_{A_q}}
\bigg(\int_{\Omega}
\Big\langle\Big(\dashint_{U_{*,\varepsilon}(x)}
|f|^2\Big)^{\frac{p}{2}}\Big\rangle^{\frac{q}{p}}\omega\bigg)^{\frac{1}{q}},
\end{equation}
Also, for any $p<\bar{p}<\infty$, there holds
\begin{equation}\label{pri:11}
\bigg(\int_{\Omega}\Big\langle\Big(\dashint_{
U_{\varepsilon}(x)}
|\nabla u_\varepsilon|^2 dx\Big)^{\frac{p}{2}}
\Big\rangle^{\frac{q}{p}}\dashint_{B_\varepsilon(x)}\omega\bigg)^{\frac{1}{q}}
\lesssim_{\lambda,\lambda_1,\lambda_2,d,M_0,p,\bar{p},q,[\omega]_{A_q}} \bigg(\int_{\Omega}
\Big\langle\Big(\dashint_{
U_{\varepsilon}(x)}|f|^2 dx\Big)^{\frac{\bar{p}}{2}}\Big\rangle^{\frac{q}{\bar{p}}}
\dashint_{B_\varepsilon(x)}\omega\bigg)^{\frac{1}{q}}.
\end{equation}
In particular, we also have
\begin{equation}\label{pri:26}
\bigg(\int_{\Omega}\Big\langle
\Big(\dashint_{U_\varepsilon(x)}|\nabla u_\varepsilon|^2\Big)^{\frac{p}{2}}
\Big\rangle^{\frac{q}{p}}\bigg)^{\frac{1}{q}}
\lesssim_{\lambda,\lambda_1,\lambda_2,d,M_0,p,\bar{p},q} \bigg(\int_{\Omega}
\Big\langle\Big(\dashint_{
U_{\varepsilon}(x)}|f|^2
\Big)^{\frac{\bar{p}}{2}}\Big\rangle^{\frac{q}{\bar{p}}}\bigg)^{\frac{1}{q}},
\end{equation}
and by taking $\omega_\sigma(x):=
[\text{dist}(x,\partial\Omega_0)]^{q-1}$ with
$\Omega_0\supseteq \Omega$
satisfying $\partial\Omega_0\in C^2$ and   $\text{dist}(\cdot,\partial\Omega_0)=O(\varepsilon)$ on $\partial\Omega$,  we obtain
\begin{equation}\label{pri:11B}
\bigg(\int_{\Omega}\Big\langle
\Big(\dashint_{
U_\varepsilon(x)}
|\nabla u_\varepsilon|^2 \Big)^{\frac{p}{2}}
\Big\rangle^{\frac{q}{p}}\omega_\sigma\bigg)^{\frac{1}{q}}
\lesssim_{\lambda,\lambda_1,\lambda_2,d,M_0,p,\bar{p},q} \bigg(\int_{\Omega}
\Big\langle\Big(\dashint_{
U_{\varepsilon}(x)}
|f|^2 \Big)^{\frac{\bar{p}}{2}}\Big\rangle^{\frac{q}{\bar{p}}}
\omega_\sigma\bigg)^{\frac{1}{q}}.
\end{equation}
\end{proposition}

\begin{remark}
\emph{Similar to Proposition $\ref{P:10}$, in the scalar case of the equations $\eqref{pde:2-D}$
established on a general Lipschitz domain, we take $q_1=4+\theta$ if $d=2$; and $q_1=3+\theta$ if $d\geq 3$, where $0<\theta\ll 1$. When we pass from the estimate $\eqref{pri:11}$
to $\eqref{pri:11B}$, not all of the weights can be equally
interchangeable with average integrals.
Finally, we mention that
the loss in the stochastic integrability
of $\eqref{pri:9-3*}$ and $\eqref{pri:11}$ is caused by the fact
that $\chi_{*}$ and $\nabla u_\varepsilon$ are not independent, and
the reader can observe it in the proof of Lemma $\ref{lemma:12}$
(see e.g. the estimate $\eqref{loss}$).}
\end{remark}

\begin{lemma}\label{lemma:11}
Let $\Omega$ be a bounded regular SKT (or $C^1$) domain, and $1<p,q<\infty$.
Assume the same conditions hold as in Proposition $\ref{P:7}$ in such the case. Let
$u_\varepsilon$ and $f$ be associated with
the equations $\eqref{pde:2-D}$. Then, there exists a
stationary random field $\chi_*$
given as in Corollary $\ref{cor:1}$ with the $\frac{1}{L}$-Lipschitz
continuity such that
\begin{equation}\label{pri:9-2}
\int_{\Omega}
\Big\langle\Big(\dashint_{U_{*,\varepsilon}(x)}|\nabla u_\varepsilon|^2 \Big)^{\frac{p}{2}}
\Big\rangle^{\frac{q}{p}}
\lesssim_{\lambda,d,M_0,p,q}
\int_{\Omega}
\Big\langle\Big(\dashint_{U_{*,\varepsilon}(x)}
|f|^{2}\Big)^{\frac{p}{2}}\Big\rangle^{\frac{q}{p}}.
\end{equation}
Moreover, for any $\omega\in A_q$, we have
\begin{equation}\label{pri:9-1}
\int_{\Omega}
\Big\langle\Big(\dashint_{U_{*,\varepsilon}(x)}|\nabla u_\varepsilon|^2 \Big)^{\frac{p}{2}}
\Big\rangle^{\frac{q}{p}}\omega
\lesssim_{\lambda,d,M_0,p,q,[\omega]_{A_q}}
\int_{\Omega}
\Big\langle\Big(\dashint_{U_{*,\varepsilon}(x)}
|f|^{2}\Big)^{\frac{p}{2}}\Big\rangle^{\frac{q}{p}}\omega.
\end{equation}
\end{lemma}

\begin{proof}
Based upon Proposition $\ref{P:10}$, the idea is to repeat
using Shen's real arguments, which was first applied
to studying annealed Calder\'on-Zygmund estimates for
the stochastic singular integral given by
$\nabla(\nabla\cdot a\nabla)^{-1}
\nabla\cdot a$ in \cite{Duerinckx-Otto20}. We modify
the original proof to fulfill our purpose, and
the whole proof will be completed in two parts:
Part A (Steps 1-3) and Part B (Steps 4-6).
We prove the estimate $\eqref{pri:9-2}$ in
Part A. Then, based upon the derived result $\eqref{pri:9-2}$,
reusing Shen's real arguments leads to the desired estimate $\eqref{pri:9-1}$
for any weight $\omega\in A_q$.
The basic idea is similar to that given for Proposition $\ref{P:10}$,
but unlike the decomposition made in Proposition $\ref{P:10}$,
the concrete form of $W_B$ and $V_B$ defined below is
independent of $r_B$.
In addition,
we omit the detail from a family of ``local estimates''
to global ones
(since it has been well presented in the previous two subsections),
whereas we merely focus on verifying the preconditions
satisfied by $W_B$ and $V_B$, respectively,
in Lemma $\ref{shen's lemma2}$.

\medskip
\noindent
\textbf{Step 1.} Reduction and outline of the proof of $\eqref{pri:9-2}$.
Let $p,q_1\in(1,\infty)$ be arbitrarily fixed, and we need to show
the estimate $\eqref{pri:9-2}$ for the case $p\leq q<q_1$, and then
the case $q<p$ follows from a duality argument.
We start from making the decomposition according to
$\eqref{pde:2-D}$.
For any ball $B$ with
the property that either $x_B\in\partial\Omega$ or
$4B\subset\Omega$, let $w_\varepsilon,v_\varepsilon\in H_0^1(\Omega)$ satisfy the following equations:
\begin{equation}\label{pde:21}
 -\nabla\cdot a^\varepsilon\nabla w_\varepsilon =
 \nabla \cdot fI_{2B\cap\Omega},
 \quad\text{and}\quad
 -\nabla\cdot a^\varepsilon\nabla v_\varepsilon
 =\nabla \cdot fI_{\Omega\setminus2B}\quad\text{in}\quad \Omega.
\end{equation}
For the ease of statement,
we introduce the following notation:
\begin{equation*}
 F(x)
 :=\Big\langle\Big(\dashint_{U_{*,\varepsilon}(x)}|\nabla u_\varepsilon|^2\Big)^{\frac{p}{2}}\Big\rangle^{\frac{1}{p}};
 \qquad
 g(x)
 :=\Big\langle\Big(\dashint_{
 U_{*,\varepsilon}(x)}
 |f|^{2}\Big)^{\frac{p}{2}}\Big\rangle^{\frac{1}{p}},
\end{equation*}
and
\begin{equation*}
 W_B(x):=\Big\langle
 \Big(\dashint_{U_{*,\varepsilon}(x)}
 |\nabla w_\varepsilon|^2\Big)^{\frac{p}{2}}\Big\rangle^{\frac{1}{p}};
 \qquad
 V_B(x):=\Big\langle\Big(\dashint_{
 U_{*,\varepsilon}(x)}
 |\nabla v_\varepsilon|^2\Big)^{\frac{p}{2}}\Big\rangle^{\frac{1}{p}}.
\end{equation*}
Thus we have $F\leq W_B+V_B$ due to the relationship
$u_\varepsilon = w_\varepsilon + v_\varepsilon$ in $\Omega$.
According to the preconditions stated in
Lemma $\ref{shen's lemma2}$ (taking $\omega=1$ here), we have to establish
the following two estimates\footnote{
To make the estimates consistent
with the left-hand side of
the estimates $\eqref{pri:3.3a}$ and $\eqref{pri:3.3b}$
in Lemma $\ref{shen's lemma2}$, one can replace
$B$ in $\eqref{f:14},\eqref{f:13}$ with $16B$, correspondingly,
since the concrete form of $V_B$ and $W_B$ is independent of $r_B$ here.}:
\begin{subequations}
\begin{align}
& \dashint_{\frac{1}{2}B\cap\Omega}W_B^{p}
 \lesssim \dashint_{7B\cap\Omega}
 g^{p};
\label{f:14} \\
& \Big(\dashint_{\frac{1}{8}B\cap\Omega}V_B^{q_1}\Big)^{\frac{p}{q_1}}
 \lesssim \dashint_{7B\cap\Omega}(F^{p}+g^{p}).
\label{f:13}
\end{align}
\end{subequations}

Then, a family of local estimates like $\eqref{f:3.1x}$ can be established
by applying Lemma $\ref{shen's lemma2}$ (taking $\omega=1$ here),
which together with the covering argument leads to the
stated estimate $\eqref{pri:9-2}$
in the case of $p\leq q<\infty$.
%Although
%we merely derive the estimate $\eqref{pri:9-1}$ within $A_1$ for now,
%taking the weight $\omega=1$ simply offers us the desired
%estimate $\eqref{pri:9-2}$, which benefits in the second part of the proof.
Therefore, the remainder of the proof is mainly to demonstrate the estimates $\eqref{f:14}$ and $\eqref{f:13}$.

\medskip
\noindent
\textbf{Step 2.} Show the estimate $\eqref{f:14}$.
It is reduced to showing
\begin{equation}\label{f:14-1}
\dashint_{\frac{1}{2}B\cap\Omega}
\Big(\dashint_{U_{*,\varepsilon}(x)}
|\nabla w_\varepsilon|^2\Big)^{\frac{p}{2}}
 \lesssim \dashint_{7B\cap\Omega}
\Big(\dashint_{U_{*,\varepsilon}(x)}
|f|^{2}\Big)^{\frac{p}{\bar{2}}}.
\end{equation}
The main work is to analyze the size of
support set of $\dashint_{U_{*,\varepsilon}(x)} fI_{2B\cap\Omega}$,
and we discuss it by two cases: (1) $r_B\geq \frac{\varepsilon}{4}\chi_{*}(x_B/\varepsilon)$;
(2) $r_B < \frac{\varepsilon}{4}\chi_{*}(x_B/\varepsilon)$.
We start from the first case, and notice that if
$\dashint_{U_{*,\varepsilon}(x)}
fI_{2B\cap\Omega}\not=0$, then there must
hold $|x-x_B|-\varepsilon\chi_{*}(x/\varepsilon)<2r_B$ due to
the geometry fact. In this respect, we conclude that
\begin{equation}\label{f:4.1}
 |x-x_B|\leq 2r_B+ \varepsilon \chi_{*}(x/\varepsilon)
 \leq  2r_B+ \varepsilon \chi_{*}(x_B/\varepsilon)
 + |x-x_B|/L \quad
 \Rightarrow \quad
 |x-x_B| \leq 7r_B.
\end{equation}
Hence, with the help of Proposition $\ref{P:10}$ we have
\begin{equation}\label{f:4.2}
\begin{aligned}
\dashint_{B\cap\Omega}
\Big(\dashint_{U_{*,\varepsilon}(x)}
|\nabla w_\varepsilon|^2\Big)^{\frac{p}{2}}
&\leq \frac{1}{|B|}
\int_{\Omega}
\Big(\dashint_{U_{*,\varepsilon}(x)}
|\nabla w_\varepsilon|^2\Big)^{\frac{p}{2}}\\
&\lesssim^{\eqref{A}}
\frac{1}{|B|}
\int_{\Omega}
\Big(\dashint_{U_{*,\varepsilon}(x)}
|fI_{2B\cap\Omega}|^2\Big)^{\frac{p}{2}}
\lesssim^{\eqref{f:4.1}} \dashint_{7B\cap\Omega}
\Big(\dashint_{U_{*,\varepsilon}(x)}
|f|^2\Big)^{\frac{p}{2}}.
\end{aligned}
\end{equation}

Now, we proceed to deal with the case (2)
$r_B < \frac{\varepsilon}{4}\chi_{*}(x_B/\varepsilon)$.
For any $x_0\in\frac{1}{2}B\cap\Omega$,
by noting $2B\subset B_{*,\varepsilon}(x_0)$ and using
the energy estimate, we have
\begin{equation}\label{pri:11A}
 \dashint_{U_{*,\varepsilon}(x_0)} |\nabla w_\varepsilon|^2
 \lesssim
 \dashint_{U_{*,\varepsilon}(x_0)} |f|^{2}.
\end{equation}
Then, appealing to Lemma $\ref{lemma:9}$ (and the fact
$r_B<(\frac{2L}{2L-1})\frac{\varepsilon}{4}\chi_{*}(x_0/\varepsilon)$),
it follows that
\begin{equation*}
\Big(\dashint_{U_{*,\varepsilon}(x_0)} |\nabla w_\varepsilon|^2\Big)^{\frac{p}{2}}
\lesssim
\Big(\dashint_{U_{*,\varepsilon}(x_0)} |f|^{2}
\Big)^{\frac{p}{2}}
\lesssim^{\eqref{pri:3.5-1}}
\dashint_{5B\cap\Omega}
\Big(\dashint_{U_{*,\varepsilon}(x)} |f|^{2}\Big)^{\frac{p}{2}}.
\end{equation*}
Integrating both sides above with respect to $x_0\in\frac{1}{2}B\cap\Omega$ yields the estimate $\eqref{f:14-1}$ in such the case. This together with
$\eqref{f:4.2}$ completes the whole argument of the estimate $\eqref{f:14-1}$, and finally taking expectation on both
sides of $\eqref{f:14-1}$ offers us the stated estimate $\eqref{f:14}$.

\medskip
\noindent
\textbf{Step 3.} Establish the estimate $\eqref{f:13}$.
By using Minkowski's inequality, H\"older's inequality
and Lemmas $\ref{lemma:14}$
in the order, we obtain
\begin{equation}\label{f:2.13}
\begin{aligned}
\Big(\dashint_{\frac{1}{8}B\cap\Omega}
V_B^{q_1}\Big)^{\frac{p}{q_1}}
&= \bigg(\dashint_{\frac{1}{8}B\cap\Omega}
\Big\langle\Big(\dashint_{U_{*,\varepsilon}(x)}
|\nabla v_\varepsilon|^2\Big)^{\frac{p}{2}}\Big\rangle^{\frac{q_1}{p}}
dx\bigg)^{\frac{p}{q_1}}
\leq
\bigg\langle
\Big(\dashint_{\frac{1}{8}B\cap\Omega}
\Big(\dashint_{
U_{*,\varepsilon}(x)}|\nabla v_\varepsilon|^2
\Big)^{\frac{q_1}{2}}dx\Big)^{\frac{p}{q_1}}\bigg\rangle\\
%&\leq
%\bigg\langle
%\Big(\dashint_{\textcolor[rgb]{1.00,0.00,0.00}{\frac{1}{8}B\cap\Omega}}
%\Big(\dashint_{\textcolor[rgb]{1.00,0.00,0.00}{U_{*,\varepsilon}(x)}}
%|\nabla v_\varepsilon|^2\Big)^{\frac{q_1\gamma}{2}}\Big)^{\frac{p}{q_1\gamma}}
%\bigg\rangle\Big(\dashint_{B}\omega^{\gamma'}\Big)^{\frac{p}{q_1\gamma'}}\\
&\lesssim^{\eqref{pri:3.8}}
\dashint_{\frac{1}{2}B\cap\Omega}\Big\langle
\Big(\dashint_{U_{*,\varepsilon}(x)}
|\nabla v_\varepsilon|^2\Big)^{\frac{p}{2}}\Big\rangle,
\end{aligned}
\end{equation}
which together with the fact $u_\varepsilon = w_\varepsilon + v_\varepsilon$ in $\Omega$ and the estimate $\eqref{f:14-1}$
leads to the desired estimate $\eqref{f:13}$. This ends the first part of the proof.

\medskip
\noindent
\textbf{Step 4.} Show the estimate $\eqref{pri:9-1}$.
Let $0<q_0-1\ll 1$, and $q_1\in(1,\infty)$ is preferred large sufficiently. On account of
Lemma $\ref{shen's lemma2}$, for any $q_0<q<q_1$ and $\omega\in A_{q/q_0}$, it suffices to establish
\begin{subequations}
\begin{align}
& \dashint_{\frac{1}{2}B\cap\Omega}W_B^{q_0}
 \lesssim \dashint_{7B\cap\Omega}
 g^{q_0};
\label{f:14-2} \\
& \Big(\dashint_{\frac{1}{8}B\cap\Omega}V_B^{q_1}\omega\Big)^{\frac{q_0}{q_1}}
 \lesssim \Big(\dashint_{7B\cap\Omega}(F^{q_0}+g^{q_0})\Big)
 \Big(\dashint_{B}\omega\Big)^{\frac{q_0}{q_1}}.
\label{f:13-2}
\end{align}
\end{subequations}
Then, the same argument given for the estimate
$\eqref{f:3.18}$ and $\eqref{f:3.38}$ leads to
\begin{equation*}
\begin{aligned}
\dashint_{\Omega}F^q\omega
&\lesssim
\Big(\dashint_{\Omega}F^{q_0}\Big)^{\frac{q}{q_0}}
\Big(\dashint_{\Omega}\omega\Big)
+\dashint_{\Omega}g^{q}\omega\\
&\lesssim^{\eqref{pri:9-2}}
\Big(\dashint_{\Omega}g^{q_0}\Big)^{\frac{q}{q_0}}
\Big(\dashint_{\Omega}\omega\Big)
+\dashint_{\Omega}|F|^{q}\omega
\lesssim \dashint_{\Omega}|F|^{q}\omega,
\end{aligned}
\end{equation*}
where the last inequality above follows from H\"older's inequality and
the property of $A_{q/q_0}$ class
(see the argument given for $\eqref{f:3.38}$). By noting the condition $0<q_0-1\ll 1$,
we have proved the estimate $\eqref{pri:9-1}$ for any $\omega\in A_{q}$.

\medskip
\noindent
\textbf{Step 5.} Show the estimate $\eqref{f:14-2}$, which is parallel to Step 2. For the case $r_B\geq\frac{\varepsilon}{4}\chi_*(x_B/\varepsilon)$,
with the help of the estimate $\eqref{pri:9-2}$ and the fact $\eqref{f:4.1}$, we have
\begin{equation}\label{f:2.12}
\begin{aligned}
\dashint_{B\cap\Omega}
\Big\langle\Big(\dashint_{U_{*,\varepsilon}(x)}
|\nabla w_\varepsilon|^2
\Big)^{\frac{p}{2}}\Big\rangle^{\frac{q_0}{p}}dx
&\lesssim \frac{1}{|B|}
\int_{\Omega}
\Big\langle\Big(\dashint_{U_{*,\varepsilon}(x)}
|\nabla w_\varepsilon|^2\Big)^{\frac{p}{2}}\Big\rangle^{\frac{q_0}{p}}
dx\\
&\lesssim^{\eqref{pri:9-2}}
\frac{1}{|B|}
\int_{\Omega}
\Big\langle\Big(\dashint_{U_{*,\varepsilon}(x)}
|fI_{2B\cap\Omega}|^2\Big)^{\frac{p}{2}}\Big\rangle^{\frac{q_0}{p}}dx
\lesssim^{\eqref{f:4.1}} \dashint_{7B\cap\Omega}
F^{q_0}.
\end{aligned}
\end{equation}
Then we turn to the case $r_B<\frac{\varepsilon}{4}\chi_*(x_B/\varepsilon)$.
For any $x_0\in\frac{1}{2}B\cap\Omega$, in view of the estimate
$\eqref{pri:11A}$, we have
\begin{equation*}
\Big\langle\Big(\dashint_{U_{*,\varepsilon}(x_0)} |\nabla w_\varepsilon|^2
\Big)^{\frac{p}{2}}\Big\rangle^{\frac{q_0}{p}}
 \lesssim
\Big\langle\Big(\dashint_{U_{*,\varepsilon}(x_0)} |f|^{2}
\Big)^{\frac{p}{2}}\Big\rangle^{\frac{q_0}{p}},
\end{equation*}
and integrating both sides above with respect to $x_0\in\frac{1}{2}B\cap\Omega$ also leads to \eqref{f:2.12}. This completes the proof for the estimate $\eqref{f:14-2}$.

\medskip
\noindent
\textbf{Step 6.} Show the estimate $\eqref{f:13-2}$,
which is parallel to Step 3.
Let $\gamma,\gamma'$ be such that $1/\gamma+1/\gamma'=1$
with $0<\gamma'-1\ll 1$. By using H\"older's inequality,
the reverse H\"older property of $A_q$-class and Minkowski's inequality in
the order, we obtain
\begin{equation}\label{f:3.41}
\begin{aligned}
\Big(\dashint_{\frac{1}{8}B\cap\Omega}V_B^{q_1}\omega\Big)^{\frac{1}{q_1}}
&\leq
\bigg(\dashint_{\frac{1}{8}B\cap\Omega}
\Big\langle\big(
\dashint_{U_{*,\varepsilon}(x)}|\nabla v_\varepsilon|^2
\big)^{\frac{p}{2}}\Big\rangle^{\frac{q_1\gamma}{p}}dx
\bigg)^{\frac{1}{q_1\gamma}}\Big(\dashint_{B}
\omega^{\gamma'}\Big)^{\frac{1}{q_1\gamma'}}\\
&\lesssim^{\eqref{pri:R-1}}
\bigg(\dashint_{\frac{1}{8}B\cap\Omega}
\Big\langle\big(
\dashint_{U_{*,\varepsilon}(x)}|\nabla v_\varepsilon|^2
\big)^{\frac{p}{2}}\Big\rangle^{q_1\gamma}dx
\bigg)^{\frac{1}{q_1\gamma p}}\Big(\dashint_{B}
\omega\Big)^{\frac{1}{q_1}}\\
&\lesssim
\bigg\langle
\bigg(\dashint_{\frac{1}{8}B\cap\Omega}
\Big(
\dashint_{U_{*,\varepsilon}(x)}
|\nabla v_\varepsilon|^2\Big)^{\frac{pq_1\gamma}{2}}
dx\bigg)^{\frac{1}{q_1\gamma}}\bigg\rangle^{\frac{1}{p}}
\Big(\dashint_{B}
\omega\Big)^{\frac{1}{q_1}}
\end{aligned}
\end{equation}
On account of the estimate $\eqref{pri:3.8}$,
the right-hand side of $\eqref{f:3.41}$ is controlled by
\begin{equation}\label{f:3.42}
\bigg\langle
\bigg(\dashint_{\frac{1}{2}B\cap\Omega}
\Big(
\dashint_{U_{*,\varepsilon}(x)}
|\nabla v_\varepsilon|^2\Big)^{\frac{1}{2}}
dx\bigg)^{p}\bigg\rangle^{\frac{1}{p}}
\Big(\dashint_{B}
\omega\Big)^{\frac{1}{q_1}}
\leq
\bigg(\dashint_{\frac{1}{2}B\cap\Omega}
\Big\langle\Big(
\dashint_{U_{*,\varepsilon}(x)}
|\nabla v_\varepsilon|^2\Big)^{\frac{p}{2}}
\Big\rangle^{\frac{1}{p}}
dx\bigg)
\Big(\dashint_{B}
\omega\Big)^{\frac{1}{q_1}},
\end{equation}
where we also employ Minkowski's inequality for the inequality above.
Plugging $\eqref{f:3.42}$ back into $\eqref{f:3.41}$,
\begin{equation*}
\begin{aligned}
\Big(\dashint_{\frac{1}{8}B\cap\Omega}V_B^{q_1}\omega\Big)^{\frac{1}{q_1}}
&\lesssim
\bigg(\dashint_{\frac{1}{2}B\cap\Omega}
V_B
\bigg)
\Big(\dashint_{B}
\omega\Big)^{\frac{1}{q_1}}
\lesssim
\Big(\dashint_{\frac{1}{2}B\cap\Omega}
V_B^{q_0}
\Big)^{\frac{1}{q_0}}
\Big(\dashint_{B}
\omega\Big)^{\frac{1}{q_1}}\\
&\lesssim
\Big(\dashint_{\frac{1}{2}B\cap\Omega}
(F^{q_0}+W_B^{q_0})
\Big)^{\frac{1}{q_0}}
\Big(\dashint_{B}
\omega\Big)^{\frac{1}{q_1}}
\lesssim^{\eqref{f:14-2}}
\Big(\dashint_{7B\cap\Omega}
(F^{q_0}+g^{q_0})
\Big)^{\frac{1}{q_0}}
\Big(\dashint_{B}
\omega\Big)^{\frac{1}{q_1}}
\end{aligned}
\end{equation*}
which consequently leads to $\eqref{f:13-2}$.
This ends the whole proof.
\end{proof}

\medskip

\begin{lemma}\label{lemma:12}
Let $\chi_{*}$ be given as in Corollary $\ref{cor:1}$ with the $\frac{1}{L}$-Lipschitz continuity.
Assume that $F\in C^{\infty}_0(\Omega;L^\infty_{\langle\cdot\rangle})$\footnote{
It represents the space of smooth functions with compact support and values taken in $L^\infty_{\langle\cdot\rangle}$,
and the subscript of $L^\infty_{\langle\cdot\rangle}$ stresses that its probability measure is introduced from
the ensemble $\langle\cdot\rangle$.},
$1\leq p\leq q<\infty$ and $1\leq s<\infty$.
Then, for any $p_1>p$,  there hold
\begin{equation}\label{pri:13}
\bigg(\int_{\Omega}\Big\langle\big(
\dashint_{U_{*,\varepsilon}(x)}|F|^s\big)^{\frac{p}{s}}
\Big\rangle^{\frac{q}{p}}\omega(x) dx\bigg)^{\frac{1}{q}}
\lesssim
\bigg(\int_{\Omega}\Big\langle\big(
\dashint_{U_\varepsilon(x)}|F|^s\big)^{\frac{p_1}{s}}
\Big\rangle^{\frac{q}{p_1}}\dashint_{B_{\varepsilon}(x)}\omega dx\bigg)^{\frac{1}{q}},
\end{equation}
and
\begin{equation}\label{pri:14}
\bigg(\int_{\Omega}\Big\langle\big(
\dashint_{U_{\varepsilon}(x)}|F|^s\big)^{\frac{p}{s}}
\Big\rangle^{\frac{q}{p}}
\dashint_{B_{\varepsilon}(x)}\omega dx\bigg)^{\frac{1}{q}}
\lesssim
\bigg(\int_{\Omega}\Big\langle\big(
\dashint_{U_{*,\varepsilon}(x)}|F|^s\big)^{\frac{p_1}{s}}
\Big\rangle^{\frac{q}{p_1}}\omega(x) dx\bigg)^{\frac{1}{q}}
\end{equation}
for any $\omega\in A_q$, where the multiplicative constant depends
on $\lambda,\lambda_1,\lambda_2,d,M_0,s,p_1,p,q$ and $[\omega]_{A_q}$.
\end{lemma}

\begin{proof}
\vspace{0.2cm}
\noindent
\textbf{Arguments for $\eqref{pri:13}$.}
The main proof is similar to that given in \cite{Duerinckx-Otto20}, and we provide a proof with details for the reader's convenience. The difficulty is how to estimate the average integral over $D_{*,\varepsilon}(x)$, while the basic idea is to use the conditional expectation.
This requires us to decompose the random radius $\chi_*$ into dyadic ones, and then for each
dyadic radius, the related average integral owns a polynomial growth with respect to
this radius. Then, on account of $\beta$-moment estimates of the random radius $\chi_*$ for
any $\beta<\infty$,
we use the loss of random index to trade the boundedness (see Step 3).
By rescaling arguments, it suffices to show
\begin{equation}\label{pri:a-1}
\int_{\Omega/\varepsilon}\Big\langle\big(
\dashint_{U_{*}(y)}|\tilde{F}|^s\big)^{\frac{p}{s}}
\Big\rangle^{\frac{q}{p}}\tilde{\omega}(y) dy
\lesssim_{\lambda,\lambda_1,\lambda_2,d,M_0,
[\omega]_{A_q},s,p_1,p,q}
\int_{\Omega/\varepsilon}\Big\langle\big(
\dashint_{U_1(y)}|\tilde{F}|^s\big)^{\frac{p_1}{s}}
\Big\rangle^{\frac{q}{p_1}}\dashint_{B_1(y)}\tilde{\omega} dy,
\end{equation}
where $x=\varepsilon y\in \Omega$, and we use the following notation
throughout the proof: $\tilde{F}(y) = F(\varepsilon y)=F(x)$,
and $\tilde{\omega}(y)=\omega(\varepsilon y)$, as well as
$U_{*}(y):=B_{\chi_{*}(y)}(y)\cap(\Omega/\varepsilon)$ and
$U_R(y):=B_R(y)\cap(\Omega/\varepsilon)$.
We now split the proof into three steps to finish the argument for
$\eqref{pri:13}$.

\medskip
\noindent
\textbf{Step 1.}
We start from establishing the following deterministic radius's growth estimate, i.e.,
for any $R\geq 1$, there holds
\begin{equation}\label{f:16}
\begin{aligned}
\int_{\Omega/\varepsilon}\Big\langle
\big(\dashint_{U_R(y)}|\tilde{F}|^{s}\big)^{\frac{p}{s}} \Big\rangle^{\frac{q}{p}}\tilde{\omega} dy
&\lesssim_{d,M_0,[\omega]_{A_q}}
R^{d(\frac{q}{p}-\frac{q}{s})_{+}+d(q-1)}
\int_{\Omega/\varepsilon}
\Big\langle
\big(\dashint_{U_1(y)}|\tilde{F}|^s\big)^{\frac{p}{s}}
\Big\rangle^{\frac{q}{p}} \dashint_{B_{1}(y)}\tilde{\omega}dy,
\end{aligned}
\end{equation}
in which $1\leq p\leq q<\infty$, and $(\frac{q}{p}-\frac{q}{s})_{+}
:=\max\{0,\frac{q}{p}-\frac{q}{s}\}$. We claim that
\begin{equation}\label{f:15}
\Big\langle\big(\dashint_{U_R(y)}|\tilde{F}|^{s}\big)^{\frac{p}{s}} \Big\rangle^{\frac{q}{p}}
\lesssim_{d,M_0}
R^{d(\frac{q}{p}-\frac{q}{s})_{+}}
\dashint_{U_R(y)}\Big\langle
\big(\dashint_{U_1(z)}|\tilde{F}|^s\big)^{\frac{p}{s}}
\Big\rangle^{\frac{q}{p}},
\end{equation}
which will be proved in Step 2.
%We recall the property of $A_q$ class (see \cite[pp.134]{Duoandikoetxea01}), i.e.,
%\begin{equation}
% \omega(Q)\big(\frac{|S|}{|Q|}\big)^{p}\leq C_0\omega(S),
%\qquad \forall S\subset Q,
%\end{equation}
%where \textcolor[rgb]{0.00,0.00,1.00}{we also recall $\omega(Q):=\int_{Q}\omega$,}
%and
By using Fubini's theorem, for any $R>0$ we have
\begin{equation}\label{f:17}
\begin{aligned}
 \int_{\Omega/\varepsilon}
 \Big(\dashint_{U_R(x)}f\Big) g(x) dx
 &= \int_{\Omega/\varepsilon}
 \Big(\dashint_{U_R(y)}g\Big) f(y) dy,
\end{aligned}
\end{equation}
where we note that $\text{supp}(f)\subset\Omega/\varepsilon$.
Therefore, it follows from the estimate $\eqref{f:15}$ and
Lemma $\ref{weight}$ that
\begin{equation*}
\begin{aligned}
\int_{\Omega/\varepsilon}\Big\langle\big(
\dashint_{U_R(y)}|\tilde{F}|^{s}\big)^{\frac{p}{s}} \Big\rangle^{\frac{q}{p}}\tilde{\omega} dy
&\lesssim_{d,M_0}
R^{d(\frac{q}{p}-\frac{q}{s})_{+}}
\int_{\Omega/\varepsilon}
\dashint_{U_R(y)}
\Big\langle
\big(\dashint_{U_1(z)}|\tilde{F}|^s\big)^{\frac{p}{s}}
\Big\rangle^{\frac{q}{p}}
\tilde{\omega} dy\\
&=^{\eqref{f:17}} R^{d(\frac{q}{p}-\frac{q}{s})_{+}}
\int_{\Omega/\varepsilon}
\Big\langle
\big(\dashint_{U_1(y)}|\tilde{F}|^s\big)^{\frac{p}{s}}
\Big\rangle^{\frac{q}{p}}
\dashint_{D_R(y)}\tilde{\omega} dy\\
&\lesssim_{d,M_0,[\omega]_{A_q}}^{\eqref{f:18},\eqref{f:18-0}}
R^{d(\frac{q}{p}-\frac{q}{s})_{+}+d(q-1)}
\int_{\Omega/\varepsilon}
\Big\langle
\big(\dashint_{U_1(y)}|\tilde{F}|^s\big)^{\frac{p}{s}}
\Big\rangle^{\frac{q}{p}} \dashint_{B_{1}(y)}\tilde{\omega}dy.
\end{aligned}
\end{equation*}

\medskip
\noindent
\textbf{Step 2.} Show the estimate $\eqref{f:15}$. For any $R\geq 1$, it suffices to establish
\begin{equation}\label{f:20}
\Big(\dashint_{U_R(y)}|\tilde{F}|^{s}\Big)^{\frac{p}{s}}
\lesssim_{d,M_0} R^{d(1-\frac{p}{s})_{+}}
\dashint_{U_R(y)}
\big(\dashint_{U_1(z)}|\tilde{F}|^s\big)^{\frac{p}{s}},
\end{equation}
since by taking $\langle\cdot\rangle^{\frac{q}{p}}$ on both sides above
we can immediately derive $\eqref{f:15}$. The estimate $\eqref{f:20}$ is demonstrated by two cases:

Case 1. For any $p\geq s$, it follows from
Lemma $\ref{lemma:9}$ and H\"older's inequality that
\begin{equation}\label{f:2.14}
\begin{aligned}
\Big(\dashint_{U_R(y)}|\tilde{F}|^s\Big)^{\frac{p}{s}}
\lesssim_{d,M_0}^{\eqref{pri:3.6-1}}
\Big(\dashint_{U_R(y)} \dashint_{U_1(z)}|\tilde{F}|^s\Big)^{\frac{p}{s}}
\lesssim
\dashint_{U_R(y)} \Big(\dashint_{U_1(z)}|\tilde{F}|^s\Big)^{\frac{p}{s}}.
\end{aligned}
\end{equation}

Case 2. For any $p\leq s$, we employ a covering argument and Lemma $\ref{lemma:9}$ to have
\begin{equation*}
\begin{aligned}
\Big(\int_{U_R(y)}|\tilde{F}|^s
\Big)^{\frac{p}{s}}
&\leq \Big(\sum_{i}
\int_{U_{\frac{1}{4}}(z_i)}
|\tilde{F}|^s\Big)^{\frac{p}{s}}
\lesssim_d \sum_{i}
\Big(\dashint_{U_{\frac{1}{4}}(z_i)}
|\tilde{F}|^s\Big)^{\frac{p}{s}}\\
&\lesssim_{d,M_0}^{\eqref{pri:3.5-1}}
\sum_{i} \dashint_{U_{\frac{1}{4}}(z_i)}
\big(\dashint_{U_1(z)}|\tilde{F}|^{s}\big)^{\frac{p}{s}}
\lesssim
\int_{U_R(y)}
\big(\dashint_{U_1(z)}|\tilde{F}|^{s}\big)^{\frac{p}{s}},
\end{aligned}
\end{equation*}
which together with $\eqref{f:2.14}$ leads to the stated estimate $\eqref{f:20}$.

\medskip
\noindent
\textbf{Step 3.} Show the estimate $\eqref{pri:a-1}$.
For the ease of statement, let $d\mu(y) := \tilde{w}dy$.
On account of the property of conditional expectations (and for any $\delta>0$ to be fixed later), we have
\begin{equation*}
\begin{aligned}
\int_{\Omega/\varepsilon}
\Big<\big(\dashint_{U_{*}(y)}
|\tilde{F}|^s\big)^{\frac{p}{s}}\Big>^{\frac{q}{p}}
d\mu(y)
&\lesssim_{d} \int_{\Omega/\varepsilon}
\Big<\sum_{n=1}^\infty I_{\{
2^{n}-1\leq \chi_*(y)<2^{n+1}-1\}}\big(\dashint_{U_{2^{n+1}}(y)}
|\tilde{F}|^s\big)^{\frac{p}{s}}\Big>^{\frac{q}{p}} d\mu(y)\\
&\lesssim \int_{\Omega/\varepsilon}
\Big<\sum_{n=1}^\infty 2^{-nd\delta}\big[\chi_*(y)\big]^{d\delta} I_{\{
2^{n}-1\leq \chi_*(y)<2^{n+1}-1\}}
\big(\dashint_{U_{2^{n+1}}(y)}
|\tilde{F}|^s\big)^{\frac{p}{s}}\Big>^{\frac{q}{p}}d\mu(y)\\
&\lesssim \int_{\Omega/\varepsilon}
\Big(\sum_{n=1}^\infty 2^{-nd\delta}
\Big<\underbrace{\big[\chi_*(y)\big]^{d\delta} I_{\{
2^{n}-1\leq \chi_*(y)<2^{n+1}-1\}}\big(\dashint_{U_{2^{n+1}}(y)}
|\tilde{F}|^s\big)^{\frac{p}{s}}}_{f_n}\Big>\Big)^{\frac{q}{p}}
d\mu(y).
\end{aligned}
\end{equation*}
By H\"{o}lder's inequality (we split
$nd\delta = nd\delta(\frac{p}{q}+\frac{q-p}{q})$ and choose $\delta$
to be such that
$\sum_{n=0}^{\infty}2^{-nd\delta}=\delta^{-1}$ for some $\delta\in(0,1)$),
the last line above is controlled by
\begin{equation*}
\begin{aligned}
\int_{\Omega/\varepsilon}
\Big(
&\sum_{n=1}^\infty 2^{-nd\delta}
\big<f_n\big>^{\frac{q}{p}}\Big)
\Big(\sum_{n=1}^\infty 2^{-nd\delta}\Big)^{\frac{q}{p}-1} d\mu
\lesssim
\delta^{1-\frac{p}{q}}
\sum_{n=1}^\infty
2^{-nd\delta}\int_{\Omega/\varepsilon} \big<f_n\big>^{\frac{q}{p}}d\mu
\\
&\lesssim\delta^{1-\frac{p}{q}}
\sum_{n=1}^\infty
2^{nd\delta(\frac{q}{p}-1)}
 \int_{\Omega/\varepsilon}
\Big<
\big(\dashint_{U_{2^{n+1}(y)}}
I_{\{
2^{n}-1\leq
\chi_*(y)<2^{n+1}-1\}}|\tilde{F}|^s\big
)^{\frac{p}{s}}\Big>^{\frac{q}{p}}
d\mu(y)\\
&\lesssim_{d,M_0,[\omega]_{A_q}}^{\eqref{f:16}}
\delta^{1-\frac{p}{q}}
\sum_{n=1}^\infty
(2^{ndq})^{(\frac{1}{p}-\frac{1}{s})_{+}
+\delta(\frac{1}{p}-\frac{1}{q})
+1-\frac{1}{q}}
\int_{\Omega/\varepsilon}
\Big<
\big(\dashint_{U_1(y)}
I_{\{
2^{n}-1\leq
\chi_*(y)<2^{n+1}-1\}}|\tilde{F}|^s\big
)^{\frac{p}{s}}\Big>^{\frac{q}{p}}
\dashint_{B_{1}(y)}\tilde{\omega}dy.
\end{aligned}
\end{equation*}
Moreover, we notice that for
any $y\in\Omega/\varepsilon$,
\begin{equation}\label{loss}
\begin{aligned}
\Big<
\big(\dashint_{U_1(y)}
I_{\{
2^{n}-1\leq
\chi_*(y)<2^{n+1}-1\}}|\tilde{F}|^s\big
)^{\frac{p}{s}}\Big>
&=\Big<I_{\{
2^{n}-1\leq
\chi_*(y)<2^{n+1}-1\}}
\big(\dashint_{U_1(y)}
|\tilde{F}|^s\big
)^{\frac{p}{s}}\Big>\\
&\leq \big\langle I_{\chi_*(y)\geq 2^n-1}
\big\rangle^{\frac{p_1-p}{p_1}}
\Big<
\big(\dashint_{U_1(y)}
|\tilde{F}|^s\big
)^{\frac{p_1}{s}}\Big>^{\frac{p}{p_1}}.
\end{aligned}
\end{equation}
For any $N_0\geq 1$ (to be fixed later), by using Chebyshev’s inequality, one can derive that
\begin{equation*}
\big\langle I_{\chi_*(y)\geq 2^n-1}
\big\rangle^{\frac{p_1-p}{p_1}}
= \mathbb{P}({\chi}_{*}(y)\geq 2^n-1)^{\frac{p_1-p}{p_1}}
\leq   \Big(\frac{\langle|\chi_{*}(y)|^{N_0}\rangle}{(2^n-1)^{N_0}}\Big)^{\frac{p_1-p}{p_1}}
\lesssim_{\lambda,\lambda_1,\lambda_2,d,N_0} 2^{N_0} 2^{-nN_0(\frac{p_1-p}{p_1})},
\end{equation*}
where $\mathbb{P}$ is the probability measure associated with the ensemble
$\langle\cdot\rangle$. Set $N_1:=
dq\big[(\frac{1}{p}-\frac{1}{s})_{+}
+\delta(\frac{1}{p}-\frac{1}{q})
+1-\frac{1}{q}\big]$
and
the above computations consequently yield the desired estimate $\eqref{pri:a-1}$, i.e.,
\begin{equation*}
\begin{aligned}
\int_{\Omega/\varepsilon}  \Big<\big(\dashint_{U_{*}(y)}
|\tilde{F}|^s\big)^{\frac{p}{s}}\Big>^{\frac{q}{p}}
\tilde{\omega}
dy
&\lesssim\delta^{1-\frac{p}{q}}
\sum_{n=0}^{\infty}2^{n[N_1
- N_0(\frac{(p_1-p)q}{p_1p})]}
\int_{\Omega/\varepsilon}
\Big<
\big(\dashint_{U_1(y)}
|\tilde{F}|^s\big
)^{\frac{p_1}{s}}\Big>^{\frac{q}{p_1}}
\dashint_{B_{1}(y)}\tilde{\omega} dy\\
&\lesssim_{\lambda,\lambda_1,\lambda_2,d,M_0,
[\omega]_{A_q},
s,p_1,p,q}
\int_{\Omega/\varepsilon}
\Big<
\big(\dashint_{U_1(y)}
|\tilde{F}|^s\big
)^{\frac{p_1}{s}}\Big>^{\frac{q}{p_1}}
\dashint_{B_{1}(y)}\tilde{\omega} dy,
\end{aligned}
\end{equation*}
where the second inequality follows by fixing
$N_0>\frac{N_1p_1p}{(p_1-p)q}$,
and this implies $\eqref{pri:13}$ by a rescaling argument (
where we mention that $[\omega]_{A_q}$ is a scaling invariant quantity
due to $\eqref{f:18-0}$).

\vspace{0.2cm}
\noindent
\textbf{Arguments for $\eqref{pri:14}$.}
The idea is similar to that given for the estimate $\eqref{pri:13}$. We
therefore focus on verifying the following deterministic radius’s growth estimate.
For any $1\leq p\leq q<\infty$ with $1\leq s<\infty$ and $R\geq 1$,
we first need to establish
\begin{equation}\label{f:21}
\begin{aligned}
\bigg(\int_{\Omega/\varepsilon}
\Big\langle
\Big(\dashint_{U_{2R}(y)}|\tilde{F}|^s
\Big)^{\frac{p}{s}}
\Big\rangle^{\frac
{q}{p}}
\tilde{\omega}(y)dy\bigg)^{\frac{1}
{q}}
\gtrsim
R^{-d(\frac{1}{s}-\frac{1}{q})_{+}
+\frac{d\delta_0}{q}}
\bigg(\int_{\Omega/\varepsilon}
\Big\langle
\big(\dashint_{U_1(y)}|\tilde{F}|^s\big)^{\frac{p}{s}}
\Big\rangle^{\frac
{q}{p}}
\dashint_{B_1(y)} \tilde{\omega} dy\bigg)^{\frac{1}{q}},
\end{aligned}
\end{equation}
where the index $\delta_0>0$ is usually very small,
coming from the reverse H\"older property of
$A_p$-weight (see Lemma $\ref{weight}$),
and the multiplicative constant
above relies on $d,M_0,q$, and $[\omega]_{A_q}$.
In fact, the proof of $\eqref{f:21}$ is more complicated than its counterpart $\eqref{f:16}$,
and therefore we divide it into three cases to complete the whole argument.

\medskip
\noindent
\textbf{Case 1.}  For $p\leq q\leq s$,
we show the estimate
\begin{equation}\label{f:22}
\int_{\Omega/\varepsilon}
\Big\langle
\Big(\dashint_{U_{2R}(y)}|\tilde{F}|^s
\Big)^{\frac{p}{s}}
\Big\rangle^{\frac{q}{p}}
\tilde{\omega}(y)dy
\gtrsim R^{\delta_0 d}
\int_{\Omega/\varepsilon}
\Big\langle
\big(\dashint_{U_1(y)}|\tilde{F}|^s\big)^{\frac{p}{s}}
\Big\rangle^{\frac{q}{p}}
\dashint_{B_1(y)} \tilde{\omega} dy.
\end{equation}
%We start from introducing the reverse H\"older's property of $A_q$ class, i.e., there exists $\delta_0>0$ such that
%%\begin{equation}\label{f:19}
%% \omega(S)\leq C_0\big(\frac{|S|}{|Q|}\big)^{\delta_0}\omega(Q),
%%\qquad \forall S\subset Q
%%\end{equation}
%(see e.g. \cite[Corollary 7.6]{Duoandikoetxea01}).
By using Lemma $\ref{weight}$ and Fubini's theorem, we have
\begin{equation}\label{f:2.15}
\begin{aligned}
\int_{\Omega/\varepsilon}
\Big\langle
\big(\dashint_{U_1(y)}|\tilde{F}|^s\big)^{\frac{p}{s}}
\Big\rangle^{\frac{q}{p}}
\dashint_{B_1(y)} \tilde{\omega} dy
&\lesssim^{\eqref{f:19}} R^{-\delta_0 d}
\int_{\Omega/\varepsilon}
\Big\langle
\big(\dashint_{U_1(y)}|\tilde{F}|^s\big)^{\frac{p}{s}}
\Big\rangle^{\frac{q}{p}}
\dashint_{B_R(y)} \tilde{\omega} dy \\
&\leq^{\eqref{f:17}}
 R^{-\delta_0 d}\int_{\Omega/\varepsilon}
\dashint_{U_R(y)}
\Big\langle
\big(\dashint_{U_1(z)}|\tilde{F}|^s\big)^{\frac{p}{s}}
\Big\rangle^{\frac{q}{p}}
\tilde{\omega}(y)dy.
\end{aligned}
\end{equation}

Recalling the notation $d\mu(y)=\tilde{\omega}(y)dy$,
and then employing H\"older's inequality, Minkowski's inequality,
and Lemma $\ref{lemma:9}$, in the order, we obtain
\begin{equation}\label{f:2.16}
\begin{aligned}
&\int_{\Omega/\varepsilon}
\dashint_{U_R(y)}
\Big\langle
\big(\dashint_{U_1(z)}
|\tilde{F}|^s\big)^{\frac{p}{s}}
\Big\rangle^{\frac{q}{p}}
d\mu(y)
\leq
\int_{\Omega/\varepsilon}
\Big(\dashint_{U_R(y)}
\Big\langle
\big(\dashint_{U_1(z)}
|\tilde{F}|^s\big)^{\frac{p}{s}}
\Big\rangle^{\frac{s}{p}}\Big)^{\frac{q}{s}}
d\mu(y) \\
&\leq
\int_{\Omega/\varepsilon}
\Big\langle
\Big(\dashint_{U_R(y)}
\dashint_{U_1(z)}|\tilde{F}|^s
\Big)^{\frac{p}{s}}
\Big\rangle^{\frac{q}{p}}
d\mu(y)
\lesssim^{\eqref{pri:3.6-1}}
\int_{\Omega/\varepsilon}
\Big\langle
\Big(\dashint_{U_{2R}(y)}|\tilde{F}|^s
\Big)^{\frac{p}{s}}
\Big\rangle^{\frac{q}{p}}
d\mu(y).
\end{aligned}
\end{equation}
Combining the estimates $\eqref{f:2.15}$ and $\eqref{f:2.16}$, we have
the stated estimate $\eqref{f:22}$.

\medskip
\noindent
\textbf{Case 2.} For $p\leq s\leq q$, we show the estimate
\begin{equation}\label{f:23}
\int_{\Omega/\varepsilon}
\Big\langle
\Big(\dashint_{U_{2R}(y)}|\tilde{F}|^s
\Big)^{\frac{p}{s}}
\Big\rangle^{\frac{q}{p}}
\tilde{\omega}dy
\gtrsim
R^{-\frac{dq}{s}+d+\delta_0d}\int_{\Omega/\varepsilon}
\Big\langle
\big(\dashint_{U_1(y)}|\tilde{F}|^s\big)^{\frac{p}{s}}
\Big\rangle^{\frac{q}{p}}
\dashint_{B_1(y)} \tilde{\omega} dy.
\end{equation}
For any $y\in\Omega/\varepsilon$,
by applying Lemma $\ref{lemma:9}$ and Minkowski's inequality, we first arrive at
\begin{equation}\label{f:2.18}
\begin{aligned}
\Big\langle
\Big(\dashint_{U_{2R}(y)}|\tilde{F}|^s
\Big)^{\frac{p}{s}}
\Big\rangle^{\frac{q}{p}}
&\gtrsim
\Big\langle
\Big(\dashint_{U_R(y)}
\dashint_{U_1(z)}|\tilde{F}|^s
\Big)^{\frac{p}{s}}
\Big\rangle^{\frac{q}{p}}
\geq
\Big(\dashint_{U_R(y)}
\Big\langle\big(\dashint_{U_1(z)}
|\tilde{F}|^s\big)^{\frac{p}{s}}
\Big\rangle^{\frac{s}{p}}
\Big)^{\frac{q}{s}} =: I.
\end{aligned}
\end{equation}
Then, by covering the region $D_R(y)$ by a family of balls (or half-balls)
denoted by $\{D_1(y_i)\}$, and
using H\"older's inequality of $l^p$-version and $L^p$-version in
the order, we obtain
\begin{equation}\label{f:2.17}
\begin{aligned}
R^{\frac{dq}{s}}I
&\gtrsim_d\Big(\sum_i\int_{U_1(y_i)}
\Big\langle\big(\dashint_{U_1(z)}
|\tilde{F}|^s\big)^{\frac{p}{s}}
\Big\rangle^{\frac{s}{p}}
\Big)^{\frac{q}{s}}
\geq
\sum_i\Big(\int_{U_1(y_i)}
\Big\langle\big(\dashint_{U_1(z)}
|\tilde{F}|^s\big)^{\frac{p}{s}}
\Big\rangle^{\frac{s}{p}}
\Big)^{\frac{q}{s}}\\
&\geq
\sum_i\Big(\int_{U_1(y_i)}
\Big\langle\big(\dashint_{U_1(z)}
|\tilde{F}|^s\big)^{\frac{p}{s}}
\Big\rangle
\Big)^{\frac{q}{p}}
\sim
\int_{U_R(y)}\Big\langle
\int_{U_1(z)}
\big(\dashint_{U_1(z')}|\tilde{F}|^s\big)^{\frac{p}{s}}
\Big\rangle^{\frac{q}{p}}.
\end{aligned}
\end{equation}
Combining the inequalities $\eqref{f:2.18}$ and $\eqref{f:2.17}$, and integrating both sides with respect to $y\in\Omega/\varepsilon$ in terms of
$\tilde{\omega}dy$ (also denoted by $d\mu$), we obtain
\begin{equation*}
\begin{aligned}
&\int_{\Omega/\varepsilon}
\Big\langle
\Big(\dashint_{U_{2R}(y)}|\tilde{F}|^s
\Big)^{\frac{p}{s}}
\Big\rangle^{\frac{q}{p}}
d\mu(y)
\gtrsim
R^{-\frac{dq}{s}+d}\int_{\Omega/\varepsilon}
\dashint_{U_R(y)}\Big\langle
\int_{U_1(z)}
\big(\dashint_{U_1(z')}
|\tilde{F}|^s\big)^{\frac{p}{s}}
\Big\rangle^{\frac{q}{p}}
d\mu(y)\\
&\gtrsim^{\eqref{f:17},\eqref{f:19}}
R^{-\frac{dq}{s}+d+\delta_0d}\int_{\Omega/\varepsilon}
\Big\langle \int_{U_1(y)}
\big(\dashint_{U_1(z)}
|\tilde{F}|^s\big)^{\frac{p}{s}}
\Big\rangle^{\frac{q}{p}}
\dashint_{B_1(y)} \tilde{\omega} dy\\
&\gtrsim^{\eqref{f:20}}
R^{-\frac{dq}{s}+d+\delta_0d}\int_{\Omega/\varepsilon}
\Big\langle
\big(\dashint_{U_1(y)}
|\tilde{F}|^s\big)^{\frac{p}{s}}
\Big\rangle^{\frac{q}{p}}
\dashint_{B_1(y)} \tilde{\omega} dy,
\end{aligned}
\end{equation*}
which completes the proof of $\eqref{f:23}$.

\medskip
\noindent
\textbf{Case 3.} For $s\leq p\leq q$, we establish
the estimate \eqref{f:23}.
The argument is similar to that given for Case 2.
For any $y\in\Omega/\varepsilon$, a routine decomposition leads to
\begin{equation*}
\begin{aligned}
\Big\langle
\Big(\dashint_{U_{2R}(y)}|\tilde{F}|^s
\Big)^{\frac{p}{s}}
\Big\rangle^{\frac{q}{p}}
\gtrsim_{d} R^{-\frac{dq}{s}}
\Big\langle
\Big(\int_{U_{2R}(y)}|\tilde{F}|^s
\Big)^{\frac{p}{s}}
\Big\rangle^{\frac{q}{p}}
\gtrsim
R^{-\frac{dq}{s}}
\Big\langle
\Big(\sum_{i}\int_{U_1(y_i)}|\tilde{F}|^s
\Big)^{\frac{p}{s}}
\Big\rangle^{\frac{q}{p}}.
\end{aligned}
\end{equation*}
By using H\"older's inequality of $l^p$-version twice, one can further obtain that
\begin{equation*}
\begin{aligned}
\Big\langle
\Big(\dashint_{U_{2R}(y)}|\tilde{F}|^s
\Big)^{\frac{p}{s}}
\Big\rangle^{\frac{q}{p}}
&\gtrsim R^{-\frac{dq}{s}}
\Big(\sum_{i}
\Big\langle\big(\int_{U_1(y_i)}|\tilde{F}|^s
\big)^{\frac{p}{s}}\Big\rangle
\Big)^{\frac{q}{p}}\\
&\gtrsim R^{-\frac{dq}{s}}
\sum_{i}
\Big\langle\big(\int_{U_1(y_i)}|\tilde{F}|^s
\big)^{\frac{p}{s}}\Big\rangle^{\frac{q}{p}}
\approx
R^{-\frac{dq}{s}+d}
\dashint_{U_R(y)}
\Big\langle\big(\dashint_{U_1(z)}|\tilde{F}|^s
\big)^{\frac{p}{s}}\Big\rangle^{\frac{q}{p}}.
\end{aligned}
\end{equation*}
Integrating both sides above with respect to $y\in\Omega/\varepsilon$ in terms of $d\mu$, and then appealing to the following estimate:
\begin{equation*}
\begin{aligned}
\int_{\Omega/\varepsilon}
\dashint_{U_R(y)}
\Big\langle\big(\dashint_{U_1(z)}|\tilde{F}|^s
\big)^{\frac{p}{s}}\Big\rangle^{\frac{q}{p}}d\mu(y)
&=^{\eqref{f:17}}\int_{\Omega/\varepsilon}
\Big\langle\big(\dashint_{U_1(y)}|\tilde{F}|^s
\big)^{\frac{p}{s}}\Big\rangle^{\frac{q}{p}}
\dashint_{D_R(y)} \tilde{\omega}dy\\
&\gtrsim^{\eqref{f:19}}
R^{\delta_0 d}
\int_{\Omega/\varepsilon}
\Big\langle\big(\dashint_{U_1(y)}|\tilde{F}|^s
\big)^{\frac{p}{s}}\Big\rangle^{\frac{q}{p}}
\dashint_{B_1(y)} \tilde{\omega}dy,
\end{aligned}
\end{equation*}
we can immediately reach the stated estimate $\eqref{f:23}$ in the case of $s\leq p\leq q$.

As a result,
combining the estimates \eqref{f:22} and \eqref{f:23},
we have the stated estimate $\eqref{f:21}$. The remainder of the proof
for $\eqref{pri:14}$
is analogy to those computations shown in Step 3, and we don't reproduce it here.
\end{proof}

%%%%%%%%%%%%%%%%%%%%%%%%%%%%%%%%%%%%%%%%%%%%%%%%%%%%%
%%%%下面的内容作为备忘（由随机尺度到1，下界估计部分，略去）%%%%
%%%%%%%%%%%%%%%%%%%%%%%%%%%%%%%%%%%%%%%%%%%%%%%%%%%%%
%%%%%%%%%%%%%%%%%%%%%%%%%%%%%%%%%%%%%%%%%%%%%%%%%%%%%
%

%
%
%%%%%%%%%%%%%%%%%%%%%%%%%%%%%%%%%%%%%%%%%%%%%%%%%%%%%
%%%%%%%%%%%%%%%%%%%%%%%%%%%%%%%%%%%%%%%%%%%%%%%%%%%%%
%%%%%%%%%%%%%%%%%%%%%%%%%%%%%%%%%%%%%%%%%%%%%%%%%%%%%
%%%%%上面的内容作为备忘（由随机尺度到1，下界估计部分，略去）%%%
%%%%%%%%%%%%%%%%%%%%%%%%%%%%%%%%%%%%%%%%%%%%%%%%%%%%%
%%%%%%%%%%%%%%%%%%%%%%%%%%%%%%%%%%%%%%%%%%%%%%%%%%%%%
%%%%%%%%%%%%%%%%%%%%%%%%%%%%%%%%%%%%%%%%%%%%%%%%%%%%%
%%%%%%%%%%%%%%%%%%%%%%%%%%%%%%%%%%%%%%%%%%%%%%%%%%%%%

\medskip
\noindent
\textbf{Proof of Proposition $\ref{P:7}$.}
In terms of the Lipschitz domain, the proof of the estimate $\eqref{pri:9-3}$ can be analogically demonstrated by Steps 1,2,3 in the proof of Lemma $\ref{lemma:11}$, so we don't repeat the details here.
By setting $s=2$ and $\omega = 1$, we can directly apply Lemma $\ref{lemma:12}$ to the both sides
of $\eqref{pri:9-3}$, and get the stated result $\eqref{pri:9-3*}$.
Concerned with the regular SKT(or $C^1$) domain,
the estimate $\eqref{pri:9}$
 comes from $\eqref{pri:9-1}$ in Lemma $\ref{lemma:11}$.
 Based upon Lemma $\ref{lemma:12}$,
 combining the estimates $\eqref{pri:9}$, $\eqref{pri:13}$, and
 $\eqref{pri:14}$ leads to $\eqref{pri:11}$.
By taking $\omega=1$ in $\eqref{pri:11}$, we have the estimate $\eqref{pri:26}$. Moreover, let $\Omega_0\supseteq \Omega$
satisfy $\partial\Omega_0\in C^2$ and   $\text{dist}(\cdot,\partial\Omega_0)=O(\varepsilon)$ on $\partial\Omega$.
We define $\omega_\sigma(x):=[\text{dist}(x,\partial\Omega_0)]^{q-1}$,
it is known that $\omega_\sigma\in A_q$
(see e.g. \cite[Theorem 3.1]{Duran-Sanmartino-Toschi04}), and
$\dashint_{B_\varepsilon(x)}\omega_\sigma \sim \omega_\sigma(x)$
for any $x\in\Omega$. Therefore, $\eqref{pri:11}$ provides us
with $\eqref{pri:11B}$ for $\omega_\sigma$.
\qed

\medskip
\noindent
\textbf{Proof of Theorem $\ref{thm:1}$.}
The results of Theorem $\ref{thm:1}$ follows from Propositions $\ref{P:10}$ and $\ref{P:7}$.
\qed

\section{Homogenization errors}\label{section:4}

\subsection{\centering Strong norm estimates}\label{subsection:4.1}

\noindent
Essentially, it continues the idea of using duality and weight functions to overcome the loss of the convergence rate, caused by boundary layers, and we will take the scheme developed by the second author in \cite{Xu16} to handle this problem.
%Since a local smoothness hypothesis $\eqref{a:3}$
%Nevertheless, the present model involves randomness, and we need to use
%the annealed Calder\'on-Zygmund estimates even if we are working under the energy framework in terms of spatial variables.
%Although we introduce the distance function as a weight to accelerate the convergence rate, here we don't really rely on weighted annealed Calder\'on-Zygmund estimates (compared to the weak norm estimates later on). Also, to complete the whole scheme, we establish the random cancellation property and $H^1$-error estimates (see Lemmas $\ref{lemma:5}$ and $\ref{lemma:6}$, respectively).
%Moreover, it is not difficult to observe that the non-smooth boundary condition will affect the stochastic integrability on the homogenization error, and it remains to be a question whether this can be improved under current conditions.

\begin{proposition}[optimal errors]\label{P:8}
Let $\Omega\ni\{0\}$ be a bounded Lipschitz domain and $0<\varepsilon\ll 1$.
Suppose that  $\langle\cdot\rangle$ is stationary, satisfying the spectral gap condition $\eqref{a:2}$, and the (admissible) coefficient additionally satisfies $\eqref{a:3}$ and
the symmetry condition $a=a^*$.
Let $u_\varepsilon,\bar{u}\in H_0^1(\Omega)$ be associated with $f\in L^2(\Omega)$ and $b\in H^1(\partial\Omega)$ by
\begin{equation}\label{pde:16}
\left\{\begin{aligned}
-\nabla\cdot a^{\varepsilon}\nabla u_\varepsilon
&= f &\quad&\text{in}\quad\Omega;\\
u_\varepsilon
&= b &\quad&\text{on}\quad\partial\Omega,
\end{aligned}\right.\qquad \text{and}
\qquad
\left\{\begin{aligned}
-\nabla\cdot \bar{a}\nabla \bar{u}
&= f &\quad&\text{in}\quad\Omega;\\
\bar{u}
&= b &\quad&\text{on}\quad\partial\Omega.
\end{aligned}\right.
\end{equation}
Then, for any $p<\infty$, we have
\begin{equation}\label{pri:5.1}
 \Big\langle \big(\int_{\Omega}|u_\varepsilon-\bar{u}|^2
 \big)^{\frac{p}{2}}\Big\rangle^{\frac{1}{p}}
 \lesssim_{\lambda,\lambda_1,
\lambda_2,d,M_0,p} \mu_d(R_0/\varepsilon)\varepsilon\ln(R_0/\varepsilon)
 \Big\{\|f\|_{L^2(\Omega)}+\|b\|_{H^1(\partial\Omega)}\Big\}.
\end{equation}
\end{proposition}

Let $\eta, \tilde{\eta}\in C_{0}^{1}(\Omega)$
be cut-off functions, satisfying
\begin{equation}\label{cutoff1}
\left\{\begin{aligned}
  & 0\leq \eta,\tilde{\eta}\leq 1,\quad \max\{|\nabla \eta|,
  |\nabla \tilde{\eta}|\}\leq C \varepsilon^{-1};\\
  & \text{supp}(\eta)
  \subset \Omega\backslash O_{3\varepsilon},~~
  \text{supp}(\tilde{\eta})\subset
  \Omega\backslash O_{7\varepsilon};\\
  & \eta=1\quad \text{in}~~
  \Omega\backslash O_{4\varepsilon},\quad
  \tilde{\eta}=1\quad \text{in}~~
  \Omega\backslash O_{8\varepsilon},
  \end{aligned}\right.
\end{equation}
where $O_r:=\{x\in\Omega:\text{dist}(x,\partial\Omega)\leq r\}$.
By the above definition, it's known that
$(1-\eta)\tilde{\eta}
=0~\text{in}~ \Omega$.
For the ease of statement it's fine to assume that
$(\phi(0),\sigma(0))=0$, otherwise we can replace
$(\phi,\sigma)$ by $(\bar{\phi},\bar{\sigma}) =
(\phi-\phi(0),\sigma-\sigma(0))$.

\begin{lemma}[errors in $H^1$-norm]\label{lemma:6}
Assume the same conditions hold as in Proposition $\ref{P:8}$.
Let $u_\varepsilon,\bar{u}\in H_0^1(\Omega)$ be associated with $f\in L^2(\Omega)$ and $b\in H^1(\partial\Omega)$ by
the equations $\eqref{pde:16}$, respectively.
Let $\eta\in C_0^1(\Omega)$ be a cut-off function satisfying $\eqref{cutoff1}$,
and we consider the following error term\footnote{
Compared to
the error $w_\varepsilon$ stated in $\eqref{expansion:1}$,
we set $u=\bar{u}$ in $\eqref{expansion:1}$ here,
where $\bar{u}$ is the effective solution in $\eqref{pde:16}$.}
\begin{equation}\label{eq:5.1}
\begin{aligned}
  z_\varepsilon = u_\varepsilon - \bar{u}
  -\varepsilon\phi_i^\varepsilon \varphi_i
  \quad \text{with}\quad
  \varphi_i = \eta\partial_i\bar{u}.
\end{aligned}
\end{equation}
Then, for any $p<\infty$, there holds
\begin{equation}\label{pri:5.2}
\Big\langle\big(\int_{\Omega}|\nabla z_\varepsilon|^2\big)^{\frac{p}{2}}\Big\rangle^{\frac{1}{p}}
\lesssim_{\lambda,\lambda_1,
\lambda_2,d,M_0,p}
\mu_d(R_0/\varepsilon) \varepsilon^{\frac{1}{2}}
\Big\{\|f\|_{L^2(\Omega)}+\|b\|_{H^1(\partial\Omega)}\Big\},
\end{equation}
%Moreover, there exists $p_1:=\frac{2d}{d-1}+\theta$ with $\theta>0$ depending only on $\Omega$, such that for all $p<p_1$ we have
%\begin{equation}
%\begin{aligned}
%\bigg(\int_{\Omega} \Big\langle\big(\dashint_{D_\varepsilon(x)}|\nabla z_\varepsilon|^2\big)^{\frac{p}{2}}\Big\rangle^{\frac{2}{p}}
%\bigg)^{\frac{1}{2}}
%\lesssim \mu_d(R_0/\varepsilon) \varepsilon^{\frac{1}{2}}
%\Big\{\|f\|_{L^2(\Omega)}+\|b\|_{H^1(\partial\Omega)}\Big\},
%\end{aligned}
%\end{equation}
%Moreover, if $\Omega$ is a bounded Reifenberg flat (or $C^1$) domain, taking $\delta_\sigma(x):=\text{dist}(x,\partial\Omega_0)$ we have
%\begin{equation}\label{pri:5.9}
%\begin{aligned}
%\bigg(\int_{\Omega} \Big\langle\big(\dashint_{D_\varepsilon(x)}|\nabla z_\varepsilon|^2\big)^{\frac{p}{2}}\Big\rangle^{\frac{2}{p}}
%\delta_\sigma\bigg)^{\frac{1}{2}}
%\lesssim \mu_d(R_0/\varepsilon) \varepsilon
%\Big\{\|f\|_{L^2(\Omega)}+\|b\|_{H^1(\partial\Omega)}\Big\}.
%\end{aligned}
%\end{equation}
where the up to constant is independent of $\varepsilon$.
\end{lemma}

\begin{proof}
By using the antisymmetry of flux corrector and the equations $\eqref{pde:16}$, a routine computation leads to
\begin{equation}\label{pde:17}
 -\nabla\cdot a^\varepsilon\nabla z_\varepsilon
 =\nabla \cdot \Big[\varepsilon(a^\varepsilon\phi_i^\varepsilon-\sigma_i^\varepsilon)
 \nabla \varphi_i
 -(a^\varepsilon-\bar{a})(\nabla u_\varepsilon -\varphi)\Big]
 \quad \text{in}\quad \Omega,
\end{equation}
with $z_\varepsilon = 0$ on $\partial\Omega$, where $\varphi
=(\varphi_1,\cdots,\varphi_d)$, and
$\text{supp}(\varphi)\subset\Omega\setminus O_{2\varepsilon}$.

Therefore, with the help of the energy estimate and taking
$\big\langle(\cdot)^{p}\big\rangle^{\frac{1}{p}}$, we have
\begin{equation*}%\label{f:55}
\begin{aligned}
\Big\langle\big(\int_{\Omega}|\nabla z_\varepsilon|^2\big)^{\frac{p}{2}}\Big\rangle^{\frac{1}{p}}
&\lesssim_{\lambda,d}
\Big\langle\big(\int_{\Omega}
|\varepsilon(\phi_i^\varepsilon,\sigma_i^\varepsilon)
 \nabla \varphi_i |^2\big)^{\frac{p}{2}}\Big\rangle^{\frac{1}{p}}
+ \Big\langle\big(\int_{\Omega}
|(1-\eta)\nabla\bar{u}|^2\big)^{\frac{p}{2}}\Big\rangle^{\frac{1}{p}}\\
&\lesssim \Big(\int_{\Omega}\big\langle|\varepsilon(\phi_i^\varepsilon,\sigma_i^\varepsilon)|^p
\big\rangle^{\frac{2}{p}}|\nabla\varphi_i|^2\Big)^{\frac{1}{2}}
+ \Big(\int_{O_{4\varepsilon}}|\nabla\bar{u}|^2\Big)^{\frac{1}{2}}\\
&\lesssim_{\lambda,\lambda_1,\lambda_2,d,p}^{\eqref{pri:3}}
\varepsilon
\Big(\int_{\Omega}
|\nabla(\eta\partial\bar{u})|^2\mu_d^2(\cdot/\varepsilon)\Big)^{\frac{1}{2}}
+ \Big(\int_{O_{4\varepsilon}}|\nabla\bar{u}|^2\Big)^{\frac{1}{2}},
\end{aligned}
\end{equation*}
where we employ Minkowski's inequality in the second line above.
This implies that
\begin{equation}\label{f:5.12}
\begin{aligned}
\Big\langle\big(\int_{\Omega}|\nabla z_\varepsilon|^2\big)^{\frac{p}{2}}\Big\rangle^{\frac{2}{p}}
&\lesssim_{\lambda,\lambda_1,\lambda_2,d,p}
\mu_d^2(R_0/\varepsilon)
\int_{O_{4\varepsilon}}
|\nabla\bar{u}|^2
+ \varepsilon^2\mu_d^2(R_0/\varepsilon)
\int_{\Omega\setminus O_{\varepsilon}}
|\nabla^2\bar{u}|^2.
\end{aligned}
\end{equation}
For the ease of statement, it is fine to assume $\|f\|_{L^2(\Omega)}+\|b\|_{H^1(\partial\Omega)}=1$.
In this regard, the problem is reduced to  the following ``layer'' and ``co-layer'' type estimates
(which was introduced in \cite{Xu16}, inspired by Shen's work \cite{Shen16}):
\begin{equation}\label{f:5.2}
\begin{aligned}
&\|\nabla \bar{u}\|_{L^2(O_{4\varepsilon})}
+ \varepsilon^{-\frac{1}{2}}\|\nabla \bar{u}\delta^{\frac{1}{2}}\|_{L^2(O_{4\varepsilon})}
+ \varepsilon\|\nabla^2 \bar{u}\|_{L^2(\Omega\setminus O_\varepsilon)}
\lesssim_{\lambda,d,M_0} \varepsilon^{\frac{1}{2}};\\
&\|\nabla \bar{u}\delta^{-\frac{1}{2}}\|_{L^2(\Omega\setminus O_\varepsilon)}
+ \|\nabla^2 \bar{u}\delta^{\frac{1}{2}}\|_{L^2(\Omega\setminus O_\varepsilon)}
\lesssim_{\lambda,d,M_0} \ln(R_0/\varepsilon),
\end{aligned}
\end{equation}
where $\delta(x):=\text{dist}(x,\partial\Omega)$ for any $x\in\Omega$.

Consequently, by plugging the estimates $\eqref{f:5.2}$ back into $\eqref{f:5.12}$, we have
\begin{equation*}
\begin{aligned}
\Big\langle\big(\int_{\Omega}|\nabla z_\varepsilon|^2\big)^{\frac{p}{2}}\Big\rangle^{\frac{1}{p}}
&\lesssim_{\lambda,\lambda_1,
\lambda_2,d,M_0,p} \mu_d(R_0/\varepsilon) \varepsilon^{\frac{1}{2}},
\end{aligned}
\end{equation*}
which offers us the desired estimate $\eqref{pri:5.2}$.
\end{proof}

\medskip

\noindent
\textbf{Proof of Proposition $\ref{P:8}$.}
We finish the whole argument by three steps.

\noindent
\textbf{Step 1.} Construct the auxiliary equations and simplify the problem.
Given arbitrary $g\in C_0^\infty(\Omega)$, we consider the following
adjoint problems:
\begin{equation}\label{pde:18}
\left\{\begin{aligned}
-\nabla\cdot a^{*\varepsilon}\nabla v_\varepsilon
&= g &\quad&\text{in}\quad\Omega;\\
v_\varepsilon
&= 0 &\quad&\text{on}\quad\partial\Omega;
\end{aligned}\right. \qquad \text{and}
\qquad
\left\{\begin{aligned}
-\nabla\cdot \bar{a}^{*}\nabla \bar{v}
&= g &\quad&\text{in}\quad\Omega;\\
\bar{v}
&= 0 &\quad&\text{on}\quad\partial\Omega.
\end{aligned}\right.
\end{equation}
We denote the related homogenization error term by  $\tilde{z}_\varepsilon:= v_\varepsilon-\bar{v} - \varepsilon\phi_j^{*\varepsilon}(\tilde{\eta}\partial_j\bar{v})$,
and have
\begin{equation}\label{f:61}
\nabla v_\varepsilon
= \nabla \tilde{z}_\varepsilon
+ \nabla \bar{v}
+ \nabla\phi_j^{*\varepsilon}
(\tilde{\eta}\partial_j\bar{v})
+ \varepsilon\phi_j^{*\varepsilon}
\nabla(\tilde{\eta}\partial_j\bar{v}),
\end{equation}
where $\tilde{\eta}$ is the cut-off function defined in $\eqref{cutoff1}$,
and $\phi^*$ is the corresponding corrector associated with $\nabla\cdot a^{*}\nabla$.
Moreover, recalling the error term $\eqref{eq:5.1}$, it follows from
the equations $\eqref{pde:17}$ and $\eqref{pde:18}$ that
\begin{equation*}
\begin{aligned}
\int_{\Omega} z_\varepsilon g
=-\int_{\Omega}\nabla v_\varepsilon\cdot
\varepsilon(a^\varepsilon\phi_i^{\varepsilon}-\sigma^\varepsilon_i)
\nabla\varphi_i
-\int_{\Omega}\nabla v_\varepsilon\cdot
(a^\varepsilon-\bar{a})(\nabla\bar{u}-\varphi).
\end{aligned}
\end{equation*}
By the linearity of the equations, one may assume $\|g\|_{L^2(\Omega)}=1$ and
$\|f\|_{L^2(\Omega)}+\|b\|_{H^1(\partial\Omega)}=1$.
Thus, the desired estimate $\eqref{pri:5.1}$ immediately follows from
the duality argument, triangle inequality and the following two estimates: for any $p<\infty$, (it suffices to show the case $p>2$ while the case $p<2$ follows from H\"older's inequality trivially.)
\begin{subequations}
\begin{align}
 & \Big\langle\Big|\int_{\Omega}\nabla v_\varepsilon\cdot
\varepsilon(a^\varepsilon\phi_i^{\varepsilon}-\sigma^\varepsilon_i)
\nabla\varphi_i\Big|^p\Big\rangle^{\frac{1}{p}}
\lesssim \mu_d(R_0/\varepsilon)\varepsilon\ln(R_0/\varepsilon);
\label{pri:5.6}\\
 & \Big\langle\Big|\int_{\Omega}\nabla v_\varepsilon\cdot
(a^\varepsilon-\bar{a})(\nabla\bar{u}-\varphi)\Big|^p\Big\rangle^{\frac{1}{p}}
\lesssim \varepsilon\ln(R_0/\varepsilon),
\label{pri:5.7}
\end{align}
\end{subequations}
where the multiplicative constant depends on
$\lambda,\lambda_1,\lambda_2,d,p$, and $M_0$.

\medskip
\noindent
\textbf{Step 2.} Arguments for $\eqref{pri:5.6}$.
The main idea is to accelerate the convergence rate by appealing to the constructed error term $\eqref{f:61}$. Therefore, plugging
the formula $\eqref{f:61}$ into the left-hand side of $\eqref{pri:5.6}$,
we have four terms in hand. By the
properties of cut-off functions $\eta,\tilde{\eta}$ defined in $\eqref{cutoff1}$, we reduce the error estimate to the ``layer'' and ``co-layer'' type estimates $\eqref{f:5.2}$.
Replacing $\nabla v_\varepsilon$ with
$\nabla\tilde{z}_\varepsilon$ in the left-hand side of $\eqref{pri:5.6}$,
by using H\"older's inequality ($\frac{1}{p}=\frac{1}{p_1}+\frac{1}{p_2}$), Minkowski's inequality and Lemma $\ref{lemma:6}$
in the order, we first obtain
\begin{equation*}
\begin{aligned}
&\Big\langle\big|\int_{\Omega}\nabla \tilde{z}_\varepsilon\cdot
\varepsilon(a^\varepsilon\phi_i^{\varepsilon}-\sigma^\varepsilon_i)
\nabla\varphi_i\big|^p\Big\rangle^{\frac{1}{p}}
\lesssim
\varepsilon
\Big\langle\big(\int_{\Omega}|\nabla \tilde{z}_\varepsilon|^2\big)^{\frac{p_1}{2}}
\Big\rangle^{\frac{1}{p_1}}
\Big(\int_{\Omega}\big\langle|(\phi_i^{\varepsilon},\sigma_i^{\varepsilon})
|^{p_2}\big\rangle^{\frac{2}{p_2}}|\nabla(\eta\partial_i\bar{u})|^2\Big)^{\frac{1}{2}}\\
&\lesssim^{\eqref{pri:3}}
\varepsilon
\Big\langle\big(\int_{\Omega}|\nabla \tilde{z}_\varepsilon|^2\big)^{\frac{p_1}{2}}
\Big\rangle^{\frac{1}{p_1}}
\Big(\int_{\Omega}
|\nabla(\eta\partial_i\bar{u})|^2
\mu_d^2(\cdot/\varepsilon)\Big)^{\frac{1}{2}}
\lesssim^{\eqref{pri:5.2},\eqref{f:5.2}}_{
\lambda,\lambda_1,
\lambda_2,d,M_0,p}
\mu_d^2(R_0/\varepsilon)\varepsilon,
\end{aligned}
\end{equation*}
where it is fine to choose $p_1=p_2=2p$ in the above calculus (also
used in later computations).

Replacing $\nabla v_\varepsilon$ with
$\nabla\bar{v}$ in the left-hand side of $\eqref{pri:5.6}$,
we then similarly have
\begin{equation}\label{f:56}
\begin{aligned}
&\Big\langle\big|\int_{\Omega}
\nabla\bar{v}
\cdot
\varepsilon(a^\varepsilon\phi_i^{\varepsilon}-\sigma^\varepsilon_i)
\nabla\varphi_i\big|^p\Big\rangle^{\frac{1}{p}}
\lesssim
\Big\langle \big(\int_{\text{supp}(\eta)}
|\nabla\bar{v}|^2\delta^{-1}\big)^{\frac{p_1}{2}}
\Big\rangle^{\frac{1}{p_1}}
\Big\langle\big(\int_{\Omega}
|\varepsilon(\phi_i^\varepsilon,\sigma_i^\varepsilon)\nabla\varphi_i|^2\delta\big)^{\frac{p_2}{2}}
\Big\rangle^{\frac{1}{p_2}}\\
&\lesssim \Big(\int_{\text{supp}(\eta)}
|\nabla\bar{v}|^2\delta^{-1}\Big)^{\frac{1}{2}}
\Big(\int_{\Omega}\big\langle|\varepsilon(\phi_i^\varepsilon,\sigma_i^\varepsilon)|^{p_2}
\big\rangle^{\frac{2}{p_2}}|\nabla(\eta\partial_i\bar{u})|^2\delta
\Big)^{\frac{1}{2}}
\lesssim^{\eqref{pri:3},\eqref{f:5.2}}_{
\lambda,\lambda_1,
\lambda_2,d,M_0,p}
\mu_d(R_0/\varepsilon)\varepsilon\ln(R_0/\varepsilon).
\end{aligned}
\end{equation}
Replacing $\nabla v_\varepsilon$ with
$\nabla\phi_j^{*\varepsilon}
(\tilde{\eta}\partial_j\bar{v})$ in the left-hand side of $\eqref{pri:5.6}$,
by H\"older's inequality, Minkowski's inequality and
Lemma $\ref{lemma:3}$,
we arrive at
\begin{equation*}
\begin{aligned}
&\quad\Big\langle\big|\int_{\Omega}
\nabla\phi_j^{*\varepsilon}
(\tilde{\eta}\partial_j\bar{v})
\cdot
\varepsilon(a^\varepsilon\phi_i^{\varepsilon}-\sigma^\varepsilon_i)
\nabla\varphi_i\big|^p\Big\rangle^{\frac{1}{p}}\\
&\leq
\Big\langle\big(\int_{\text{supp}(\eta)}
|\nabla\phi_j^{*\varepsilon}(\tilde{\eta}\partial_j\bar{v})|^2\delta^{-1}\big)^{\frac{p_1}{2}}
\Big\rangle^{\frac{1}{p_1}}
\Big\langle\big(\int_{\text{supp}(\eta)}|\varepsilon(\phi_i^\varepsilon,\sigma_i^\varepsilon)
\nabla\varphi_i|^2\delta\big)^{\frac{p_2}{2}}\Big\rangle^{\frac{1}{p_2}}\\
&\lesssim^{\eqref{pri:2},\eqref{pri:3}}
\varepsilon
\Big(\int_{\text{supp}(\tilde{\eta})}
|\nabla\bar{v}|^2 \delta^{-1}\Big)^{\frac{1}{2}}
\Big(\int_{\text{supp}(\eta)}|\nabla^2\bar{u}|^2
\mu_d^2(\cdot/\varepsilon) \delta\Big)^{\frac{1}{2}}
\lesssim^{\eqref{f:5.2}}_{\lambda,\lambda_1,
\lambda_2,d,M_0,p}\mu_d
(R_0/\varepsilon)\varepsilon\ln(R_0/\varepsilon).
\end{aligned}
\end{equation*}
where the second inequality is due to the fact $\text{supp}(\nabla\eta)\cap
\text{supp}(\tilde{\eta})=\emptyset$.
Replacing $\nabla v_\varepsilon$ with
$\varepsilon\phi_j^{*\varepsilon}
\nabla(\tilde{\eta}\partial_j\bar{v})$ in the left-hand side of $\eqref{pri:5.6}$,
by the same token, we derive that
\begin{equation}\label{f:60}
\begin{aligned}
&\Big\langle\big|\int_{\Omega}
\varepsilon\phi_j^{*\varepsilon}
S_\varepsilon\big(\nabla(\tilde{\eta}\partial_j\bar{v})\big)
\cdot
\varepsilon(a^\varepsilon\phi_i^{\varepsilon}-\sigma^\varepsilon_i)
\nabla\varphi_i\big|^p\Big\rangle^{\frac{1}{p}}\\
&\leq \Big\langle
\Big(\int_{\Omega}|\varepsilon\phi_j^{*\varepsilon}
\nabla(\tilde{\eta}\partial_j\bar{v})|^2\Big)^{\frac{p_1}{2}}\Big\rangle^{\frac{1}{p_1}}
\Big\langle\Big(\int_{\Omega}
|\varepsilon(\phi_i^\varepsilon,\sigma_i^\varepsilon)\nabla\varphi_i|^2
\Big)^{\frac{p_2}{2}}\Big\rangle^{\frac{1}{p_2}}
\lesssim^{\eqref{pri:3},\eqref{f:5.2}}_{\lambda,\lambda_1,
\lambda_2,d,M_0,p}\mu_d^2(R_0/\varepsilon)\varepsilon.
\end{aligned}
\end{equation}

\medskip
\noindent
\textbf{Step 3.} Arguments for $\eqref{pri:5.7}$.
With the help of the cut-off function $\eta$, we can find that
the left-hand side of $\eqref{pri:5.7}$ merely involves the ``layer'' part:
\begin{equation}\label{f:5.3}
\int_{\Omega}\nabla v_\varepsilon\cdot
(a^\varepsilon-\bar{a})(\nabla\bar{u}-\varphi)
= \int_{\Omega}\nabla v_\varepsilon\cdot
(a^\varepsilon-\bar{a})\nabla\bar{u}(1-\eta).
\end{equation}
By noting the fact
$\text{supp}(\tilde{\eta})\cap\text{supp}(\nabla\eta)
=\emptyset$ again, it follows from \eqref{f:61} and H\"older's inequality that
\begin{equation*}
\begin{aligned}
\int_{\Omega}\nabla v_\varepsilon\cdot
(a^\varepsilon-\bar{a})\nabla\bar{u}(1-\eta)
&=^{\eqref{f:61}} \int_{\Omega}\big(\nabla z_\varepsilon+\nabla\bar{v}\big)\cdot
(a^\varepsilon-\bar{a})\nabla\bar{u}(1-\eta)\\
&\lesssim
\bigg[\Big(\int_{\Omega}|\nabla z_\varepsilon|^2\Big)^{\frac{1}{2}}
+\Big(\int_{O_{4\varepsilon}}|\nabla \bar{v}|^2\Big)^{\frac{1}{2}}\bigg]
\Big(\int_{O_{4\varepsilon}}|\nabla \bar{u}|^2\Big)^{\frac{1}{2}}.
\end{aligned}
\end{equation*}
Only the term containing $z_\varepsilon$ involves randomness, and
appealing to Lemma $\ref{lemma:6}$ we have
\begin{equation}\label{f:5.10}
\begin{aligned}
&\Big\langle\big|\int_{\Omega}\nabla v_\varepsilon
\cdot
(a^\varepsilon-\bar{a})
\nabla\bar{u}(1-\eta)\big|^p
\Big\rangle^{\frac{1}{p}}\\
&\lesssim
\bigg[\Big\langle\big(\int_{\Omega}
|\nabla z_\varepsilon|^2\big)^{\frac{p}{2}}\Big\rangle^{\frac{1}{p}}
+\Big(\int_{O_{4\varepsilon}}|\nabla \bar{v}|^2\Big)^{\frac{1}{2}}\bigg]
\Big(\int_{O_{4\varepsilon}}|\nabla \bar{u}|^2\Big)^{\frac{1}{2}}
%&\lesssim
%\bigg[\Big(\int_{\Omega}
%\Big\langle\big(\dashint_{D_\varepsilon(x)}|\nabla z_\varepsilon|^2\big)^{\frac{p}{2}}\Big\rangle^{\frac{2}{p}}
%\Big)^{\frac{1}{2}}
%+\Big(\int_{O_{4\varepsilon}}|\nabla \bar{v}|^2\Big)^{\frac{1}{2}}\bigg]
%\Big(\int_{O_{4\varepsilon}}|\nabla \bar{u}|^2\Big)^{\frac{1}{2}}
\lesssim^{\eqref{pri:5.2},\eqref{f:5.2}}_{\lambda,\lambda_1,
\lambda_2,d,M_0,p} \mu_d(R_0/\varepsilon)\varepsilon.
\end{aligned}
\end{equation}
This completes the whole proof.
\qed

\subsection{\centering Weak norm estimates}\label{subsection:4.2}

\noindent
In this subsection,
modulo the homogenization commutator (see e.g.
\cite[pp.29]{Josien-Otto22}),
we discuss the estimate of the homogenization error in the weak norm
for the regular SKT (or $C^1$) domain. To overcome the loss introduced by boundary layer, we employ the weighted annealed Calder\'on-Zygmund estimates to accelerate the convergence rate.

\begin{proposition}\label{P:9}
Let $\Omega$ be a bounded regular SKT (or $C^1$) domain including the zero point and $0<\varepsilon\ll1$.
Suppose that  $\langle\cdot\rangle$ is stationary, satisfying the spectral gap condition $\eqref{a:2}$, as well as $\eqref{a:3}$.
%and the coefficient additionally satisfies the symmetry condition $a=a^*$.
Let $u_\varepsilon,u_0\in H_0^1(\Omega)$ be associated with
$f\in C_0^{1}(\Omega;\mathbb{R}^d)$ by
\begin{equation}\label{pde:22-2}
\left\{\begin{aligned}
-\nabla\cdot a^{\varepsilon}\nabla u_\varepsilon
&= \nabla\cdot f &\quad&\text{in}\quad\Omega;\\
u_\varepsilon
&= 0 &\quad&\text{on}\quad\partial\Omega,
\end{aligned}\right.\qquad \text{and}
\qquad
\left\{\begin{aligned}
-\nabla\cdot \bar{a}\nabla \bar{u}
&= \nabla\cdot f &\quad&\text{in}\quad\Omega;\\
\bar{u}
&= 0 &\quad&\text{on}\quad\partial\Omega.
\end{aligned}\right.
\end{equation}
For any $h\in C_0^\infty(\Omega;\mathbb{R}^d)$, the random
variable $H^\varepsilon$ is defined as in $\eqref{eq:5.2}$.
Then for all $p<\infty$ we have
\begin{equation}\label{pri:6}
\begin{aligned}
\varepsilon^{-\frac{d}{2}} \big\langle (H^\varepsilon -\langle H^\varepsilon\rangle)^{2p} \big\rangle^{\frac{1}{2p}}
\lesssim_{\lambda,\lambda_1,
\lambda_2,d,M_0,p,s}
 \mu_d(R_0/\varepsilon)\varepsilon\ln^{\frac{1}{2s'}}(R_0/\varepsilon)
\Big(\int_{\Omega} |\nabla h|^{2s}\Big)^{\frac{1}{2s}}
\Big(\int_{\Omega}
 |R_0\nabla f|^{2s'}\Big)^{\frac{1}{2s'}},
\end{aligned}
\end{equation}
where $s,s'>1$ are associated with $1/s'+1/s=1$.
\end{proposition}

\begin{proof}
The main idea is appealing to sensitive arguments, which is
similar to that given in \cite[Proposition 6.1]{Josien-Otto22}, while
the present contribution is to establish
the fluctuation property of the two-scale expansion error with
a bounded domain considered. Let $s,s'>1$ satisfy $1/s+1/s'=1$, and
the proof is divided into four steps.

\medskip
\noindent
\textbf{Step 1.} Representation of the functional derivative
of $H^\varepsilon$ and reduction.
Let $\check{h},\hat{h},\tilde{h}$ represent the different rescale way
of $h$, as follows:
\begin{equation}\label{scaling1}
 \check{h}:=\varepsilon^d h(\varepsilon\cdot);
 \quad \hat{h}:= h(\varepsilon\cdot);
 \quad \tilde{h}:= \frac{1}{\varepsilon}h(\varepsilon\cdot).
\end{equation}
Throughout the proof, we will not distinguish the corresponding derivative symbols since the symbols of the functions under different scale transformations are marked as in $\eqref{scaling1}$.
Thus, in view of the equations $\eqref{pde:22-2}$ we can rewrite it as
\begin{equation*}
\left\{\begin{aligned}
-\nabla\cdot a \nabla \tilde{u}_{\varepsilon}
&= \nabla\cdot \hat{f}
&\quad&\text{in}~~\Omega/\varepsilon;\\
\tilde{u}_{\varepsilon}
&= 0
&\quad&\text{on}~\partial\Omega/\varepsilon,
\end{aligned}\right.
\qquad
\left\{\begin{aligned}
-\nabla\cdot \bar{a} \nabla \tilde{\bar{u}}
&= \nabla\cdot \hat{f}
&\quad&\text{in}~~\Omega/\varepsilon;\\
\tilde{\bar{u}}
&= 0
&\quad&\text{on}~\partial\Omega/\varepsilon.
\end{aligned}\right.
\end{equation*}
Recall that $w_\varepsilon:= u_\varepsilon-\bar{u}-\varepsilon\phi_i^\varepsilon\varphi_i$,
and we now fix $\varphi_i = \eta\partial_i\bar{u}$, where $\eta\in C_0^1(\Omega)$ is a cut-off function satisfying
$\eqref{cutoff1}$.
It follows that
\begin{equation}\label{scaling2}
\hat{\varphi}_i = \hat{\eta}\partial_i\tilde{\bar{u}},
\quad \tilde{w}_{\varepsilon} = \tilde{u}_{\varepsilon} - \tilde{\bar{u}}-\phi_i\hat{\varphi}_i,
\quad\text{and}\quad
\nabla\tilde{w}_{\varepsilon} = \nabla w_\varepsilon.
\end{equation}

Therefore, the error of two-scale expansions satisfies
\begin{equation*}
\begin{aligned}
-\nabla \cdot a\nabla
\tilde{w}_{\varepsilon}
= \nabla\cdot \big[(a\phi_i-\sigma_i)\nabla\hat{\varphi}_i
+ (a-\bar{a})(1-\hat{\eta})\nabla\tilde{\bar{u}}\big]
\quad \text{in}\quad \Omega/\varepsilon,
\end{aligned}
\end{equation*}
with $\tilde{w}_{\varepsilon} = 0$ on $\partial\Omega/\varepsilon$.
We now denote the random variable by
\begin{equation*}
H_\varepsilon : = \int_{\Omega/\varepsilon}
\check{h} \cdot (a-\bar{a})\big(\nabla \tilde{u}_{\varepsilon}
-\nabla\tilde{\bar{u}}
-\nabla\phi_i\hat{\varphi}_i\big).
\end{equation*}
and it follows from changing variable in $\eqref{eq:5.2}$ that $H_{\varepsilon}=H^{\varepsilon}$.
The advantage of the expression of $H_\varepsilon$ is that we can
directly appeal to the $L^p$-version
spectral gap inequality\footnote{It can be derived from
the spectral gap condition $\eqref{a:2}$ (see e.g.
\cite[pp.17-18]{Josien-Otto22}).}
\begin{equation}\label{Lp-SG}
\big\langle (H_\varepsilon -
\langle H_\varepsilon \rangle)^{2p} \big\rangle^{\frac{1}{p}}
\lesssim \Big\langle\Big(\int_{\mathbb{R}^d}
\big(\dashint_{B_1(y)}|\frac{\partial H_\varepsilon}{\partial a}|\big)^2 dy\Big)^p\Big\rangle^{\frac{1}{p}}
\end{equation}
to obtain the desired estimate $\eqref{pri:6}$, whenever we have
the concrete expression of $\frac{\partial H_\varepsilon}{\partial a}$.

To do so, for any $h\in C_0^\infty(\Omega;\mathbb{R}^d)$, we construct the auxiliary equations:
\begin{equation}\label{pde:22}
\left\{\begin{aligned}
-\nabla \cdot a^*\nabla v^*
&=\nabla\cdot(a^{*}\phi_{j}^{*}-\sigma_{j}^{*})\nabla\check{h}_j
&\quad&\text{in}~~\Omega/\varepsilon;\\
v^*&=0 &\quad&\text{on}~~\partial\Omega/\varepsilon,
\end{aligned}\right.
\end{equation}
and
\begin{equation*}
-\nabla \cdot a^*\nabla z_j^*
=\nabla\cdot(a^{*}\phi_{i}^{*}-\sigma_{i}^{*})\nabla(\check{h}_i
\hat{\varphi}_j)
\quad \text{in}~~\mathbb{R}^d,
\end{equation*}
where $-\nabla\cdot a^{*}\nabla$ is the adjoint operator of
$-\nabla\cdot a\nabla$, and $(\phi^{*},\sigma^{*})$ is the corresponding
extended corrector.
Then, a routine computation leads to
\begin{equation}\label{f:6.0}
\begin{aligned}
\frac{\partial H_\varepsilon}{\partial a}
&= \check{h}_j(e_j+\nabla\phi_j^{*})\otimes (\nabla\tilde{w}_\varepsilon+\phi_i\nabla\hat{\varphi}_i)
+ \nabla\phi_j^*\check{h}_j\otimes(\nabla\tilde{\bar{u}}-\hat{\varphi})\\
&-\big[\nabla z_j^* + \phi_i^*\nabla(\check{h}_i\hat{\varphi}_j)\big]
\otimes (\nabla\phi_j+e_j)
+ (\nabla v^* + \phi^{*}_j\nabla\check{h}_j)\otimes\nabla\tilde{u}_\varepsilon,
\end{aligned}
\end{equation}
and the details can be found in \cite[pp.31-32]{Josien-Otto22} and we leave it to the reader. To complete the whole proof, we
need to establish the following estimates:
\begin{subequations}
\begin{align}
&\begin{aligned}
&\bigg\langle
 \bigg(\int_{\mathbb{R}^d}
 \Big(\dashint_{B_1(y)}\big|
 \check{h}_j(e_j+\nabla\phi_j^{*})\otimes (\nabla\tilde{w}_{\varepsilon}
 +\phi_i\nabla\hat{\varphi}_i)\big|\Big)^2 dy
 \bigg)^{p}\bigg\rangle^{\frac{1}{p}}\\
&\qquad\qquad\qquad \qquad\qquad
\lesssim \varepsilon^{d+2}\mu_d^{2}(R_0/\varepsilon)
\ln^{\frac{1}{s'}}(R_0/\varepsilon)
\Big(\int_{\Omega} |\nabla h|^{2s} dx\Big)^{\frac{1}{s}}
\Big(\int_{\Omega} |R_0\nabla f|^{2s'} dx\Big)^{\frac{1}{s'}};
\end{aligned} \label{f:6.a}\\
 & \begin{aligned}
\bigg\langle
 \bigg(\int_{\Omega/\varepsilon}
 \Big(\dashint_{U_1(y)}\big|
 \nabla\phi^{*}_j \check{h}_j \otimes
& (\nabla\tilde{\bar{u}}-\hat{\varphi})\big|\Big)^2 dy
 \bigg)^{p}\bigg\rangle^{\frac{1}{p}}\\
& \lesssim  %^{\eqref{f:62},\eqref{f:63}}
\varepsilon^{d+2}\Big(\int_{\Omega}
|\nabla h|^{2s}\Big)^{\frac{1}{s}}
\Big(\int_{\Omega}
|R_0\nabla f|^{2s'}
\Big)^{\frac{1}{s'}},
\end{aligned}\label{f:6.b}
\end{align}
\end{subequations}
and
\begin{subequations}
\begin{align}
&\begin{aligned}
\bigg\langle
 \bigg(\int_{\Omega/\varepsilon}
 \Big(\dashint_{U_1(y)}\big|
 (\nabla v^* + \phi^{*}_j\nabla\check{h}_j)
 &\otimes\nabla\tilde{u}_\varepsilon\big|\Big)^2 dy
 \bigg)^{p}\bigg\rangle^{\frac{1}{p}}\\
&\lesssim
\varepsilon^{d+2}\mu_d^{2}(R_0/\varepsilon)
\Big(\int_{\Omega}
|\nabla h|^{2s}\Big)^{\frac{1}{s}}
\Big(\int_{\Omega}
 |f|^{2s'}\Big)^{\frac{1}{s'}};
\end{aligned}  \label{f:6.c}\\
&\begin{aligned}
\bigg\langle
 \bigg(\int_{\mathbb{R}^d}
& \Big(\dashint_{B_1(y)}
 \big|
 \big(\nabla z_j^* + \phi^{*}_i\nabla(\check{h}_i\hat{\varphi}_j)\big)\otimes
 (\nabla\phi_j+e_j)\big|\Big)^2 dy
 \bigg)^{p}\bigg\rangle^{\frac{1}{p}}\\
&\lesssim
\varepsilon^{d+2}\mu_d^2(R_0/\varepsilon)\Big(\int_{\Omega}|
\nabla h|^{2s}\Big)^{\frac{1}{s}}
\bigg\{\Big(\int_{\Omega}|
f|^{2s'}\Big)^{\frac{1}{s'}}
+ \ln^{\frac{1}{s'}}(R_0/\varepsilon)
\Big(\int_{\Omega}
|R_0\nabla f|^{2s'}\Big)^{\frac{1}{s'}}\bigg\},
\end{aligned} \label{f:6.d}
\end{align}
\end{subequations}
where the multiplicative constant depends on
$\lambda,\lambda_1,\lambda_2,d,s,p$, and $M_0$.

Admitting them for a while, combining the estimates $\eqref{Lp-SG}$, $\eqref{f:6.0}$, $\eqref{f:6.a}$,
$\eqref{f:6.b}$, $\eqref{f:6.c}$ and $\eqref{f:6.d}$ leads to the desired
estimate $\eqref{pri:6}$.

\medskip
\noindent
\textbf{Step 2.} As a preparation, we need to show the following layer and co-layer type estimates:
\begin{subequations}\label{}
\begin{align}\label{}
& \int_{O_\varepsilon} |\nabla \bar{u}|^{2s'}
\lesssim_{d,\lambda,M_0,s} \varepsilon R_0^{2s'-1}\int_{\Omega}|\nabla f|^{2s'}; \label{f:6.10a}\\
& \int_{\Omega\setminus O_\varepsilon}
|\nabla^2\bar{u}|^{2s'}\delta^{2s'-1}(x)dx
\lesssim_{d,\lambda,M_0,s}
\ln(R_0/\varepsilon) R_0^{2s'-1}\int_{\Omega}|\nabla f|^{2s'}, \label{f:6.10b}
\end{align}
\end{subequations}
where $\delta(x):=\text{dist}(x,\partial\Omega)$ for any $x\in\Omega$.
The original idea of the proof is similar to that given in \cite{Shen16}, and we provide a proof for
the reader's convenience. Due to the non-smoothness of the boundary, we first divide
the effective equation of $\eqref{pde:22-2}$ into
\begin{equation}\label{pde:22-3}
 (1)~-\nabla\cdot \bar{a}\nabla\bar{u}^{(1)} = \nabla\cdot f_0\quad\text{in}\quad \mathbb{R}^d;
 \qquad
 (2)~\left\{\begin{aligned}
 \nabla\cdot \bar{a}\nabla \bar{u}^{(2)}
&= 0 &\quad&\text{in}\quad\Omega;\\
\bar{u}^{(2)}
&= -\bar{u}^{(1)} &\quad&\text{on}\quad\partial\Omega,
 \end{aligned}\right.
\end{equation}
where $f_0$ is the zero-extension of $f$ to the whole space $\mathbb{R}^d$.
It is not hard to see that $\bar{u} = \bar{u}^{(1)}+\bar{u}^{(2)}$ on $\overline{\Omega}$,
and our later analysis is based upon this decomposition.

We first address the estimate $\eqref{f:6.10a}$, and obtain
\begin{equation}\label{f:6.13}
\Big(\int_{O_\varepsilon} |\nabla \bar{u}|^{2s'}\Big)^{\frac{1}{2s'}}
 \leq
 \Big(\int_{O_\varepsilon} |\nabla \bar{u}^{(1)}|^{2s'}\Big)^{\frac{1}{2s'}}
 + \Big(\int_{O_\varepsilon} |\nabla \bar{u}^{(2)}|^{2s'}\Big)^{\frac{1}{2s'}}.
\end{equation}
For any
$t\in [0,2\varepsilon]$, one may define $S_t:=\Gamma_t(\partial\Omega)$ and
$\Gamma_t(x') := x'-tn(x')$ for a.e. $x'\in\partial\Omega$ and $n$ is the outward unit normal
vector associated with $\partial\Omega$. There exists a vector field $\rho\in C^1(\Omega;\mathbb{R}^d)$ such that $\rho\cdot n \geq c_0>0$ on $S_t$ for each $t\in[0,2\varepsilon]$, as well as $|\nabla \rho|\lesssim 1/R_0$.
It follows from the divergence theorem that
\begin{equation}\label{f:6.11*}
\begin{aligned}
\int_{S_t} n\cdot\rho|\nabla\bar{u}^{(1)}|^{2s'} dS
&= \int_{\Omega\setminus O_{t}}\nabla\cdot \big(\rho|\nabla\bar{u}^{(1)}|^{2s'}\big)\\
&= \int_{\Omega\setminus O_{t}}
\nabla\cdot\rho|\nabla\bar{u}^{(1)}|^{2s'}
+\int_{\Omega\setminus O_{t}}|\nabla\bar{u}^{(1)}|^{2s'-2}
\rho\cdot\nabla^2\bar{u}^{(1)}\nabla\bar{u}^{(1)},
\end{aligned}
\end{equation}
where we note that $S_t$ is the boundary of
$\Omega\setminus O_{t}$ and the outward
unit normal vector associated with $S_t$ is regarded as the same as
that of $\partial\Omega$ since they are sufficiently close to each other.
This together with H\"older's inequality leads to
\begin{equation}\label{f:6.11}
\begin{aligned}
\int_{S_t}|\nabla \bar{u}^{(1)}|^{2s'} dS
&\leq \frac{1}{c_0}\int_{S_t} n\cdot \rho |\nabla \bar{u}^{(1)}|^{2s'} dS
\lesssim^{\eqref{f:6.11*}} \frac{1}{R_0}\int_{\Omega}|\nabla \bar{u}^{(1)}|^{2s'}
+  \int_{\Omega}|\nabla \bar{u}^{(1)}|^{2s'-1}|\nabla^2 \bar{u}^{(1)}|\\
&\lesssim \Big(\int_{\Omega}|\nabla \bar{u}^{(1)}|^{\frac{2s'd}{d-1}}\Big)^{\frac{d-1}{d}}
+ \Big(\int_{\Omega}|\nabla \bar{u}^{(1)}|^{\frac{2s'd}{d-1}}\Big)^{\frac{(2s'-1)(d-1)}{2s'd}}
\Big(\int_{\Omega}|\nabla^2 \bar{u}^{(1)}|^{\frac{2s'd}{2s'+d-1}}\Big)^{\frac{2s'+d-1}{2s'd}}.
\end{aligned}
\end{equation}
Then, by boundedness of the Riesz potential and singular integrals
(see e.g. \cite[Chapter 7]{Giaquinta-Martinazzi12}), we have
\begin{subequations}
\begin{align}
&\Big(\int_{\mathbb{R}^d}
|\nabla \bar{u}^{(1)}|^{\frac{2s'd}{d-1}}\Big)^{\frac{d-1}{d}}
\lesssim_{d,s} \Big(\int_{\mathbb{R}^d}|\nabla\cdot f_0|^{\frac{2s'd}{2s'+d-1}}\Big)^{\frac{2s'+d-1}{d}}
\lesssim \Big(\int_{\Omega}|\nabla f|^{\frac{2s'd}{2s'+d-1}}\Big)^{\frac{2s'+d-1}{d}};
\label{f:6.16a}\\
&\int_{\mathbb{R}^d}|\nabla^2 \bar{u}^{(1)}|^{q}
\lesssim_{d,q} \int_{\mathbb{R}^d}|\nabla\cdot f_0|^{q}
\lesssim \int_{\Omega}|\nabla f|^{q}\qquad\forall~ 1<q<\infty.
\label{f:6.16b}
\end{align}
\end{subequations}
%(Here we merely employ the case $q=\frac{2ds'}{2s'+d-1}$.)
Plugging the estimates $\eqref{f:6.16a}$ and $\eqref{f:6.16b}$ back into $\eqref{f:6.11}$, there holds
\begin{equation}\label{f:6.12}
\int_{S_t}|\nabla \bar{u}^{(1)}|^{2s'} dS
\lesssim \Big(\int_{\Omega}|\nabla f|^{\frac{2ds'}{2s'+d-1}}\Big)^{\frac{2s'+d-1}{d}}
\lesssim_d R_0^{2s'-1} \int_{\Omega}|\nabla f|^{2s'}
\end{equation}
for any $t\in[0,2\varepsilon]$. By co-area formula, we have
\begin{equation}\label{f:6.14}
 \int_{O_\varepsilon} |\nabla \bar{u}^{(1)}|^{2s'}
 \lesssim \varepsilon \max_{t\in[0,2\varepsilon]}\int_{S_t}|\nabla \bar{u}^{(1)}|^{2s'} dS
 \lesssim \varepsilon R_0^{2s'-1} \int_{\Omega}|\nabla f|^{2s'}.
\end{equation}

Moreover, by the definition of the non-tangential maximal function (see \cite[pp.207-209]{Shen18}), as well
as the corresponding estimates (see \cite[Theorem 7.2]{Hofmann-Mitrea-Taylor10}
for a regular SKT domain; and
\cite[Theorem 2.4]{Fabes-Jodeit-Riviere78} for a $C^1$ domain), we have
\begin{equation}\label{f:6.15}
\begin{aligned}
\int_{O_\varepsilon} |\nabla \bar{u}^{(2)}|^{2s'}
&\lesssim \varepsilon\int_{\partial\Omega} |(\nabla \bar{u}^{(2)})^{*}|^{2s'}dS \\
&\lesssim_{d,M_0,s} \varepsilon\int_{\partial\Omega} |\nabla_{\text{tan}} \bar{u}^{(1)}|^{2s'}dS
\lesssim^{\eqref{f:6.12}} \varepsilon R_0^{2s'-1} \int_{\Omega}|\nabla f|^{2s'},
\end{aligned}
\end{equation}
where we also recall that $\nabla_{\text{tan}}$ is the tangential derivative with respect to $\partial\Omega$\footnote{We
refer the reader to \cite[pp.213-214]{Shen18} for the concrete definition and properties of the tangential derivative.}.
As a result, combining the estimates $\eqref{f:6.13}$, $\eqref{f:6.14}$,  and  $\eqref{f:6.15}$
leads to the desired estimate $\eqref{f:6.10a}$.

We now turn to the estimate $\eqref{f:6.10b}$. To do so, we first have
\begin{equation}\label{f:6.17}
\int_{\Omega\setminus O_\varepsilon}
|\nabla^2\bar{u}^{(1)}|^{2s'}\delta^{2s'-1}(x)dx
\lesssim R_0^{2s'-1}\int_{\mathbb{R}^d} |\nabla^2\bar{u}^{(1)}|^{2s'}
\lesssim^{\eqref{f:6.16b}} R_0^{2s'-1}\int_{\Omega}|\nabla f|^{2s'}.
\end{equation}
Then, by using the interior Lipschitz estimates for the solution of the equations (2) in $\eqref{pde:22-3}$,
i.e.,
\begin{equation*}
 |\nabla^2 \bar{u}^{(2)}(x)|^{2s'}
 \lesssim \frac{1}{[\delta(x)]^{2s'}}\dashint_{B_{\delta(x)/4}(x)}|\nabla \bar{u}^{(2)}|^{2s'},
\end{equation*}
we obtain
\begin{equation}\label{f:6.18}
\begin{aligned}
\int_{\Omega\setminus O_\varepsilon}
|\nabla^2\bar{u}^{(2)}|^{2s'}\delta^{2s'-1}(x)dx
&\lesssim \int_{O_{\frac{1}{100}R_0}\setminus O_\varepsilon}\dashint_{B_{\delta(x)/4}(x)}
|\nabla \bar{u}^{(2)}|^{2s'} \delta^{-1}(x)dx
+ \frac{1}{R_0}\int_{\Omega}|\nabla\bar{u}^{(2)}|^{2s'} \\
&\lesssim \ln(R_0/\varepsilon)\int_{\partial\Omega}
|(\nabla\bar{u}^{(2)})^{*}|^{2s'} +
\frac{1}{R_0}\int_{\Omega}|\nabla\bar{u}^{(1)}|^{2s'}\\
&\lesssim^{\eqref{f:6.15},\eqref{f:6.16a},\eqref{f:6.12}}
\ln (R_0/\varepsilon) R_0^{2s'-1} \int_{\Omega}|\nabla f|^{2s'}.
\end{aligned}
\end{equation}
By the same token, combining the estimates $\eqref{f:6.17}$ and
$\eqref{f:6.18}$ offers us the desired estimate $\eqref{f:6.10b}$.

\medskip
\noindent
\textbf{Step 3.} Show the estimates $\eqref{f:6.a}$ and $\eqref{f:6.b}$.
By changing variables, we have
\begin{equation}\label{f:6.3}
\begin{aligned}
& \bigg\langle
 \bigg(\int_{\mathbb{R}^d}
 \Big(\dashint_{B_1(y)}\big|
 \check{h}_j(e_j+\nabla\phi_j^{*})\otimes (\nabla\tilde{w}_{\varepsilon}
 +\phi_i\nabla\hat{\varphi}_i)\big|\Big)^2 dy
 \bigg)^{p}\bigg\rangle^{\frac{1}{p}}\\
&= \varepsilon^{d} \bigg\langle
 \bigg(\int_{\mathbb{R}^d}
 \Big(\dashint_{B_\varepsilon(x)}\big|
 h_j(e_j+\nabla\phi_j^{*\varepsilon})\otimes
 (\underbrace{\nabla w_\varepsilon+\varepsilon\phi_i^\varepsilon
 \nabla \varphi_i}_{F_\varepsilon})\big|\Big)^2dx
 \bigg)^{p}\bigg\rangle^{\frac{1}{p}}.
\end{aligned}
\end{equation}
Then, by using H\"older's inequality, Minkowski's inequality
and Lemma $\ref{lemma:3}$ in the order, we derive that
\begin{equation}\label{f:6.19}
\begin{aligned}
&\bigg\langle
 \bigg(\int_{\mathbb{R}^d}
 \Big(\dashint_{B_\varepsilon(x)}\big|
 h_j(e_j+\nabla\phi_j^{*\varepsilon})\otimes F_\varepsilon\big|\Big)^2dx
 \bigg)^{p}\bigg\rangle^{\frac{1}{p}}\\
&\leq
 \int_{\mathbb{R}^d}
 \Big\langle
 \Big(\dashint_{U_\varepsilon(x)}|
 (e_j+\nabla\phi_j^{*\varepsilon})|^{2s'}\Big)^{\frac{sp}{s'}}
 \Big\rangle^{\frac{1}{sp}}
 \Big\langle\Big(\dashint_{U_\varepsilon(x)}|
 F_\varepsilon|^{2}\Big)^{s'p}\Big\rangle^{\frac{1}{s'p}}
 \Big(\dashint_{U_\varepsilon(x)}|
 h_j|^{2s}\Big)^{\frac{1}{s}} dx\\
&\lesssim^{\eqref{pri:2}}
 \int_{\Omega}
 \Big\langle\Big(\dashint_{U_\varepsilon(x)}|
 F_\varepsilon|^{2}\Big)^{s'p}\Big\rangle^{\frac{1}{s'p}}
 \Big(\dashint_{U_\varepsilon(x)}|
 h_j|^{2s}\Big)^{\frac{1}{s}} dx\\
&\lesssim
\bigg(\int_{\Omega}
 \Big\langle\Big(\dashint_{U_\varepsilon(x)}|
 F_\varepsilon|^{2}\Big)^{s'p}\Big\rangle^{\frac{1}{p}}
 \delta_{\sigma}^{2s'-1}(x)dx\bigg)^{\frac{1}{s'}}
\bigg(\int_{\Omega}
\dashint_{U_\varepsilon(x)}|
 h_j|^{2s}\delta_{\sigma}^{\frac{s}{s'}-2s}(x) dx\bigg)^{\frac{1}{s}},
\end{aligned}
\end{equation}
where $\delta_\sigma(x):=\text{dist}(x,\partial\Omega_0)$ and
$\Omega_0\supseteq \Omega$ satisfying $\partial\Omega_0\in C^2$ and   $\text{dist}(\cdot,\partial\Omega_0)=O(\varepsilon)$ on $\partial\Omega$
(similar to that given in Proposition $\ref{P:7}$).
Due to the Lipschitz continuity of distance function, for any $\alpha>0$, there exist $C_1,C_2$ depending only on $\alpha$ such that
\begin{equation}\label{f:5.13}
C_1\delta_{\sigma}^{\pm\alpha}(x)
\leq \dashint_{B_\varepsilon(x)}\delta_{\sigma}^{\pm\alpha}
\leq C_2\delta_{\sigma}^{\pm\alpha}(x)
\qquad\forall x\in\Omega.
\end{equation}
On the one hand, from Fubini's theorem and a weighted Hardy's inequality
(see e.g. \cite[Theorem 1.1]{Lehrback14}), it follows that
\begin{equation}\label{f:6.2}
\begin{aligned}
\bigg(\int_{\Omega}
\dashint_{U_\varepsilon(x)}|
 h|^{2s}
\delta_{\sigma}^{\frac{s}{s'}-2s}(x) dx\bigg)^{\frac{1}{s}}
&\lesssim_d^{\eqref{f:17}} \bigg(\int_{\Omega}
|h|^{2s} \dashint_{B_\varepsilon(x)}\delta_{\sigma}^{\frac{s}{s'}-2s} dx\bigg)^{\frac{1}{s}}
\lesssim^{\eqref{f:5.13}}_{d,s}
\bigg(\int_{\Omega}
|h|^{2s}\delta_{\sigma}^{\frac{s}{s'}-2s} dx\bigg)^{\frac{1}{s}}\\
&\lesssim \bigg(\int_{\Omega_0}
|h|^{2s} \delta_{\sigma}^{\frac{s}{s'}-2s} \bigg)^{\frac{1}{s}}
 \lesssim \Big(\int_{\Omega_0} |\nabla h|^{2s} \delta_{\sigma}^{\frac{s}{s'}}\Big)^{\frac{1}{s}}
 \lesssim R_0^{\frac{1}{s'}}\Big(\int_{\Omega} |\nabla h|^{2s} \Big)^{\frac{1}{s}},
\end{aligned}
\end{equation}
where we also note that $\text{supp}(h)\subset\Omega$.
On the other hand, by appealing to Proposition $\ref{P:7}$ we have
\begin{equation*}
\begin{aligned}
&\bigg(\int_{\Omega}
\Big\langle
\big(\dashint_{U_\varepsilon(x)}|\nabla w_\varepsilon|^2\big)^{s'p}\Big\rangle^{\frac{1}{p}}
\delta_{\sigma}^{2s'-1}dx\bigg)^{\frac{1}{s'}}\\
&\lesssim^{\eqref{pri:11B}}
\bigg(\int_{\Omega}
\Big\langle
\big(\dashint_{U_\varepsilon(x)}
|\varepsilon(\phi_i^\varepsilon,
\sigma_i^\varepsilon)\nabla\varphi_i|^2\big)^{s'\bar{p}}\Big\rangle^{\frac{1}{\bar{p}}}
\delta_{\sigma}^{2s'-1}dx\bigg)^{\frac{1}{s'}} + \Big(\int_{O_\varepsilon}
|\nabla\bar{u}|^{2s'}\delta_{\sigma}^{2s'-1}\Big)^{\frac{1}{s'}},
\end{aligned}
\end{equation*}
while the second term of $F_\varepsilon$ can be controlled
by the same right-hand side above. Hence,
we obtain
\begin{equation}\label{f:6.1}
\begin{aligned}
&\bigg(\int_{\Omega}
 \Big\langle\Big(\dashint_{U_\varepsilon(x)}|
 F_\varepsilon|^{2}\Big)^{s'p}\Big\rangle^{\frac{1}{p}}
 \delta_{\sigma}^{2s'-1}dx\bigg)^{\frac{1}{s'}}\\
&\lesssim^{\eqref{f:6.10a}}
\bigg(\int_{\Omega}
\Big\langle
\big(\dashint_{U_\varepsilon(x)}
|\varepsilon(\phi_i^\varepsilon,
\sigma_i^\varepsilon)\nabla\varphi_i|^2
\big)^{s'\bar{p}}\Big\rangle^{\frac{1}{\bar{p}}}
\delta_{\sigma}^{2s'-1}dx\bigg)^{\frac{1}{s'}}
+ \varepsilon^2 R_0^{2-\frac{1}{s'}}
\Big(\int_{\Omega}|\nabla f|^{2s'}\Big)^{\frac{1}{s'}}.
\end{aligned}
\end{equation}
Applying Minkowski' inequality, H\"older's inequality and Lemma $\ref{lemma:3}$ to the first term in
the second line of $\eqref{f:6.1}$, it can be estimated by
\begin{equation*}
\begin{aligned}
&\bigg(\int_{\Omega}
\Big(\dashint_{U_\varepsilon(x)}
\big\langle|\varepsilon(\phi_i^\varepsilon,
\sigma_i^\varepsilon)|^{2s'\bar{p}}\big\rangle^{\frac{1}{\bar{p}}}
|\nabla\varphi_i|^{2s'}\Big)
\delta_{\sigma}^{2s'-1}dx\bigg)^{\frac{1}{s'}}
\lesssim_d^{\eqref{f:17}}
\bigg(\int_{\Omega}
\big\langle|\varepsilon(\phi_i^\varepsilon,
\sigma_i^\varepsilon)|^{2s'\bar{p}}\big\rangle^{\frac{1}{\bar{p}}}
|\nabla\varphi_i|^{2s'}
\dashint_{B_\varepsilon(x)}\delta_{\sigma}^{2s'-1}dx\bigg)^{\frac{1}{s'}}\\
&\lesssim^{\eqref{pri:3},\eqref{f:5.13}} \varepsilon^2\mu_d^2(R_0/\varepsilon)
\Big(\int_{\Omega}
|\nabla (\eta\nabla\bar{u})|^{2s'}\delta_{\sigma}^{2s'-1}\Big)^{\frac{1}{s'}}
\lesssim \varepsilon^2\mu_d^2(R_0/\varepsilon)
\Big( \varepsilon^{-1}\int_{O_\varepsilon}|\nabla\bar{u}|^{2s'}
+\int_{\Omega\setminus O_\varepsilon}
|\nabla^2\bar{u}|^{2s'}\delta^{2s'-1}\Big)^{\frac{1}{s'}},
\end{aligned}
\end{equation*}
where we also employ the fact
$\delta\sim \delta_{\sigma}$ on $\Omega\setminus O_\varepsilon$ for the last inequality.
Thus, plugging this back into $\eqref{f:6.1}$ and then using the layer and co-layer type estimates $\eqref{f:6.10a}$ and $\eqref{f:6.10b}$, we derive that
\begin{equation*}
\begin{aligned}
\bigg(\int_{\Omega}
 \Big\langle\Big(\dashint_{U_\varepsilon(x)}|
 F_\varepsilon|^{2}\Big)^{s'p}\Big\rangle^{\frac{1}{p}}
 \delta_{\sigma}^{2s'-1}dx\bigg)^{\frac{1}{s'}}
\lesssim R_0^{2-\frac{1}{s'}} \varepsilon^2\mu_d^2(R_0/\varepsilon)\ln^{\frac{1}{s'}}(R_0/\varepsilon)
\Big(\int_{\Omega}|\nabla f|^{2s'}\Big)^{\frac{1}{s'}},
\end{aligned}
\end{equation*}
and this combining with the estimates $\eqref{f:6.2}$, $\eqref{f:6.19}$, and $\eqref{f:6.3}$ gives
the stated estimate $\eqref{f:6.a}$.

We now turn to the estimate $\eqref{f:6.b}$. In view of the notations in $\eqref{scaling1}$ and
$\eqref{scaling2}$ we have
\begin{equation*}\label{f:62}
\begin{aligned}
\bigg\langle
 \bigg(\int_{\Omega/\varepsilon}
 \Big(\dashint_{U_1(y)}\big|
 \nabla\phi^{*}_j \check{h}_j \otimes
 (\nabla\tilde{\bar{u}}-\hat{\varphi})\big|\Big)^2 dy
 \bigg)^{p}\bigg\rangle
= \varepsilon^{pd}\bigg\langle
 \bigg(\int_{\Omega}
 \Big(\dashint_{U_\varepsilon(x)}\big|
 (\nabla\phi^{*\varepsilon}_j h_j)\otimes (1-\eta)\nabla \bar{u}\big|\Big)^2dx
 \bigg)^{p}\bigg\rangle,
\end{aligned}
\end{equation*}
and it follows from Minkowski's inequality, H\"older's inequality, and Lemma $\ref{lemma:3}$ that
\begin{equation*}
\begin{aligned}
&\bigg\langle
 \bigg(\int_{\Omega}
 \Big(\dashint_{U_\varepsilon(x)}\big|
 (\nabla\phi^{*\varepsilon}_j h_j)\otimes (1-\eta)\nabla \bar{u}\big|\Big)^2dx
 \bigg)^{p}\bigg\rangle^{\frac{1}{p}}\\
&\leq
\int_{\Omega}
\dashint_{U_\varepsilon(x)}\big\langle|\nabla\phi_j^{*\varepsilon}|^{2p}
\big\rangle^{\frac{1}{p}}|h|^2|(1-\eta)\nabla\bar{u}|^2
\lesssim^{\eqref{pri:2}}_{\textcolor[rgb]{0.00,0.00,1.00}{\lambda,\lambda_1,\lambda_2,d,p}}
\Big(\int_{\Omega}
|h|^{2s}\delta^{\frac{s}{s'}-2s}\Big)^{\frac{1}{s}}
\Big(\int_{O_\varepsilon}
|\nabla\bar{u}|^{2s'}
\delta^{2s'-1}
\Big)^{\frac{1}{s'}}:=I_0.
\end{aligned}
\end{equation*}
By using the weighted Hardy's inequality again and the layer type estimate $\eqref{f:6.10a}$, the right-hand side above
is given by
\begin{equation}\label{f:63}
\begin{aligned}
I_{0}\lesssim_{\textcolor[rgb]{0.00,0.00,1.00}{d,\lambda,M_0,s}}
 R_0^{2-\frac{1}{s'}} \varepsilon^2\Big(\int_{\Omega}
|\nabla h|^{2s}\delta^{\frac{s}{s'}}\Big)^{\frac{1}{s}}
\Big(\int_{\Omega}
|\nabla f|^{2s'}
\Big)^{\frac{1}{s'}}
\lesssim
 \varepsilon^2\Big(\int_{\Omega}
|\nabla h|^{2s}\Big)^{\frac{1}{s}}
\Big(\int_{\Omega}
|R_0\nabla f|^{2s'}
\Big)^{\frac{1}{s'}},
\end{aligned}
\end{equation}
and this implies the desired estimate $\eqref{f:6.b}$.

\medskip
\noindent
\textbf{Step 4.} Show the estimate $\eqref{f:6.c}$ and $\eqref{f:6.d}$. We start from
establishing $\eqref{f:6.c}$.
By a rescaling argument,
the auxiliary equations $\eqref{pde:22}$ can be rewritten as
\begin{equation}\label{pde:22-1}
\left\{\begin{aligned}
-\nabla \cdot a^{*\varepsilon}\nabla v_\varepsilon^*
&=\varepsilon^{d+1}\nabla\cdot(a^{*\varepsilon}\phi_{j}^{*\varepsilon}
-\sigma_{j}^{*\varepsilon})\nabla h_j
&\quad&\text{in}~~\Omega;\\
v_\varepsilon^*&=0 &\quad&\text{on}~~\partial\Omega,
\end{aligned}\right.
\end{equation}
where $v_\varepsilon^{*}(x):=\varepsilon v^{*}(x/\varepsilon)=\varepsilon v^{*}(y)$ with $y\in\Omega/\varepsilon$, and we also have
\begin{equation*}
\begin{aligned}
\bigg\langle
 \bigg(\int_{\Omega/\varepsilon}
 \Big(\dashint_{U_1(y)}\big|
 (\nabla v^* + \phi^{*}_j\nabla\check{h}_j)
 &\otimes\nabla\tilde{u}_\varepsilon\big|\Big)^2 dy
 \bigg)^{p}\bigg\rangle^{\frac{1}{p}}\\
 & = \bigg\langle
 \bigg(\int_{\Omega}
 \Big(\dashint_{U_\varepsilon(x)}\big|
 (\varepsilon^{-\frac{d}{2}}\nabla v_\varepsilon^{*}+\varepsilon^{1+\frac{d}{2}}\phi_j^\varepsilon\nabla h_j)\otimes \nabla u_\varepsilon\big|\Big)^2dx
 \bigg)^{p}\bigg\rangle^{\frac{1}{p}}.
\end{aligned}
\end{equation*}
The right-hand side above can be further controlled by
\begin{equation}\label{f:6.8}
\begin{aligned}
\varepsilon^{-d}\underbrace{\bigg\langle
 \bigg(\int_{\Omega}
 \dashint_{U_\varepsilon(x)}\big|
 \nabla v_\varepsilon^{*}|^2
 \dashint_{U_\varepsilon(x)}
 |\nabla u_\varepsilon|^2 dx
 \bigg)^{p}\bigg\rangle^{\frac{1}{p}}}_{I_1}
 +\varepsilon^{2+d}
 \underbrace{\bigg\langle
 \bigg(\int_{\Omega}\dashint_{U_\varepsilon(x)}
 |\phi_j^\varepsilon\nabla h_j|^2
 \dashint_{U_\varepsilon(x)}|\nabla u_\varepsilon|^2dx\bigg)^p
 \bigg\rangle^{\frac{1}{p}}}_{I_2}.
\end{aligned}
\end{equation}

We first deal with $I_1$. It follows from Minkowski's inequality and H\"older's inequality that,
\begin{equation}\label{f:6.4}
\begin{aligned}
I_1
&\leq
 \int_{\Omega}
 \Big\langle\Big(
 \dashint_{U_\varepsilon(x)}\big|
 \nabla v_\varepsilon^{*}|^2\Big)^{sp}
 \Big\rangle^{\frac{1}{sp}}
 \Big\langle
 \Big(\dashint_{U_\varepsilon(x)}
 |\nabla u_\varepsilon|^2\Big)^{s'p}
 \Big\rangle^{\frac{1}{s'p}} dx\\
& \leq
\bigg(\int_{\Omega}
 \Big\langle\Big(
 \dashint_{U_\varepsilon(x)}\big|
 \nabla v_\varepsilon^{*}|^2\Big)^{sp}
 \Big\rangle^{\frac{1}{p}}dx\bigg)^{\frac{1}{s}}
 \bigg(\int_{\Omega}
 \Big\langle
 \Big(\dashint_{U_\varepsilon(x)}
 |\nabla u_\varepsilon|^2\Big)^{s'p}
 \Big\rangle^{\frac{1}{p}}dx\bigg)^{\frac{1}{s'}}.
\end{aligned}
\end{equation}
In view of the equations $\eqref{pde:22-1}$, by applying Proposition $\ref{P:7}$ and
Lemma $\ref{lemma:3}$, we have
\begin{equation}\label{f:6.5}
\begin{aligned}
\bigg(\int_{\Omega}
 \Big\langle\Big(
 \dashint_{U_\varepsilon(x)}\big|
 \nabla v_\varepsilon^{*}|^2\Big)^{sp}
 \Big\rangle^{\frac{1}{p}}dx\bigg)^{\frac{1}{s}}
&\lesssim^{\eqref{pri:26}} \varepsilon^{2d+2}
\bigg(\int_{\Omega}
 \Big\langle\Big(
 \dashint_{U_\varepsilon(x)}\big|
 (\phi_j^{*\varepsilon},\sigma_j^{*\varepsilon})\nabla h_j|^2\Big)^{s\bar p}
 \Big\rangle^{\frac{1}{\bar p}}dx\bigg)^{\frac{1}{s}}\\
& \lesssim
\varepsilon^{2d+2}
\bigg(\int_{\Omega}
 \Big(
 \dashint_{U_\varepsilon(x)}\big\langle\big|
 (\phi_j^{*\varepsilon},\sigma_j^{*\varepsilon})\big|^{2s\bar{p}}
 \big\rangle^{\frac{1}{s\bar{p}}} |\nabla h_j|^2
 \Big)^{s}dx\bigg)^{\frac{1}{s}}\\
& \lesssim^{\eqref{pri:3}}
\varepsilon^{2d+2}
\mu_d^2(R_0/\varepsilon)
\Big(\int_{\Omega}
|\nabla h|^{2s}dx\Big)^{\frac{1}{s}},
\end{aligned}
\end{equation}
and
\begin{equation}\label{f:6.6}
\begin{aligned}
\bigg(\int_{\Omega}
 \Big\langle
 \Big(\dashint_{U_\varepsilon(x)}
 |\nabla u_\varepsilon|^2\Big)^{s'p}
 \Big\rangle^{\frac{1}{p}}dx\bigg)^{\frac{1}{s'}}
\lesssim^{\eqref{pri:26}}
\bigg(\int_{\Omega}
 \Big\langle
 \Big(\dashint_{U_\varepsilon(x)}
 |f|^2\Big)^{s'\bar{p}}
 \Big\rangle^{\frac{1}{\bar{p}}}dx\bigg)^{\frac{1}{s'}}
\lesssim
\Big(\int_{\Omega}
 |f|^{2s'}\Big)^{\frac{1}{s'}}.
\end{aligned}
\end{equation}
Combining the estimates $\eqref{f:6.4}$, $\eqref{f:6.5}$, and $\eqref{f:6.6}$, it follows that
\begin{equation}\label{f:6.7}
\varepsilon^{-d} I_1 \lesssim \varepsilon^{d+2}
\mu_d^2(R_0/\varepsilon)
\Big(\int_{\Omega}|\nabla h|^{2s}\Big)^{\frac{1}{s}}
\Big(\int_{\Omega}
 |f|^{2s'}\Big)^{\frac{1}{s'}}.
\end{equation}

We now proceed to handle $I_2$, a similar computation as given for $I_1$ leads to
\begin{equation*}
\begin{aligned}
I_2
&\leq
\int_{\Omega}
 \Big\langle\Big(
 \dashint_{U_\varepsilon(x)}\big|
 \phi_j^{*\varepsilon}\nabla h_j|^2\Big)^{sp}
 \Big\rangle^{\frac{1}{sp}}
 \Big\langle
 \Big(\dashint_{U_\varepsilon(x)}
 |\nabla u_\varepsilon|^2\Big)^{s'p}
 \Big\rangle^{\frac{1}{s'p}} dx\\
& \leq
\bigg(\int_{\Omega}
 \Big\langle\Big(
 \dashint_{U_\varepsilon(x)}\big|
 \phi_j^{*\varepsilon}\nabla h_j|^2\Big)^{sp}
 \Big\rangle^{\frac{1}{p}}dx\bigg)^{\frac{1}{s}}
 \bigg(\int_{\Omega}
 \Big\langle
 \Big(\dashint_{U_\varepsilon(x)}
 |\nabla u_\varepsilon|^2\Big)^{s'p}
 \Big\rangle^{\frac{1}{p}}dx\bigg)^{\frac{1}{s'}} \\
& \lesssim^{\eqref{f:6.6}} \mu_d^2(R_0/\varepsilon)
\Big(\int_{\Omega}|\nabla h|^{2s}\Big)^{\frac{1}{s}}
\Big(\int_{\Omega}
 |f|^{2s'}\Big)^{\frac{1}{s'}}.
\end{aligned}
\end{equation*}
As a result, plugging this and $\eqref{f:6.7}$ back into the estimate $\eqref{f:6.8}$,
we can derive the stated estimate $\eqref{f:6.c}$.

Finally, we turn to the estimate $\eqref{f:6.d}$.
From H\"older's inequality, Minkowski's inequality, Lemma $\ref{lemma:3}$, and a triangle inequality, it follows that
\begin{equation}\label{f:6.9}
\begin{aligned}
I_3 &:= \bigg\langle
 \bigg(\int_{\mathbb{R}^d}
 \Big(\dashint_{B_1(y)}\big|
 \big(\nabla z_j^*+\phi_i^{*}\nabla(\check{h}_i\hat{\varphi}_j)\big)\otimes
 (\nabla\phi_j+e_j)\big|\Big)^2 dy
 \bigg)^{p}\bigg\rangle^{\frac{1}{p}}\\
&\leq
\int_{\mathbb{R}^d}
\Big\langle \big(\dashint_{B_1(y)}\big|
 \big(\nabla z_j^*+\phi_i^{*}\nabla(\check{h}_i\hat{\varphi}_j)\big)\big|^2\big)^{ps}
\Big\rangle^{\frac{1}{ps}}
 \Big\langle\big(\dashint_{B_1(y)}|\nabla\phi_j+e_j|^2\big)^{ps'}
 \Big\rangle^{\frac{1}{ps'}} dy \\
&\lesssim^{\eqref{pri:2}} \int_{\mathbb{R}^d}
\Big\langle \big(\dashint_{B_1(y)}\big|
 \nabla z^*\big|^2\big)^{ps}
\Big\rangle^{\frac{1}{ps}}dy
+
\int_{\mathbb{R}^d}
\Big\langle \big(\dashint_{B_1(y)}\big|\phi_i^{*}\nabla(\check{h}_i\hat{\varphi})\big|^2\big)^{ps}
\Big\rangle^{\frac{1}{ps}}dy.
\end{aligned}
\end{equation}
By the annealed Calder\'on-Zygmund estimate (see \cite[Theorem 6.1]{Duerinckx-Otto20}), for each $j=1,\cdots,d$, we have
\begin{equation*}\label{}
\begin{aligned}
\int_{\mathbb{R}^d}
\Big\langle \big(\dashint_{B_1(y)}\big|
 \nabla z_j^*\big|^2\big)^{ps}
\Big\rangle^{\frac{1}{ps}}dy
\lesssim
\int_{\mathbb{R}^d}
\Big\langle
\Big(\dashint_{B_1(y)}|(\phi_i^{*},\sigma_i^{*})\nabla
(\check{h}_i\hat{\varphi}_j)|^2\Big)^{\bar{p}s}\Big\rangle^{\frac{1}{s\bar{p}}}
dy.
\end{aligned}
\end{equation*}
This together with $\eqref{f:6.9}$ provides us with
\begin{equation*}
\begin{aligned}
I_3
&\lesssim
\int_{\mathbb{R}^d}
\Big\langle
\Big(\dashint_{B_1(y)}|(\phi_i^{*},\sigma_i^{*})\nabla
(\check{h}_i\hat{\varphi})|^2\Big)^{\bar{p}s}\Big\rangle^{\frac{1}{s\bar{p}}}
dy
= \varepsilon^{d+2}
\int_{\mathbb{R}^d}
\Big\langle
\Big(\dashint_{B_\varepsilon(x)}|(\phi_i^{*\varepsilon},
\sigma_i^{*\varepsilon})\nabla
(h_i\varphi)|^2\Big)^{\bar{p}s}\Big\rangle^{\frac{1}{s\bar{p}}}dx\\
& \lesssim^{\eqref{pri:3}}
\varepsilon^{d+2}\mu_d^2(R_0/\varepsilon)
\bigg\{\Big(\int_{\Omega}|
\nabla h|^{2s}\Big)^{\frac{1}{s}}
\Big(\int_{\Omega}|
\varphi|^{2s'}\Big)^{\frac{1}{s'}}
+ \Big(\int_{\Omega}
|h|^{2s}\delta^{\frac{s}{s'}-2s}\Big)^{\frac{1}{s}}
\Big(\int_{\Omega}
|\nabla\varphi|^{2s'}\delta^{2s'-1}\Big)^{\frac{1}{s'}}\bigg\}.
\end{aligned}
\end{equation*}
Applying $W^{1,p}$ estimates, the estimates
$\eqref{f:63}$ and $\eqref{f:6.10b}$ to the right-hand side above,
we further derive that
\begin{equation*}
\begin{aligned}
I_3\lesssim \varepsilon^{d+2}\mu_d^2(R_0/\varepsilon)
\bigg\{\Big(\int_{\Omega}|
\nabla h|^{2s}\Big)^{\frac{1}{s}}
\Big(\int_{\Omega}|
f|^{2s'}\Big)^{\frac{1}{s'}}
+ \ln^{\frac{1}{s'}}(R_0/\varepsilon)\Big(\int_{\Omega}
|\nabla h|^{2s}\Big)^{\frac{1}{s}}
\Big(\int_{\Omega}
|R_0\nabla f|^{2s'}\Big)^{\frac{1}{s'}}\bigg\},
\end{aligned}
\end{equation*}
which proves the desired estimate $\eqref{f:6.d}$,
and this ends the whole proof.
\end{proof}

\medskip
\noindent
\textbf{Proof of Theorem $\ref{thm:2}$.}
Theorem $\ref{thm:2}$ follows from Propositions $\ref{P:8}$ and $\ref{P:9}$.
\qed

\section{Appendix}\label{section:5}

\noindent
 This section mainly contains three
 lemmas and a brief introduction of regular SKT domains.
 The first lemma states some properties of the $A_q$ class
 (see Definition $\ref{def:weight}$);
 The second one is an
 improved Meyer's estimate by using a convexity argument
 (see Lemma $\ref{lemma:15}$);
 The last one is related to the reverse H\"older's inequality for the elliptic operator with constant coefficients under different boundary conditions (see Lemma $\ref{lemma:16}$). In fact,
 we provide a new scheme to establish the reverse
 H\"older's inequality in the case of the regular SKT (or $C^1$) domain,
 and the related issues are of separate concern.

\begin{lemma}[\cite{Duoandikoetxea01}]\label{weight}
Let $\omega\in A_p$, $1\leq p<\infty$. Then, we have the following properties:
\begin{subequations}
\begin{align}
& [\tilde{\omega}]_{A_p} = [\omega]_{A_p},
~\text{where}~\tilde{\omega} = \omega(\varepsilon\cdot);  \label{f:18-0}\\
& \omega(Q)\big(\frac{|S|}{|Q|}\big)^{p}\leq [\omega]_{A_p}\omega(S)
\quad \forall S\subset Q\subset\mathbb{R}^d. \label{f:18}
\end{align}
\end{subequations}
Also, there exists $\epsilon>0$, depending only on $p$ and $[\omega]_{A_p}$, such that
for any $Q$ we have
\begin{subequations}
\begin{align}
&\Big(\dashint_{Q}\omega^{1+\epsilon}\Big)^{\frac{1}{1+\epsilon}}
\lesssim_{d,p,[\omega]_{A_p}}\dashint_{Q}\omega;\label{pri:R-1} \\
& \omega(S)\lesssim_{d,p,[\omega]_{A_p}}  \big(\frac{|S|}{|Q|}\big)^{\delta_0}\omega(Q),
\qquad \forall S\subset Q,
\label{f:19}
\end{align}
\end{subequations}
where $\delta_0 = \epsilon/(1+\epsilon)$. Moreover, given any $1<p<\infty$ and $c_0>1$
there exist positive constants $C=C(d,p,c_0)$ and $\gamma=\gamma(d,p,c_0)$ such that
for all $\omega\in A_p$ we have
\begin{equation}\label{f:19-1}
  [\omega]_{A_p}\leq c_0
\quad \Rightarrow \quad
  [\omega]_{A_{p-\gamma}}\leq C.
\end{equation}
\end{lemma}

\begin{lemma}[improved Meyer's inequality]\label{lemma:15}
Let $\Omega\subset\mathbb{R}^d$ be a bounded Lipschitz domain
with $d\geq 2$, and
$\varepsilon\in(0,1]$ and $R>0$.
Suppose that the (admissible) coefficient satisfies $\eqref{a:1}$, and
$u_\varepsilon$ is the solution of $\eqref{pde:1}$.
Then, for any $0<p-2\ll 1$ and  $p_0>0$, there holds
\begin{equation}\label{pri:10}
\bigg(\dashint_{\textcolor[rgb]{1.00,0.00,0.00}{D_{R}}}
|\nabla u_\varepsilon|^p\bigg)^{\frac{1}{p}}
\lesssim_{\textcolor[rgb]{0.00,0.00,1.00}{\lambda,d,p,p_0,M_0}}
\Big(\dashint_{\textcolor[rgb]{1.00,0.00,0.00}{D_{2R}}}
|\nabla u_\varepsilon|^{p_0} \Big)^{\frac{1}{p_0}}.
\end{equation}
\end{lemma}

\begin{proof}
Based upon the uniform ellipticity conditions $\eqref{a:1}$,
we start from Caccioppoli's inequality, i.e.,
\begin{equation}\label{f:7.7}
\bigg(\dashint_{\textcolor[rgb]{1.00,0.00,0.00}{D_{R}}}
|\nabla u_\varepsilon|^2\bigg)^{\frac{1}{2}}
\lesssim_{\lambda,d,\textcolor[rgb]{0.00,0.00,1.00}{M_0}}\frac{1}{R}
\Big(\dashint_{\textcolor[rgb]{1.00,0.00,0.00}{D_{2R}}}
|u_\varepsilon|^{2} \Big)^{\frac{1}{2}},
\end{equation}
and we omit the details (see e.g. \cite[Chapter 4]{Giaquinta-Martinazzi12}). Then, Meyer's estimate comes from reverse H\"older's inequality
(see e.g. \cite[Theorem 6.38]{Giaquinta-Martinazzi12}), i.e., there exists
$0<\theta\ll 1$ depending only on $d$ and the character of $\Omega$ such that for any $2<p<2+\theta$
we have
\begin{equation}\label{f:7.1}
\bigg(\dashint_{\textcolor[rgb]{1.00,0.00,0.00}{D_{R}}}
|\nabla u_\varepsilon|^p\bigg)^{\frac{1}{p}}
\lesssim_{\lambda,d,p,\textcolor[rgb]{0.00,0.00,1.00}{M_0}}
\Big(\dashint_{\textcolor[rgb]{1.00,0.00,0.00}{D_{2R}}}
|\nabla u_\varepsilon|^{2} \Big)^{\frac{1}{2}}.
\end{equation}
Moreover, with the help of a covering argument we can derive that
for any $0<s<t<1$ there holds %(with $\bar{s}:=t-s$),
\begin{equation*}
\Big(\int_{D_{s}}|\nabla u_\varepsilon|^p\Big)^{\frac{1}{p}}
\lesssim^{\eqref{f:7.1}} \frac{1}{(t-s)^{d(\frac{1}{2}-\frac{1}{p})}}\Big(
\int_{D_{t}}|\nabla u_\varepsilon|^2\Big)^{\frac{1}{2}}.
\end{equation*}
Thus, it follows from a convexity argument (\cite[pp.173]{Fefferman-Stein72}) that for any $p_0>0$
we have
\begin{equation*}
\Big(
\int_{D_{t_0}}|\nabla u_\varepsilon|^2\Big)^{\frac{1}{2}}
\lesssim_{\textcolor[rgb]{0.00,0.00,1.00}{\lambda,d,M_0,p_0}} \Big(
\int_{D_{1}}|\nabla u_\varepsilon|^{p_0}\Big)^{\frac{1}{p_0}},
\end{equation*}
where $t_0\in(0,1)$. This together with $\eqref{f:7.1}$ will lead to
the stated estimate $\eqref{pri:10}$ after a rescaling argument.
\end{proof}

\begin{remark}\label{remark:6}
\emph{Concerned with the Neumann boundary condition,
one can similarly derive the estimate $\eqref{pri:10}$.
To see this, the right-hand side of $\eqref{f:7.7}$ is merely replaced by $\frac{1}{R}\inf_{c\in\mathbb{R}}
\big(\dashint_{B_{2R}\cap\Omega}
|u_\varepsilon-c|^{2}\big)^{\frac{1}{2}}$, and
the rest of the proof is the same.
Also, the argument given for
$\eqref{pri:10}$ leads to
\begin{equation}\label{pri:10-1}
\Big(\dashint_{B_{R}}
|\nabla v_\varepsilon|^p\Big)^{\frac{1}{p}}
\lesssim_{\lambda,d,p,p_0}
\Big(\dashint_{B_{2R}}
|\nabla v_\varepsilon|^{p_0} \Big)^{\frac{1}{p_0}},
\end{equation}
provided $\nabla\cdot a^\varepsilon\nabla v_\varepsilon = 0$
in $B_{2R}$. Let $\alpha\in(0,1)$. By a covering argument (see
e.g. Subsection $\ref{subsection:3.0}$), one can derive
\begin{subequations}
\begin{align}
&\Big(\dashint_{B_{R}}
|\nabla v_\varepsilon|^p\Big)^{\frac{1}{p}}
\lesssim_{\lambda,d,p,p_0,\alpha}^{\eqref{pri:10-1}}
\Big(\dashint_{B_{(1+\alpha)R}}
|\nabla v_\varepsilon|^{p_0} \Big)^{\frac{1}{p_0}};
\label{pri:10-a}\\
&\Big(\dashint_{D_R}|\nabla u_\varepsilon|^p\Big)^{\frac{1}{p}}
\lesssim_{\lambda,d,p,p_0,M_0,\alpha}^{\eqref{pri:10},\eqref{pri:10-1}}
\Big(\dashint_{D_{(1+\alpha)R}}
|\nabla u_\varepsilon|^{p_0}\Big)^{\frac{1}{p_0}}.
\label{pri:10-b}
\end{align}
\end{subequations}}
\end{remark}

%\textcolor[rgb]{1.00,0.00,0.00}{The following is the properties of the $A_q$-weight class used in this paper.}

We now define the regular SKT domain. It is also known as
the chord-arc domains with vanishing constant.
The notion of chord-arc domains
originated in dimension two
(see \cite[Definition 2.6]{Capogna-Kenig-Lanzani05}),
while the higher dimensional counterpart
was initially developed by S. Semmens \cite{Semmes91}
and further investigated by C. Kenig and T. Toro (see e.g. \cite{Kenig-Toro99}),
who revealed how the regularity of the Poisson kernel associated with
$\Omega$ determines the geometry of its boundary, and that's relevant to
the so-called chord-arc domains with vanishing constant.
To the authors' best knowledge, the chord-arc domains with vanishing constant
were first introduced as ``regular SKT'' domains by S. Hofmann,
M. Mitrea and M. Taylor
\cite{Hofmann-Mitrea-Taylor10} to distinguish them from
their two-dimensional counterpart, where ``SKT''
stands for Semmes, Kenig, and Toro, who contributed to its crucial development.
The following definition is mainly taken from \cite{Hofmann-Mitrea-Taylor10},
originally from \cite{Capogna-Kenig-Lanzani05}.

We start from the definition of the Hausdorff distance
(denoted by $\text{D}[\cdot,\cdot]$), namely: For $A,B\subset\mathbb{R}^d$,
\begin{equation}\label{H-dist}
\text{D}[A,B]:=\sup\big\{\text{dist}(a, B) : a\in A\big\}
  + \sup\big\{\text{dist}(A,b) : b\in B\big\}.
\end{equation}

\begin{definition}[separation property]
For a domain $\Omega\subset\mathbb{R}^d$, we say that $\partial\Omega$
separates $\mathbb{R}^d$ if for
every compact $K\subset\mathbb{R}^d$
there exists $R>0$ such that
for any $x\in\partial\Omega\cap K$ and $0<r\leq R$ we can find
an $(d-1)$-dimensional plan
$\mathcal{L}_{x,r}$ containing $x$ and a choice of unit normal vector
(denoted by $n_{x,r}$) to $\mathcal{L}_{x,r}$, such that
\begin{equation*}
\begin{aligned}
&\mathcal{F}^{+}(x,r):=\big\{x+tn_{x,r}\in B_r(x): y\in\mathcal{L}_{x,r},
~ t>\frac{r}{4}\big\}\subset\Omega;\\
&\mathcal{F}^{-}(x,r):=\big\{x+tn_{x,r}\in B_r(x): y\in\mathcal{L}_{x,r},
~ t<-\frac{r}{4}\big\}
\subset\mathbb{R}^d\setminus\Omega.
\end{aligned}
\end{equation*}
\end{definition}

\begin{definition}[Reifenberg flat domain]\label{def:Rfd}
We say that $\Omega\subset\mathbb{R}^d$ is a $(\delta,R)$-Reifenberg
flat domain if $\Omega$
has the separation property, and
for every $x\in\partial\Omega$ and $r\in(0,R]$,
there exists an $(d-1)$-dimensional plane $L(x,r)$ such that
\begin{equation*}
\frac{1}{r}\text{D}\big[\Delta_r(x),L(x,r)\cap B_r(x))\big]
\leq \delta,
\end{equation*}
where $\text{D}[\cdot,\cdot]$ is defined as in $\eqref{H-dist}$.
Moreover, if $\Omega$ is unbounded, we also require that
\begin{equation*}
\sup_{r>0}\sup_{x\in\partial\Omega}\Theta(x,r)<\delta,
\end{equation*}
where $\Theta(x,r):=\inf_{L\ni x}\big\{
\frac{1}{r}\text{D}[\Delta_r(x),L\cap B_r(x)]\big\}$ with
the infimum being taken over all $(d-1)$-dimensional planes containing
$x$.
\end{definition}

\begin{definition}[regular SKT domains]
Let $\Omega\subset\mathbb{R}^d$ be a domain.
We call $\Omega$
the regular SKT domain, if $\Omega$ is a
$(\delta,R)$-Reifenberg flat domain with a locally finite perimeter
for some $\delta>0$ and $R>0$, and moreover
for each compact $K\subset\mathbb{R}^d$ with $K\cap\partial\Omega\not=
\emptyset$, we have
\begin{subequations}
\begin{align}
& \limsup_{r\to0}\Theta_{K}(r) = 0;\\
& \lim_{r\to0}\sup_{x\in\partial\Omega\cap K}
\frac{\sigma(\Delta_r(x))}{c_n r^{d-1}}
  =1,
\end{align}
\end{subequations}
where $\Theta_{K}(r):=\sup_{x\in K}\inf_{L\ni x}\big\{\frac{1}{r}
\text{D}(\Delta_r(Q),L\cap B_r(Q))\big\}$
and $\sigma:=\mathcal{H}^{d-1}\llcorner\partial\Omega$
denotes the surface measure of $\partial\Omega$ and
$\mathcal{H}^{d-1}$ is $(d-1)$-dimensional Hausdorff measure.
\end{definition}

\begin{lemma}[reverse H\"older's inequality]\label{lemma:16}
Let $\Omega\subset\mathbb{R}^d$ be a bounded regular SKT (or $C^1$) domain with $d\geq 2$.
Assume $v$ is a solution of $\nabla\cdot\bar{a}\nabla v = 0$
in $D_{2R}$
with $v=0$ on $\Delta_{2R}$ with $R>0$.
Then, for any $2<p<\infty$ and $\alpha\in(1,2]$, there holds
\begin{equation}\label{pri:12}
 \Big(\dashint_{D_R}|\nabla v|^p\Big)^{\frac{1}{p}}
 \lesssim \Big(\dashint_{D_{\alpha R}}|\nabla v|^2\Big)^{\frac{1}{2}},
\end{equation}
where the multiplicative constant relies on $\lambda,d,p,
\alpha$, and the character of $\Omega$. Moreover, if $\Omega$ is a general Lipschitz domain,
the estimate $\eqref{pri:12}$ holds for
$2<p<\frac{2d}{d-1}+\theta$ with $0<\theta\ll 1$ depending only on
$d$ and the Lipschitz character of $\Omega$. In particular,
if the equation is scalar,
the estimate $\eqref{pri:12}$ holds for
$2<p<3+\theta$ if $d\geq 3$; and $2<p<4+\theta$ if $d=2$.
\end{lemma}

\begin{proof}
By the covering argument mentioned in Remark $\ref{remark:6}$,
it suffices to show the estimate $\eqref{pri:12}$ in the case of $\alpha=2$.
If $\Omega$ is a $C^1$ domain,
the desired estimate $\eqref{pri:12}$ will immediately follow from
interior Lipschitz estimates coupled with boundary H\"older's estimate via
the geometric property of integrals. We focus on
the case of the bounded regular SKT domain. In a nutshell, our method is a bootstrap argument, which contains two ingredients:
(1) the related nontangential maximal function estimates, i.e.,
\begin{equation}\label{f:7.2}
 \|(\nabla v)^*\|_{L^p(\partial D_R)}\lesssim_{\lambda,p,d,M_0}
 \|\nabla_{\text{tan}} v\|_{L^p(\partial D_R)}
 \qquad\forall p\in(1,\infty)
\end{equation}
(see \cite[Theorem 7.2]{Hofmann-Mitrea-Taylor10});
(2) a powerful embedding inequality associated with nontangential maximal functions, i.e.,
\begin{equation}\label{f:7.3}
 \|\nabla v\|_{L^p(D_R)}\lesssim_{p,d}
 \|(\nabla v)^*\|_{L^q(\partial D_R)},
 \qquad p = qd/(d-1)\quad\text{and}\quad q>0
\end{equation}
(see e.g. \cite[Remark 9.3]{Kenig-Lin-Shen13}), where we refer
the reader to \cite[pp.207-209]{Shen18} for the definition of
nontangential maximal functions.

\medskip

Let $t\in[1,2]$, and it follows that
\begin{equation*}
\begin{aligned}
 \|\nabla v\|_{L^p(D_t)}
\lesssim^{\eqref{f:7.3}} \|(\nabla v)^*\|_{L^q(\partial D_t)}
\lesssim^{\eqref{f:7.2}}
\|\nabla_{\text{tan}} v\|_{L^q(\partial D_t\setminus\Delta_t)}.
\end{aligned}
\end{equation*}
By co-area formula, we further derive that
\begin{equation}\label{f:7.4}
\int_{1}^{2}\|\nabla v\|_{L^p(D_t)}^{q} dt
\lesssim \int_{1}^{2}\|\nabla  v\|_{L^q(\partial D_t\setminus\Delta_t)}^qdt
\lesssim \|\nabla v\|_{L^q(D_2)}^q,
\end{equation}
On the other hand, there exists $t_0\in [1,2]$ such that
\begin{equation}\label{f:7.5}
 \|\nabla v\|_{L^p(D_{t_0})}^{q} \leq 2 \int_{1}^{2}\|\nabla v\|_{L^p(D_t)}^{q} dt.
\end{equation}
Thus, combining the estimates $\eqref{f:7.4}$ and $\eqref{f:7.5}$, we obtain
\begin{equation}\label{f:7.6}
\|\nabla v\|_{L^p(D_{1})}
\leq
\|\nabla v\|_{L^p(D_{t_0})}
\lesssim \|\nabla v\|_{L^q(D_2)}.
\end{equation}

Then, by the relationship $p=\frac{qd}{d-1}$ we can infer
the iteration formula $p_k=p_{k-1}\kappa_d=\cdots=p_0\kappa_d^{k}$ with
$p_0:=2$ and $\kappa_d:=\frac{d}{d-1}$. In view of $\eqref{f:7.6}$, we have
$\|\nabla v\|_{L^{p_k}(D_{1})}\lesssim \|\nabla v\|_{L^{2}(D_{2k})}$, which
together with covering and rescaling arguments leads to
\begin{equation*}
\|\nabla v\|_{L^{p_k}(D_{R})}
\leq C_k R^{d(\frac{1}{p_k}-\frac{1}{2})}\|\nabla v\|_{L^{2}(D_{2R})},
\quad \text{and} \quad \lim_{k\to\infty}C_k = \infty.
\end{equation*}
This offers us the stated estimate $\eqref{pri:12}$ by mentioning that
the dependence of the multiplicative constant on $k$ is essentially the dependence on $p$. For a general Lipschitz domain, we refer the reader
to the remark below.
\end{proof}

\begin{remark}
\emph{In terms of Neumann boundary conditions, the estimate
$\eqref{f:7.2}$ should be replaced by
$\|(\nabla v)^*\|_{L^p(\partial D_R)}\lesssim_p
 \|(\partial v/\partial\nu)\|_{L^p(\partial D_R)}$ (see \cite[Theorem 7.3]{Hofmann-Mitrea-Taylor10}), and then
following the same arguments we can analogically establish $\eqref{pri:12}$ for Neumann boundary problems, where $\partial v/\partial\nu$ is
the corresponding conormal derivative of $v$.}
\end{remark}

\begin{remark}\label{remark:5}
\emph{If $\Omega$ is a general Lipschitz domain, the estimate $\eqref{f:7.2}$
is not true, and the upper bound of $1<p<2+\theta$
with $0<\theta\ll1$ is optimal (see
\cite{Shen06} and references therein for more details). In this regard, the proceed of ``bootstrap'' given in the proof of
Lemma $\ref{lemma:16}$ can not happen. In the scalar case, the result
was shown in \cite[Theorem B]{Shen05} together with \cite[Theorem 1.2, Proposition 1.4]{Jerison-Kenig95}.}
\end{remark}

\begin{center}
\textbf{Acknowledgements}
\end{center}

\noindent
The authors are grateful to the two anonymous referees for
valuable comments and suggestions to greatly improve
the statements of the article.
The first author is supported by China Postdoctoral
Science Foundation (Grant No. 2022M710228).
The second author deeply appreciated
Prof. Felix Otto for his instruction and encouragement when he held a post-doctoral position at the Max Planck
Institute for Mathematics in the Sciences (in Leipzig).
The second author would like to thank the Department of Mathematics
at Peking University for providing him with excellent office
conditions during his visit in the spring of 2022. This work
was supported in part by the Young Scientists Fund of the National Natural
Science Foundation of China (Grant No. 11901262);
the Fundamental Research Funds for the Central Universities
(Grant No.lzujbky-2021-51).

\end{document}